%% file: front.tex
\documentclass[11pt]{amsart}

\input{header} 

\begin{document}
\thispagestyle{empty}

\begin{abstract} 
We develop the Dirac formalism for infinite-dimensional quadratic $\ZZ$-graded color Lie algebras with finite-dimensional components. Cubic Dirac operators are defined in completions of the quantum Weil algebra determined by the $\ZZ$-grading. The same grading fixes the normal-ordering convention. Normal ordering introduces a cohomological obstruction to the construction, measured by a color analogue of the Kac–Peterson class. When this class is trivial, we construct cubic and relative cubic Dirac operators satisfying the expected invariance properties and Parthasarathy-type square formulas. We further extend the Chern–Weil homomorphism to completed $\mathfrak{g}$-differential algebras and use it to identify the classical precursor of the cubic Dirac operator with the Chern–Simons element associated with the invariant quadratic polynomial determined by the quadratic structure.

As applications, we consider symmetrizable Kac--Moody superalgebras. In this setting, the Kac--Peterson class is trivial, with primitive given by the Weyl vector, which yields the linear correction defining the cubic Dirac operator. We then use the relative Dirac operator to extract representation-theoretic information from highest weight supermodules. As an explicit example, for the affine Kac--Moody superalgebra associated with $\mathfrak{osp}(1\vert 2n)$, we compute $\ker\operatorname{D}_{\mathfrak{g},\mathfrak{g}_{\bar{0}}}$ on integrable highest weight supermodules. Finally, for unitarizable highest weight supermodules, we explain why the usual Dirac inequality is not available in the affine setting.
\end{abstract}

\maketitle

\setlength{\parindent}{1em}
\setcounter{tocdepth}{2}

\tableofcontents

\input{main} 

\bibliography{literatur}
\bibliographystyle{alpha}
\end{document}

%% file: header.tex
\usepackage{marginnote}
\usepackage{anysize} \marginsize{2.5cm}{2.5cm}{2.5cm}{2.5cm}
\usepackage{comment}
\usepackage{listings}
\usepackage{xcolor}
\usepackage{amsmath}
\usepackage{mathtools}
\usepackage{booktabs}
\usepackage[all]{xy}
\usepackage[utf8]{inputenc}
\usepackage{varioref}
\usepackage{upgreek}
\usepackage{amsfonts}
\usepackage{amssymb}
\usepackage{bbm}
\usepackage{esint}
\usepackage{graphicx}
\usepackage{subcaption}
\usepackage{tikz}
\usepackage{empheq}
\usepackage{enumitem}
\usepackage{tikz-cd}
\usepackage{upgreek}
\usepackage{todonotes}
\usetikzlibrary{matrix,arrows,decorations.pathmorphing}
\usepackage{mathrsfs}
\usepackage[hypertexnames=false,backref=page,pdftex,
 	pdfpagemode=UseNone,
 	breaklinks=true,
 	extension=pdf,
 	colorlinks=true,
 	linkcolor=blue,
 	citecolor=red,
 	urlcolor=blue,
 ]{hyperref}

\def\sp{\operatorname{sp}}

\newcommand{\tr}{\operatorname{tr}}

\newcommand{\Hom}{\operatorname{Hom}}

\newcommand{\diag}{\operatorname{diag}}

\newcommand{\sgn}{\operatorname{sgn}}

\newcommand{\str}{\operatorname{str}}

\newcommand{\id}{\operatorname{id}}

\renewcommand{\Im}{\operatorname{im}}
\newcommand{\Ker}{\operatorname{ker}}

\newcommand{\End}{\operatorname{End}}
\newcommand{\ad}{\operatorname{ad}}

\newcommand{\Dirac}{\operatorname{D}}

\def\WW{{\mathcal{W}}}
\def\sp{{\mathfrak{sp}}}

\def\gr{{\mathrm{gr}}}

\def\calQ{{\mathcal{Q}}}

\def\br{{\operatorname{br}}}

\def\CC{{\mathbb C}}
\def\RR{{\mathbb R}}
\def\ss{{\mathfrak s}}

\def\psl{{\mathfrak{psl}}}
\def\rk{{\operatorname{rk}}}
\def\ggbar{{\overline{\gg}}}

\def\Mat{{\text{Mat}}}
\def\deg{{\operatorname{deg}}}
\def\sl{{\mathfrak{sl}}}
\def\st{{\text{st}}}
\def\Cl{{\operatorname{Cl}}}

\def\jj{{\mathfrak{j}}}
\def\Weil{{\mathcal{W}_{\varepsilon}(\gg)}}
\def\etr{{\operatorname{tr}_{\varepsilon}}}
\def\osp{{\mathfrak{osp}}}

\def\ng{{0_{\scriptscriptstyle \Gamma}}}
\def\cWeil{{\widehat{\WW}_{\varepsilon}(\gg)}}
\def\widew{{\widehat{\WW}}}
\def\dw{{\diff^{W}}}
\def\hh{{\mathfrak h}}
\def\diff{{\operatorname{d}}}
\def\calB{{\mathcal{B}}}
\def\CS{{\mathrm{CS}}}
\def\univ{{\mathrm{univ}}}
\def\Der{{\mathrm{Der}}}
\def\Auniv{{A_{\text{univ}}}}
\def\calD{{\mathcal{D}}}
\def\calA{{\mathcal{A}}}
\def\pp{{\mathfrak p}}

\def\KP{{\psi_{\operatorname{KP}}}}
\def\cat{{\textbf{Vec}_{\Gamma}^{\varepsilon}}}
\def\KPsharp{{\psi^{\sharp}_{\operatorname{KP}}}}

\def\ll{{\mathfrak l}}

\def\bb{{\mathfrak b}}
\def\even{{\mathfrak{g}_{\bar{0}}}}
\def\odd{{\mathfrak{g}_{\bar{1}}}}
\def\pp{{\mathfrak p}}

\def\NN{{\mathbb N}}
\def\nn {{\mathfrak{n}}}
\def\pp{{\mathfrak{p}}}
\def\gg{{\mathfrak{g}}}

\def\ZZ{{\mathbb Z}}

\def\slmn{{\mathfrak{sl}(m\vert n)}}

\def\UE{{\mathfrak U}}

\newcommand{\gl}{{\mathfrak{gl}}}

\newcommand{\abs}[1]{\left|{#1}\right|}

\makeatletter
\newcommand*{\da@rightarrow}{\mathchar"0\hexnumber@\symAMSa 4B }
\newcommand*{\da@leftarrow}{\mathchar"0\hexnumber@\symAMSa 4C }
\newcommand*{\xdashrightarrow}[2][]{%
  \mathrel{%
    \mathpalette{\da@xarrow{#1}{#2}{}\da@rightarrow{\,}{}}{}%
  }%
}
\newcommand{\xdashleftarrow}[2][]{%
  \mathrel{%
    \mathpalette{\da@xarrow{#1}{#2}\da@leftarrow{}{}{\,}}{}%
  }%
}
\newcommand*{\da@xarrow}[7]{%
  \sbox0{$\ifx#7\scriptstyle\scriptscriptstyle\else\scriptstyle\fi#5#1#6\m@th$}%
  \sbox2{$\ifx#7\scriptstyle\scriptscriptstyle\else\scriptstyle\fi#5#2#6\m@th$}%
  \sbox4{$#7\dabar@\m@th$}%
  \dimen@=\wd0 %
  \ifdim\wd2 >\dimen@
    \dimen@=\wd2 %
  \fi
  \count@=2 %
  \def\da@bars{\dabar@\dabar@}%
  \@whiledim\count@\wd4<\dimen@\do{%
    \advance\count@\@ne
    \expandafter\def\expandafter\da@bars\expandafter{%
      \da@bars
      \dabar@ 
    }%
  }%
  \mathrel{#3}%
  \mathrel{%
    \mathop{\da@bars}\limits
    \ifx\\#1\\%
    \else
      _{\copy0}%
    \fi
    \ifx\\#2\\%
    \else
      ^{\copy2}%
    \fi
  }%
  \mathrel{#4}%
}
\makeatother

\makeatletter
\newsavebox\myboxA
\newsavebox\myboxB
\newlength\mylenA

\newcommand*\xtilde[2][0.8]{%
    \sbox{\myboxA}{$\m@th#2$}%
    \setbox\myboxB\null
    \ht\myboxB=\ht\myboxA%
    \dp\myboxB=\dp\myboxA%
    \wd\myboxB=#1\wd\myboxA
    \sbox\myboxB{$\m@th\widetilde{\copy\myboxB}$}
    \setlength\mylenA{\the\wd\myboxA}
    \addtolength\mylenA{-\the\wd\myboxB}%
    \ifdim\wd\myboxB<\wd\myboxA%
       \rlap{\hskip 0.5\mylenA\usebox\myboxB}{\usebox\myboxA}%
    \else
        \hskip -0.5\mylenA\rlap{\usebox\myboxA}{\hskip 0.5\mylenA\usebox\myboxB}%
    \fi}

\newbox\usefulbox

\def\getslant #1{\strip@pt\fontdimen1 #1}

\def\xxtilde #1{\mathchoice
 {{\setbox\usefulbox=\hbox{$\m@th\displaystyle #1$}%
    \dimen@ \getslant\the\textfont\symletters \ht\usefulbox
    \divide\dimen@ \tw@ 
    \kern\dimen@ 
    \xtilde{\kern-\dimen@ \box\usefulbox\kern\dimen@ }\kern-\dimen@ }}
 {{\setbox\usefulbox=\hbox{$\m@th\textstyle #1$}%
    \dimen@ \getslant\the\textfont\symletters \ht\usefulbox
    \divide\dimen@ \tw@ 
    \kern\dimen@ 
    \xtilde{\kern-\dimen@ \box\usefulbox\kern\dimen@ }\kern-\dimen@ }}
 {{\setbox\usefulbox=\hbox{$\m@th\scriptstyle #1$}%
    \dimen@ \getslant\the\scriptfont\symletters \ht\usefulbox
    \divide\dimen@ \tw@ 
    \kern\dimen@ 
    \xtilde{\kern-\dimen@ \box\usefulbox\kern\dimen@ }\kern-\dimen@ }}
 {{\setbox\usefulbox=\hbox{$\m@th\scriptscriptstyle #1$}%
    \dimen@ \getslant\the\scriptscriptfont\symletters \ht\usefulbox
    \divide\dimen@ \tw@ 
    \kern\dimen@ 
    \xtilde{\kern-\dimen@ \box\usefulbox\kern\dimen@ }\kern-\dimen@ }}%
 {}}

\newcommand*\xoverline[2][0.75]{%
    \sbox{\myboxA}{$\m@th#2$}%
    \setbox\myboxB\null
    \ht\myboxB=\ht\myboxA%
    \dp\myboxB=\dp\myboxA%
    \wd\myboxB=#1\wd\myboxA
    \sbox\myboxB{$\m@th\overline{\copy\myboxB}$}
    \setlength\mylenA{\the\wd\myboxA}
    \addtolength\mylenA{-\the\wd\myboxB}%
    \ifdim\wd\myboxB<\wd\myboxA%
       \rlap{\hskip 0.5\mylenA\usebox\myboxB}{\usebox\myboxA}%
    \else
        \hskip -0.5\mylenA\rlap{\usebox\myboxA}{\hskip 0.5\mylenA\usebox\myboxB}%
    \fi}

\def\xxoverline #1{\mathchoice
 {{\setbox\usefulbox=\hbox{$\m@th\displaystyle #1$}%
    \dimen@ \getslant\the\textfont\symletters \ht\usefulbox
    \divide\dimen@ \tw@ 
    \kern\dimen@ 
    \overline{\kern-\dimen@ \box\usefulbox\kern\dimen@ }\kern-\dimen@ }}
 {{\setbox\usefulbox=\hbox{$\m@th\textstyle #1$}%
    \dimen@ \getslant\the\textfont\symletters \ht\usefulbox
    \divide\dimen@ \tw@ 
    \kern\dimen@ 
    \xoverline{\kern-\dimen@ \box\usefulbox\kern\dimen@ }\kern-\dimen@ }}
 {{\setbox\usefulbox=\hbox{$\m@th\scriptstyle #1$}%
    \dimen@ \getslant\the\scriptfont\symletters \ht\usefulbox
    \divide\dimen@ \tw@ 
    \kern\dimen@ 
    \xoverline{\kern-\dimen@ \box\usefulbox\kern\dimen@ }\kern-\dimen@ }}
 {{\setbox\usefulbox=\hbox{$\m@th\scriptscriptstyle #1$}%
    \dimen@ \getslant\the\scriptscriptfont\symletters \ht\usefulbox
    \divide\dimen@ \tw@ 
    \kern\dimen@ 
    \xoverline{\kern-\dimen@ \box\usefulbox\kern\dimen@ }\kern-\dimen@ }}%
 {}}
\makeatother

\makeatletter
\newcommand{\mylabel}[2]{#2\def\@currentlabel{#2}\label{#1}}
\makeatother

\makeatletter
\newcommand{\Mac}{}
\DeclareRobustCommand{\Mac}{%
  M%
  \raisebox{\dimexpr\fontcharht\font`M-\height}{%
    \check@mathfonts\fontsize{\sf@size}{0}\selectfont
    c%
  }%
}
\makeatother

\newtheoremstyle{citing}
  {}
  {}
  {\itshape}
  {}
  {\bfseries}
  {\textbf{.}}
  {.5em}
  {\thmnote{#3}}

\theoremstyle{plain}
\newtheorem{theorem}{Theorem}

\newtheorem{lemma}[theorem]{Lemma}
\newtheorem{corollary}[theorem]{Corollary}

\newtheorem{proposition}[theorem]{Proposition}

\theoremstyle{remark}
\newtheorem{example}[theorem]{Example}

\theoremstyle{definition}

\newtheorem{definition}[theorem]{Definition}

\numberwithin{equation}{section}

\theoremstyle{remark}
\newtheorem{remark}[theorem]{Remark}

{\theoremstyle{citing}
}

{\theoremstyle{definition}
}

\title{Dirac operators for infinite-dimensional color Lie algebras}

\author{Steffen Schmidt}
\address{Center for Quantum Mathematics, 
University of Southern Denmark, DK-5230 Odense, Denmark}
\email{stschmidt@imada.sdu.dk}
\author{Konstantin Wernli}
\address{Center for Quantum Mathematics, 
University of Southern Denmark, DK-5230 Odense, Denmark}
\email{kwernli@imada.sdu.dk}

\let\origmaketitle\maketitle
\def\maketitle{
  \begingroup
  \def\uppercasenonmath##1{} 
  \let\MakeUppercase\relax 
  \origmaketitle
  \endgroup
}


%% file: main.tex
\section{Introduction}
Let $(\gg,\ll)$ be a finite-dimensional quadratic pair of complex Lie superalgebras. Thus $\gg$ is equipped with a non-degenerate even supersymmetric invariant bilinear form $B$, and $\ll\subseteq\gg$ is a Lie subsuperalgebra for which $B_{\ll} \coloneqq B\vert_{\ll}$ is non-degenerate. We write
\begin{equation}
\gg=\ll\oplus\pp,\qquad \pp=\ll^{\perp}.
\end{equation}
In this setting Chen and Kang \cite{Dirac_quadratic} constructed a \emph{relative cubic Dirac operator} in
$
(\UE(\gg)\otimes\Cl(\pp))^{\ll},
$
which is the $\ll$-basic subalgebra $W(\gg,\ll)$
of the quantum Weil algebra
$
\WW(\gg)\coloneqq\UE(\gg)\otimes\Cl(\gg)
$
of $\gg$ \cite{Schmidt_perturbations}. In terms of a homogeneous basis $\{e_a\}$ of $\pp$ with $B|_{\pp}$-dual basis $\{e^a\}$, this element is given by
\begin{equation}
\Dirac_{\gg,\ll}
=
\sum_a e^a\otimes e_a
-\tfrac{1}{12}
\sum_{a,b,c}
(-1)^{p(e_{a})p(e_{b})+p(e_{c})}B(e_{a},[e_{b},e_{c}])
\bigl(1\otimes e^a\wedge e^b\wedge e^c\bigr).
\end{equation}
It is $\ll$-invariant and satisfies the Parthasarathy-type formula
\begin{equation}
\Dirac_{\gg,\ll}^{2}
=
\Omega_{\gg}\otimes 1
+
\tfrac{1}{24}\str(\ad_{\gg}(\Omega_{\gg}))
-
j\!\left(
\Omega_{\ll}\otimes 1
+
\tfrac{1}{24}\str(\ad_{\ll}(\Omega_{\ll}))
\right),
\end{equation}
where $\Omega_{\gg}$ and $\Omega_{\ll}$ denote the quadratic Casimir elements and $j$ is the canonical embedding of $\WW(\ll)$ into $\WW(\gg)$. The associated supercommutator
$
d=[\Dirac_{\gg,\ll},\cdot]_{\WW}
$
defines a differential on $\WW(\gg,\ll)$. Its cohomology identifies with the center $\mathfrak Z(\ll)$ of $\UE(\ll)$. This construction extends Kostant's cubic Dirac operator \cite{Kostant_cubic_Dirac} to quadratic Lie superalgebras and provides the algebraic framework underlying the super-analogue of Vogan's conjecture. Representation-theoretic applications have been studied in \cite{Schmidt_Noja,Schmidt_perturbations,Schmidt,SchmidtDirac}. These works develop Dirac cohomology for Lie superalgebras, compute it in several classes of modules, and derive structural consequences for highest weight theory. In particular, the square formula gives rise to Dirac inequalities, which play a role in the classification of unitarizable highest weight supermodules, including the case of $\sl(m\vert n)$ \cite{Schmidt}.

The super setting admits a direct extension to color Lie algebras, introduced by Rittenberg and Wyler \cite{RittenbergWyler1978} and systematically studied by Scheunert \cite{Scheunert1979}. These are common generalizations of Lie algebras and Lie superalgebras, obtained by replacing the $\ZZ_2$-grading by a grading over an arbitrary abelian group $\Gamma$ and the Koszul sign rule by a commutation factor $\varepsilon:\Gamma\times\Gamma\to\CC^\times$. The first genuinely color examples occur for $\Gamma=\ZZ_2\times\ZZ_2$, for instance the $\ZZ_2\times\ZZ_2$-graded general linear algebras $\mathfrak{gl}(m_1,m_2\vert n_1,n_2)$ and their orthosymplectic analogues (see~\cite{RittenbergWyler1978}). Recently, color Lie algebras have reappeared in mathematical physics as a natural language for symmetry structures beyond the ordinary boson--fermion paradigm. This renewed interest is particularly visible in $\ZZ_2\times\ZZ_2$-graded models, including the graded symmetries of the L\'evy--Leblond equation \cite{AizawaKuznetsovaTanakaToppan2016}, $\ZZ_2\times\ZZ_2$-graded mechanics and its quantization \cite{AizawaKuznetsovaToppan2021}, double-graded supersymmetric quantum mechanics \cite{BruceDuplij2020}, and $\ZZ_2\times\ZZ_2$-graded parastatistics in multiparticle quantum Hamiltonians \cite{Toppan2021}. Moreover, the first author showed in \cite{Schmidt_perturbations} that the general formulation of cubic Dirac operators for Lie superalgebras naturally leads to color Lie algebras, since the quantum Weil algebra carries an intrinsic color Lie algebra structure.

The finite-dimensional color analogue of the relative cubic Dirac operator was constructed by Meyer \cite{Meyer}. The properties of $\Dirac_{\gg,\ll}$, the Parthasarathy-type square formula, and the algebraic form of Vogan's conjecture extend to this setting after replacing the super sign rules by the corresponding $\varepsilon$-sign rules.

This article develops the Dirac formalism for infinite-dimensional $\ZZ$-graded color Lie algebras $\gg=\bigoplus_{i\in\ZZ}\gg_i$ with finite-dimensional homogeneous components. It provides the first such formalism for infinite-dimensional color Lie algebras, simultaneously generalizing Meyer’s finite-dimensional color construction \cite{Meyer} and Meinrenken’s infinite-dimensional Lie algebra construction \cite{Meinrenken}. In particular, the framework applies to symmetrizable Kac--Moody superalgebras and yields new Dirac-theoretic tools for studying their representations. The main difficulty is to control the divergences caused by infinite-dimensionality. We replace the resulting formal infinite sums by $\varepsilon$-normally ordered elements in a suitable completion of $\mathcal{W}(\gg)$. The failure of the finite-dimensional Dirac identities is then governed by a generalized Kac--Peterson cocycle, whose cohomology class is the precise obstruction to the expected square formula. When this obstruction vanishes, we construct a corrected cubic Dirac operator that restores the finite-dimensional square identity. Its existence and square formula then yield representation-theoretic consequences.

Finally, by extending the Chern–Weil formalism to completed $\gg$-differential algebras and identifying the classical Dirac element with the Chern–Simons element associated with the quadratic invariant polynomial, we place the construction at the intersection of representation theory, characteristic classes, and gauge theory and provide a natural point of departure for applications to current algebras, loop-group representation theory, and geometric and field-theoretic quantization.

\subsection{Main Results and Organization}
We describe the main results and the organization of the paper. Throughout, $\gg$ is a $\mathbb Z$-graded quadratic color Lie algebra with finite-dimensional homogeneous components, as introduced in Section~\ref{subsec::conventions}. We denote by $\gg^{\ast}$ its restricted dual. The commutation factor is denoted by $\varepsilon$, and the non-degenerate invariant $\varepsilon$-symmetric bilinear form by $B$. The $\mathbb Z$-grading determines the decomposition
\begin{equation}\label{eq::decomposition_intro}
\gg=\gg_-\oplus\gg_+,\qquad
\gg_+=\bigoplus_{i>0}\gg_i,\qquad
\gg_-=\bigoplus_{i\leq0}\gg_i.
\end{equation}

Section~\ref{sec::preliminaries} develops the algebraic framework required to extend the finite-dimensional Dirac formalism to this setting. The starting point is the usual finite-dimensional formulas for the quadratic Casimir element and the cubic Dirac operator, which become infinite sums for a $\mathbb Z$-graded infinite-dimensional color Lie algebra. The decomposition \eqref{eq::decomposition_intro} fixes the normal-ordering convention and determines the completions in which these infinite sums are defined. We construct the completed exterior and Clifford algebras $\widehat{\bigwedge}_{\varepsilon}(\gg)$ and $\widehat{\operatorname{Cl}}_{\varepsilon}(\gg)$ in Section~\ref{subsec::completion_exterior_clifford_superalgebra}, the completed symmetric and enveloping algebras $\widehat{S}_{\varepsilon}(\gg)$ and $\widehat{\UE}_{\varepsilon}(\gg)$ in Section~\ref{subsec::completion_symmetric_universal_superalgebra}, and the completed Weil and quantum Weil algebras $\widehat{W}_{\varepsilon}(\gg)$ and $\widehat{\WW}_{\varepsilon}(\gg)$ in Section~\ref{subsec::completion_Weil_quantum_Weil_superalgebra}. The usual exterior-to-Clifford quantization is replaced by a normally ordered map
$
\widehat q:\widehat{\bigwedge}_{\varepsilon}(\gg)\longrightarrow\widehat{\operatorname{Cl}}_{\varepsilon}(\gg),
$
constructed in Definition~\ref{def::normal_ordered_quantization}. Similarly, the PBW symmetrization is replaced by a normally ordered map
$
\widehat Q:\widehat S_{\varepsilon}(\gg)\longrightarrow\widehat U_{\varepsilon}(\gg),
$
constructed in Definition~\ref{def::normal_ordered_color_PBW_symmetrization}. Combining these maps yields an isomorphism of graded vector spaces
$
\widehat{\calQ}:\widehat{W}_{\varepsilon}(\gg)\xrightarrow{\sim}\widehat{\WW}_{\varepsilon}(\gg).
$
Thus Section~\ref{sec::preliminaries} provides the completed classical and quantum algebras, together with the normally ordered quantization maps, in which the infinite-dimensional Dirac formalism is developed.

Section~\ref{sec::preliminaries} also introduces the generalized Kac--Peterson class, which identifies the cohomological obstruction to the finite-dimensional identities after normal ordering. The decomposition \eqref{eq::decomposition_intro} defines a color Lie algebra $2$-cocycle
$
\KP\in \widehat{\bigwedge}_{\varepsilon}^{2}(\gg^{\ast})
$
by
\begin{equation}
\KP(x,y)
\coloneqq
\tfrac12\etr(\ad_x\pi_{-}\ad_y\pi_{+})
-\tfrac12\varepsilon(x,y)\etr(\ad_y\pi_{-}\ad_x\pi_{+})
\end{equation}
for homogeneous $x,y\in\gg$, where $\pi_{\pm}$ denote the projections onto $\gg_{\pm}$. Its cohomology class is the Kac--Peterson class. It measures the anomaly introduced by normal ordering: the normally ordered Casimir and cubic Dirac operator have the expected finite-dimensional properties precisely when this class is trivial and a corresponding linear correction is chosen.

Section~\ref{sec::Cubic_Dirac} develops the cubic Dirac formalism in the completed setting. We first construct the corrected Casimir element. The normal-ordered Casimir is the quantization $\Omega'_{\gg}\coloneqq \widehat Q(p)$ of the invariant quadratic element
\begin{equation}
p=\sum_a\varepsilon(e_a)e_ae^a=\sum_a e^ae_a\in\widehat S_{\varepsilon}^{2}(\gg).
\end{equation}
Theorem~\ref{thm::Lie_derivative_normal_ordered_Casimir} shows that its failure to be central is governed precisely by the Kac--Peterson cocycle. 
Consequently, if $\KP$ is a coboundary, so that $\KP=d\rho$ for some $\rho\in(\gg^\ast)_{(0,0_\Gamma)}$, the corrected Casimir element
$
\Omega_{\gg}\coloneqq\Omega'_{\gg}+2\rho^\sharp
$
is central.

We next construct the cubic Dirac operator and give a Chern--Simons interpretation of its classical precursor. The classical cubic Dirac element is
\begin{equation}
\calD\coloneqq \sum_a e^a\otimes e_a+1\otimes\phi \in \widehat{W}_{\varepsilon}(\gg),
\end{equation}
where $\phi\in\widehat{\bigwedge}_{\varepsilon}^{3}(\gg)$ is the invariant cubic tensor encoding the Lie bracket through $B$. To explain the origin of this element, Section~\ref{subsec::CS_element} extends the Chern--Weil formalism for $\gg$-differential algebras to color Lie algebras and to the completed setting. Theorem~\ref{thm::Chern_Weil_map} establishes the corresponding universal property of $\widehat{W}_{\varepsilon}(\gg)$. For the invariant quadratic polynomial defined by $B$, Proposition~\ref{prop::D_is_CS_element} identifies the resulting Chern--Simons element with $\calD$.

Normal-ordered quantization gives the uncorrected cubic Dirac operator $\Dirac'\coloneqq\widehat{\calQ}(\calD)$. As for the Casimir element, its failure to satisfy the finite-dimensional identities is governed by the Kac--Peterson cocycle. Theorem~\ref{thm::square_uncorrected} determines this obstruction explicitly. If $\KP=d\rho$, the corrected cubic Dirac operator
\begin{equation}
\Dirac\coloneqq\Dirac'+1\otimes\rho^\sharp
\end{equation}
is $\gg$-invariant and satisfies the Parthasarathy-type square formula
\begin{equation}
\Dirac^{2}
=
\Omega_{\gg}\otimes 1
+
\tfrac{1}{24}\etr(\ad_{\gg_{0}}(\Omega_{\gg_{0}}))(1\otimes 1)
+
B(\rho^{\sharp},\rho^{\sharp})(1\otimes 1),
\end{equation}
as shown in Corollary~\ref{cor::D_square}.

Section~\ref{subsec::relative_cubic_Dirac_operator} extends the construction to relative cubic Dirac operators. Let $(\gg,\ll)$ be a quadratic pair satisfying the assumptions of that section. The completed quantum Weil algebra of $\ll$ admits a diagonal embedding
$
j:\widehat{\mathcal W}_{\varepsilon}(\ll)\longrightarrow
\widehat{\mathcal W}_{\varepsilon}(\gg),
$
which allows us to define the relative cubic Dirac operator by
\begin{equation}
\Dirac_{\gg,\ll}\coloneqq \Dirac_{\gg}-j(\Dirac_{\ll}).
\end{equation}
Theorem~\ref{thm::square_Dirac} shows that $\Dirac_{\gg,\ll}$ belongs to the completion of the $\ll$-basic subalgebra of $\WW_{\varepsilon}(\gg)$. Moreover, it satisfies the relative Parthasarathy-type square formula
\begin{equation}
\Dirac_{\gg,\ll}^{2}
=
\Omega_{\gg}-j(\Omega_{\ll})
+\tfrac1{24}\operatorname{tr}_{\varepsilon}
\bigl(\operatorname{ad}_{\gg_0}(\Omega_{\gg_0})\bigr)
-\tfrac1{24}\operatorname{tr}_{\varepsilon}
\bigl(\operatorname{ad}_{\ll_0}(\Omega_{\ll_0})\bigr)
+B(\rho^\sharp,\rho^\sharp)
-B(\rho_{\ll}^\sharp,\rho_{\ll}^\sharp).
\end{equation}
If $[\ss,\ss]\subset\ll$, then the cubic term vanishes and $\Dirac_{\gg,\ll}$ is quadratic. 

This establishes the cubic and relative cubic Dirac formalism for infinite-dimensional $\mathbb Z$-graded quadratic color Lie algebras, extending the characteristic constructions and Parthasarathy-type square formulas of the finite-dimensional theory to a setting in which normal ordering and the resulting Kac--Peterson obstruction have to be taken into account.

Section~\ref{sec::KM} applies the preceding Dirac formalism to Kac–Moody superalgebras. We first show in Section~\ref{subsec::Generalities_KM_Superalgebras} that the cohomological obstruction identified above vanishes for every symmetrizable Kac–Moody superalgebra. More precisely, Proposition~\ref{prop::trivial_KP_class_KM} gives $\psi_{\mathrm{KP}}=d\rho$, where $\rho$ is the Weyl vector. Consequently, every symmetrizable Kac–Moody superalgebra admits a corrected cubic Dirac operator $\Dirac\in\widehat{\mathcal W}(\gg)$.

Sections~\ref{subsec::Oscillator_supermodule} and \ref{subsec::action_of_the_cubic_Dirac_operator} turn the relative Dirac operator for the quadratic pair $(\gg,\even)$ into a representation-theoretic tool. The essential additional object is the oscillator module $M(\odd)$ constructed in Section~\ref{subsec::Oscillator_supermodule}, which provides the natural representation of the completed Clifford algebra associated with the odd part $\odd$. For a highest weight $\gg$-supermodule $M$, this allows the relative Dirac operator to act on the completed endomorphism space
\begin{equation}
\Dirac_{\gg,\even}\in\widehat{\End}(M\otimes M(\odd)).
\end{equation}
As in the finite-dimensional Dirac theory, representation-theoretic information can be extracted from the square of the Dirac operator through its relation to the Casimir operators. We therefore determine the action of $\Dirac_{\gg,\even}^{2}$ on highest weight $\even$-submodules of $M\otimes M(\odd)$. If $M$ has highest weight $\Lambda$ and $V(\mu)$ is a highest weight $\even$-submodule of $M\otimes M(\odd)$, then $\Dirac_{\gg,\even}^{2}$ acts on $V(\mu)$ by the scalar
\begin{equation}
B(\Lambda+\rho,\Lambda+\rho)-B(\mu+\rho_{\bar0},\mu+\rho_{\bar0}).
\end{equation}
We prove that $\ker\Dirac_{\gg,\even}$ is naturally a $\even$-module. The resulting $\even$-module provides a new invariant of $M$, which we expect to be useful for extracting further representation-theoretic information about $M$ in future work. In the present setting, the square formula makes this invariant accessible: a highest weight $\even$-submodule $V(\mu)\subset M\otimes M(\odd)$ lies in $\ker\Dirac_{\gg,\even}^{2}$ precisely when the corresponding shifted $\gg$- and $\even$-Casimir eigenvalues coincide. Proposition~\ref{prop::non_vanishing_HW} gives the resulting non-vanishing criterion and shows, in particular, that $\ker\Dirac_{\gg,\even}^{2}\neq0$ for every highest weight supermodule.

Section~\ref{subsec::osp} studies the simplest example in which the preceding results can be made explicit, namely the affine Kac--Moody superalgebra $\widehat{\mathfrak{osp}}(1|2n)$. Theorem~\ref{thm::kernel_Dirac_affine_osp} determines the kernel of the relative Dirac operator on every simple integrable highest weight supermodule $L_{\bb}(\Lambda)$. As a $\even$-module, neglecting parity, it is the direct sum
\begin{equation}
\ker \Dirac_{\gg,\even} = L_{0}(\Lambda-\rho_{\bar1})
\oplus
L_{0}\bigl(r_{2\beta}(\Lambda+\rho)-\rho_{\bar0}\bigr),
\end{equation}
where $\beta$ is the unique odd simple root. Both constituents occur with multiplicity one and have admissible highest weights. We refer to Section~\ref{subsec::Generalities_KM_Superalgebras} for the definition of $\rho_{\bar1}$. Thus, in this case, the kernel of the square of the Dirac operator can be determined explicitly in terms of the highest weight $\Lambda$. We also discuss the difficulties that arise in extending this analysis to other Kac--Moody superalgebras.

Section~\ref{subsec::unitaizable_case} examines whether the relative Dirac operator can be used to obtain a Dirac inequality for unitarizable highest weight supermodules. Following Jakobsen \cite{Jakobsen_affine_Lie_superalgebras}, unitarity is defined relative to a conjugate-linear anti-involution $\omega$ of $\gg$. We introduce the corresponding notion of an $\omega$-adapted polarization. With respect to such a polarization, the relative Dirac operator is self-adjoint or skew-adjoint, depending on its sign. However, for untwisted or twisted affine Lie superalgebras, the triangular polarization used in our construction is never $\omega$-adapted. Replacing it by an $\omega$-adapted polarization destroys the compatibility with the completion in which the infinite sums defining the Dirac operator are controlled. Consequently, the usual argument leading from self-adjointness and the square formula to a Dirac inequality cannot be applied within the present framework. 

\subsection{Conventions}
We write $\ZZ_{+}$ for the set of non-negative integers and $\ZZ_{2}\coloneqq\ZZ/2\ZZ$, with elements $\bar 0$ and $\bar 1$. The ground field is $\CC$.

Unless stated otherwise, $\gg$ denotes a $\ZZ$-graded quadratic color Lie algebra. All color objects are $\Gamma$-graded. The additional $\ZZ$-grading is the grading used for completions. Thus, by a $\ZZ$-graded color Lie algebra, we mean a color Lie algebra endowed with a compatible $\ZZ$-grading. If, in addition, there is a compatible $\ZZ_{2}$-grading, we speak of a $\ZZ$-graded color superalgebra in order to emphasize the role of this additional $\ZZ_{2}$-grading.

If an object carries both a $\ZZ$-grading and a $\Gamma$-grading, integer subscripts refer to the $\ZZ$-grading and Greek subscripts to the $\Gamma$-grading.

For homogeneous $x\in\gg$, the symbol $L_x$ denotes the induced Lie derivative on the algebra under consideration. Thus the same notation is used on $\bigwedge_{\varepsilon}(\gg)$, $S_{\varepsilon}(\gg)$, $U_{\varepsilon}(\gg)$, $W_{\varepsilon}(\gg)$, and $\WW_{\varepsilon}(\gg)$. The ambient algebra determines the meaning. If several such operators occur simultaneously, we write $L_x^{\wedge}$, $L_x^S$, $L_x^U$, $L_x^W$, and $L_x^{\WW}$.

\section{Preliminaries}\label{sec::preliminaries}
This section introduces further conventions and notation, together with the basic definitions used throughout. In particular, this includes the notion of a color Lie algebra, the associated exterior, Clifford, symmetric, universal enveloping, Weil, and quantum Weil algebras and their relations, their appropriate completions, and the generalized Kac–Peterson class.

\subsection{Color Lie Algebras} \label{subsec::conventions} Color Lie algebras arise as the natural generalization of Lie superalgebras obtained by replacing the parity group $\ZZ_{2}\coloneqq\ZZ/2\ZZ$ and the super sign by an arbitrary abelian grading group $\Gamma$ and a commutation factor $\varepsilon\colon \Gamma\times\Gamma\to \CC^{\times}$. Historically, such structures already appeared in the work of Ree \cite{Ree}, were developed further in the physics literature by Rittenberg and Wyler \cite{RittenbergWyler1978} in connection with generalized statistics, and were subsequently given a systematic algebraic treatment by Scheunert \cite{Scheunert1979}.  Thus color Lie algebras provide a framework for graded Lie theory beyond the ordinary and super cases.

Let $\Gamma$ be an abelian group, written additively, with neutral element $\ng$. A \emph{$\Gamma$-graded vector space} is a vector space $V$ together with a decomposition $V=\bigoplus_{\gamma\in\Gamma}V_{\gamma}$. An element $v\in V_{\gamma}$ is called \emph{homogeneous of degree} $\gamma$, and we write $\vert v\vert =\gamma$. 

A \emph{morphism} $f\colon V\to W$ of $\Gamma$-graded vector spaces is a linear map preserving the grading, that is, $f(V_{\gamma})\subseteq W_{\gamma}$ for all $\gamma\in\Gamma$. For fixed $\Gamma$, we denote by $\Hom_{\gr}(V,W)$ the space of morphisms of $\Gamma$-graded vector spaces from $V$ to $W$, and we set $\End_{\gr}(V)\coloneqq \Hom_{\gr}(V,V)$. We write $\Hom(V,W)$ for the graded vector space of finite sums of homogeneous linear maps from $V$ to $W$, and set $\End(V)\coloneqq\Hom(V,V)$. This is again a $\Gamma$-graded vector space $\Hom(V,W) = \bigoplus_{\gamma \in \Gamma} \Hom(V,W)_{\gamma}$, where
\begin{equation}
\Hom(V,W)_{\gamma} \coloneqq \{T \in \Hom(V,W) : T(V_{\alpha}) \subseteq W_{\alpha+\gamma} \ \text{for all } \alpha \in \Gamma\}. 
\end{equation}

The \emph{$\Gamma$-graded tensor product of two $\Gamma$-graded vector spaces} $V=\bigoplus_{\alpha\in\Gamma}V_{\alpha}$ and $W=\bigoplus_{\beta\in\Gamma}W_{\beta}$ is the $\Gamma$-graded vector space
\begin{equation}V\otimes W=\bigoplus_{\gamma\in\Gamma}(V\otimes W)_{\gamma},\qquad (V\otimes W)_{\gamma}\coloneqq\bigoplus_{\alpha+\beta=\gamma}V_{\alpha}\otimes W_{\beta}.\end{equation}
In particular, if $v\in V_{\alpha}$ and $w\in W_{\beta}$ are homogeneous, then $\vert v\otimes w\vert =\alpha+\beta$.

A \emph{commutation factor} is a map $\varepsilon\colon\Gamma\times\Gamma\to\CC^{\times}$ such that, for all $\alpha,\beta,\gamma\in\Gamma$, one has (see~\cite[III.116]{Bou})
\begin{equation}\label{eq::commutation_factor}\varepsilon(\alpha+\beta,\gamma)=\varepsilon(\alpha,\gamma)\varepsilon(\beta,\gamma),\qquad \varepsilon(\alpha,\beta+\gamma)=\varepsilon(\alpha,\beta)\varepsilon(\alpha,\gamma),\qquad \varepsilon(\alpha,\beta)\varepsilon(\beta,\alpha)=1.\end{equation}
In particular, $\varepsilon(\gamma,\gamma)^{2}=1$ for all $\gamma\in\Gamma$. For $\Gamma$-graded vector spaces, a commutation factor $\varepsilon$ defines a symmetry
\begin{equation}\tau_{V,W}\colon V\otimes W\to W\otimes V,\qquad v\otimes w\mapsto \varepsilon(\vert v\vert ,\vert w\vert)w\otimes v\end{equation}
for homogeneous $v\in V$ and $w\in W$. It satisfies $\tau_{W,V}\circ\tau_{V,W}=\id_{V\otimes W}$. The symmetry induced by $\varepsilon$ yields canonical actions of symmetric groups on tensor powers $V^{\otimes n}$. Concretely, for each $n\geq 1$, there is a unique right action
$
\pi\colon S_{n}\to\operatorname{GL}(V^{\otimes n})
$
such that, for every adjacent transposition $s_{i}=(i,i+1)$,
\begin{equation}
\pi(s_{i})(v_{1}\otimes\cdots\otimes v_{n})=\varepsilon(\vert v_{i}\vert,\vert v_{i+1}\vert)v_{1}\otimes\cdots\otimes v_{i+1}\otimes v_{i}\otimes\cdots\otimes v_{n}
\end{equation}
for homogeneous $v_{1},\dots,v_{n}\in V$. In general, for $\sigma\in S_{n}$,
\begin{equation}
\pi(\sigma)(v_{1}\otimes\cdots\otimes v_{n})=p'(\sigma;v_{1},\dots,v_{n})v_{\sigma(1)}\otimes\cdots\otimes v_{\sigma(n)},
\end{equation}
where
\begin{equation}
p'(\sigma;v_{1},\dots,v_{n})\coloneqq \prod_{\substack{1\leq i<j\leq n\\ \sigma^{-1}(i)>\sigma^{-1}(j)}}\varepsilon(\vert v_{i}\vert, \vert v_{j}\vert).
\end{equation}
Analogously, there is a unique signed right action of $S_{n}$ given by \begin{equation}\pi_{\mathrm{sgn}}(\sigma)(v_1\otimes\cdots\otimes v_n)
=
p(\sigma;v_1,\dots,v_n)v_{\sigma(1)}\otimes\cdots\otimes v_{\sigma(n)},\end{equation} where $p(\sigma;v_{1},\ldots,v_{n})\coloneqq\sgn(\sigma)p'(\sigma;v_{1},\ldots,v_{n})$. Moreover,
\begin{equation}
p(\sigma\sigma';v_{1},\dots,v_{n})=p(\sigma';v_{\sigma(1)},\dots,v_{\sigma(n)})p(\sigma;v_{1},\dots,v_{n}),\qquad p(\operatorname{id};v_{1},\dots,v_{n})=1.
\end{equation}
These are the color analogues of the permutation action and signed permutation action, respectively.

In what follows, we fix the pair $(\Gamma,\varepsilon)$. For homogeneous elements $v$ and $w$, we write $\varepsilon(v,w)\coloneqq\varepsilon(\vert v\vert ,\vert w\vert)$ and $\varepsilon(v)\coloneqq\varepsilon(\vert v\vert ,\vert v\vert)$. Then $\cat$ denotes the symmetric monoidal category whose objects are $\Gamma$-graded vector spaces, whose morphisms are grading-preserving linear maps, whose tensor product is the $\Gamma$-graded tensor product, and whose symmetry is induced by the commutation factor $\varepsilon$.

The category $\cat$ is equipped with the $\varepsilon$-trace. Concretely, let $V\in\cat$ be finite-dimensional, and define $\mathcal E\in\End(V)$ by $\mathcal E(v)\coloneqq\varepsilon(v)v$ for homogeneous $v\in V$, extended linearly. The \emph{$\varepsilon$-trace} is the linear map $\tr_{\varepsilon}\colon\End(V)\to\CC$ defined by
\begin{equation}
\tr_{\varepsilon}(T)\coloneqq\tr(\mathcal E\circ T).
\end{equation}
Its basic property is the $\varepsilon$-cyclicity
\begin{equation}
\tr_{\varepsilon}(ST)=\varepsilon(S,T)\tr_{\varepsilon}(TS)
\end{equation}
for homogeneous $S,T\in\End(V)$. Of particular interest are $\varepsilon$-quadratic spaces in $\cat$. A $\Gamma$-graded vector space $V$ is called \emph{$\varepsilon$-quadratic} if it is endowed with an even, nondegenerate, $\varepsilon$-symmetric bilinear form $\langle\cdot,\cdot\rangle\colon V\times V\to\CC$. Here, \emph{even} means that $\langle v,w\rangle=0$ for homogeneous $v,w\in V$ unless $\vert v\vert +\vert w\vert = \ng$, and \emph{$\varepsilon$-symmetric} means that
$
\langle v,w\rangle=\varepsilon(v,w)\langle w,v\rangle
$
for all homogeneous $v,w\in V$. This notion is particularly relevant for the study of color Lie algebras endowed with invariant bilinear forms, to which we now turn.

A \emph{color algebra} $\calA$ is an algebra object in the symmetric monoidal category of $\Gamma$-graded vector spaces determined by the commutation factor $\varepsilon$. If $\calA$ and $\mathcal{B}$ are color algebras, then $\calA\otimes \mathcal{B}$ is again a color algebra, with multiplication
\begin{equation}(a_{1}\otimes b_{1})(a_{2}\otimes b_{2})=\varepsilon(b_{1},a_{2})a_{1}a_{2}\otimes b_{1}b_{2}\end{equation}
for homogeneous $a_{1},a_{2}\in \calA$ and $b_{1},b_{2}\in \mathcal{B}$. Moreover, a homogeneous map $D\in\End(\calA)$ of degree $\vert D\vert \in\Gamma$ is called an \emph{$\varepsilon$-derivation} if
\begin{equation}\label{def::epsilon_derivation}
D(ab)=D(a)b+\varepsilon(D,a)\,aD(b)
\end{equation}
for all homogeneous $a,b\in\calA$. We denote the $\Gamma$-graded vector space of all $\varepsilon$-derivations of $\calA$ by $\Der_{(\Gamma,\varepsilon)}(\calA)$. The two main examples of color algebras considered below also carry a compatible \(\ZZ_2\)-grading with parity function $p(\cdot)$. In this case we use the combined commutation factor
$
\varepsilon'(D,a)\coloneqq (-1)^{p(D)p(a)}\varepsilon(D,a).
$
A homogeneous $\varepsilon'$-derivation $D\in\Der_{(\Gamma,\varepsilon')}(\calA)$ is called \emph{even} or \emph{odd} according as $p(D)=\bar0$ or $p(D)=\bar1$, and then satisfies
\begin{equation}\label{eq::parity_epsilon_derivations}
D(ab)=D(a)b+(-1)^{p(D)p(a)}\varepsilon(D,a)\,aD(b)
\end{equation}
for homogeneous $a,b\in\calA$. We denote the spaces of even and odd $\varepsilon'$-derivations by $\Der_{(\Gamma,\varepsilon')}^{\bar0}(\calA)$ and $\Der^{\bar{1}}_{(\Gamma,\varepsilon')}(\calA)$, respectively.

A \emph{color Lie algebra} is a Lie algebra object in this symmetric monoidal category. Concretely:

\begin{definition}
 A \emph{color Lie algebra} is a $\Gamma$-graded vector space $\gg =\bigoplus_{\gamma \in \Gamma}\gg_{\gamma}$ equipped with a bilinear bracket $[\cdot,\cdot]: \gg \times \gg \to \gg$ such that for all homogeneous $x,y,z\in \gg$ one has
 \begin{enumerate}
 \item[a)] $[\gg_{\alpha},\gg_{\beta}] \subseteq \gg_{\alpha+\beta}$;
 \item[b)] $[x,y]=-\varepsilon(x,y)[y,x]$;
 \item[c)] $[x,[y,z]]=[[x,y],z]+\varepsilon(x,y)[y,[x,z]].$
 \end{enumerate}
\end{definition}

\begin{example}\label{ex::color_Lie_algebras}
\begin{enumerate}
\item[(a)] For $\Gamma=\ZZ_{2}$ and $\varepsilon(a,b)=(-1)^{ab}$, color Lie algebras are precisely Lie superalgebras.
\item[(b)] Let $V=\bigoplus_{\gamma\in\Gamma}V_{\gamma}$ be a $\Gamma$-graded vector space. Then
$\End(V)=\bigoplus_{\gamma\in\Gamma}\End(V)_{\gamma}$
is a color Lie algebra with $\varepsilon$-bracket
\begin{equation*}[S,T]_{\End}\coloneqq ST-\varepsilon(S,T)TS\end{equation*}
for homogeneous $S,T\in\End(V)$, extended linearly. We also denote this color Lie algebra by $\gl_{\varepsilon}(V)$.
\item[(c)] Let $V\in\cat$ be $\varepsilon$-quadratic with bilinear form $\langle\cdot,\cdot\rangle$. We write $\mathfrak{so}_{\varepsilon}(V,\langle\cdot,\cdot\rangle)$, or simply $\mathfrak{so}_{\varepsilon}(V)$, for the color Lie algebra of $\langle\cdot,\cdot\rangle$-skew endomorphisms of $V$, that is, the color Lie algebra whose homogeneous elements are precisely the homogeneous $T\in\End(V)$ satisfying
\begin{equation*}\label{def::epsilon_orthogonal_algebra}
\langle Tv,w\rangle+\varepsilon(T,v)\langle v,Tw\rangle=0
\end{equation*}
for all homogeneous $v,w\in V$. It is a color Lie algebra with the canonical $\varepsilon$-bracket, called the \emph{$\varepsilon$-orthogonal algebra} of $(V,\langle\cdot,\cdot\rangle)$. Note that $\tr_{\varepsilon}(T)=0$ for any $T \in \mathfrak{so}_{\varepsilon}(V)$ whenever $V$ is finite-dimensional.
\item[(d)] Let $\calA$ be a color algebra. If $D,E\in\Der_{(\Gamma,\varepsilon)}(\calA)$ are homogeneous, then their $\varepsilon$-commutator
\begin{equation*}
[D,E]_{\Der}\coloneqq D\circ E-\varepsilon(D,E)\,E\circ D
\end{equation*}
is again an $\varepsilon$-derivation. Hence $\Der_{(\Gamma,\varepsilon)}(\calA)$ is a color Lie algebra.
\end{enumerate}
\end{example}

A color Lie algebra $\gg$ is called \emph{quadratic} if it is endowed with an invariant bilinear form $B$ such that $(\gg,B)$ is an $\varepsilon$-quadratic $\Gamma$-graded vector space. Here, invariance means that
\begin{equation}
B([x,y],z)=B(x,[y,z])
\end{equation}
for all homogeneous $x,y,z\in\gg$.

By a \emph{color (Lie) superalgebra} we mean a color (Lie) algebra $\gg=\bigoplus_{\gamma\in\Gamma}\gg_{\gamma}$ endowed with a compatible $\ZZ_{2}$-grading. We denote the associated parity of a homogeneous element $x\in\gg$ by $p(x)\in\ZZ_{2}$. Equivalently, after replacing $\Gamma$ by $\ZZ_{2}\times\Gamma$, this is a color (Lie) algebra with commutation factor $\varepsilon'(x,y)=(-1)^{p(x)p(y)}\varepsilon(x,y)$ for all homogeneous $x,y\in \gg$.

A \emph{representation of a color Lie algebra} $\gg$ is a $\Gamma$-graded vector space $V$ together with a homomorphism of color Lie algebras $\rho:\gg\to\mathfrak{gl}_{\varepsilon}(V)$. Equivalently, a \emph{$\gg$-module} is a $\Gamma$-graded vector space $V=\bigoplus_{\gamma\in\Gamma}V_\gamma$ together with a bilinear map $\gg\times V\to V$, $(x,v)\mapsto x\cdot v$, such that for all homogeneous $x,y\in\gg$ and $v\in V$ one has $x\cdot v\in V_{\vert x\vert +\vert v\vert }$ and
\begin{equation}
[x,y]\cdot v=x\cdot(y\cdot v)-\varepsilon(x,y)y\cdot(x\cdot v).
\end{equation}
A basic example is the adjoint module, where $\gg$ acts on itself by $\ad_x(y)\coloneqq[x,y]$. Standard notions such as submodules and simple modules are defined in the usual graded sense.

We now focus on color Lie algebras equipped with an additional $\ZZ$-grading. This includes important infinite-dimensional examples, most notably Kac–Moody superalgebras.

\subsubsection{\texorpdfstring{$\ZZ$}{}-graded Color Lie Algebras} Let $\gg=\bigoplus_{i\in\ZZ}\gg_{i}$ be a $\ZZ$-graded color Lie algebra, that is, a color Lie algebra with homogeneous components $\gg_{i}=\bigoplus_{\gamma\in\Gamma}\gg_{i,\gamma}$ satisfying $[\gg_{i},\gg_{j}]\subseteq\gg_{i+j}$. We assume that each $\gg_{i}$ is finite-dimensional. In particular, $\gg_{0}$ is a color Lie algebra. Thus $\gg$ is naturally $\ZZ\times\Gamma$-graded, and every homogeneous element $x\in\gg$ has a bidegree $(i,\vert x\vert)\in\ZZ\times\Gamma$. For simplicity, we write $\deg(x)\coloneqq i$ and call it the $\ZZ$-degree of $x$. Thus $\gg_{i}$ always refers to the $\ZZ$-grading, whereas $\gg_{\gamma}$ refers to the $\Gamma$-grading. More generally, whenever an object carries both a $\ZZ$-grading and a $\Gamma$-grading, integer subscripts refer to the $\ZZ$-grading, whereas Greek subscripts refer to the $\Gamma$-grading. In this setting, a \emph{$\ZZ$-graded $\gg$-module} is a $\gg$-module $V=\bigoplus_{j\in\ZZ}V_{j}$ such that $\gg_{i}\cdot V_{j}\subseteq V_{i+j}$ for all $i,j\in\ZZ$.

If $\gg$ is quadratic with non-degenerate $\varepsilon$-symmetric invariant bilinear form $B$, we assume in addition that $B(x,y)=0$ for homogeneous $x,y\in\gg$ unless $\vert x\vert +\vert y\vert = \ng$ and $\deg(x)+\deg(y)=0$.

Set $\gg_{+} \coloneqq \bigoplus_{i >0} \gg_{i}$ and $\gg_{-} \coloneqq \bigoplus_{i \leq 0} \gg_{i}$, and let $\pi_{\pm} : \gg \to \gg_{\pm}$ be the associated even projection maps, \emph{i.e.}, they preserve the $\ZZ\times\Gamma$-grading. By construction, both are color Lie subalgebras.

The $\ZZ\times\Gamma$\emph{-graded restricted dual} $\gg^{\ast}$ of $\gg$ is the direct sum over the duals of the $\Gamma$-graded vector spaces $\gg_{i}$ with $\ZZ\times\Gamma$-grading $(\gg^{\ast})_{(i,\gamma)} = (\gg_{(-i, -\gamma)})^{\ast}$. 

We denote by $\End(\gg)$ and $\End_{\gr}(\gg)$ the direct sum of the spaces $\End(\gg)_{i} \coloneqq \bigoplus_{r} \Hom(\gg_{r}, \gg_{r+i})$ and $\End_{\gr}(\gg)_{i} \coloneqq \bigoplus_{r} \Hom_{\gr}(\gg_{r}, \gg_{r+i})$, respectively. The spaces $\End(\gg)_{i}$ and $\End_{\gr}(\gg)_{i}$ consist of linear maps
of $\ZZ$-degree $i$ with only finitely many nonzero $r$-components. Such maps are called \emph{finitary}. We set
\begin{equation} \label{def::completion_End}
 \widehat{\End}(\gg)_{i} \coloneqq \prod_{r} \Hom(\gg_{r}, \gg_{r+i}), \qquad \widehat{\End}_{\gr}(\gg)_{i} \coloneqq \prod_{r} \Hom_{\gr}(\gg_{r}, \gg_{r+i}).
\end{equation}
Then the completions $\widehat{\End}(\gg)$ and $\widehat{\End}_{\gr}(\gg)$ are the direct sums of these spaces, respectively.

Since we consider the restricted dual $\gg^{\ast}$ and fix a nondegenerate bilinear form $B$, we have a natural identification of $\gg$ and $\gg^{\ast}$, given by the usual musical isomorphisms 
\begin{equation} \label{eq::musical_isomorphisms}
\begin{aligned}
&\flat \colon \gg \xrightarrow{\ \sim\ } \gg^{\ast},
\qquad x \longmapsto x^{\flat}, \quad
x^{\flat}(y) \coloneqq B(x,y),\\[2pt]
&\sharp \colon \gg^{\ast} \xrightarrow{\ \sim\ } \gg,
\qquad \alpha \longmapsto \alpha^{\sharp}, \quad
B(\alpha^{\sharp},y)=\alpha(y),
\end{aligned}
\end{equation}
for all $y \in \gg$. In what follows, we identify $\gg$ and $\gg^{\ast}$.

Let $(e_{a})$ be a homogeneous basis of $\gg$ with $B$-dual basis $(e^{a})$, that is, $B(e_{a},e^{b})=\delta_{ab}$. Then $\vert e_{a}\vert =-\vert e^{a}\vert$ and $\deg(e^{a})=-\deg(e_{a})$. Moreover, $\varepsilon(e_{a})=\varepsilon(e_{a},e^{a})$. Indeed,
\begin{equation}1=\varepsilon(e_{a},\ng)=\varepsilon(e_{a},e_{a}+(-e_{a}))=\varepsilon(e_{a},e_{a})\varepsilon(e_{a},-e_{a}),\end{equation}
hence $\varepsilon(e_{a},e^{a})=\varepsilon(e_{a})^{-1}$. Since $\varepsilon(e_{a})^{2}=1$, it follows that $\varepsilon(e_{a})=\varepsilon(e_{a},e^{a})$.

If $(e_{a}^{\ast})$ denotes a homogeneous basis of the restricted dual such that $e_{a}^{\ast}(e_{b})=\delta_{ab}$, then the identification~\eqref{eq::musical_isomorphisms} gives
\begin{equation}
e_{a}^{\ast}=\varepsilon(e_{a})B(e^{a},\cdot) \quad\Leftrightarrow\quad (e_{a}^{\ast})^{\sharp}=\varepsilon(e_{a})e^{a}.
\end{equation}
In particular, we have a natural identification
\begin{equation}\label{eq::identification_End(g)}
 \End(\gg) \cong \gg \otimes \gg^{\ast} \cong \gg \otimes \gg, \qquad f \mapsto \sum_{a}f(e_{a})\otimes e_{a}^{\ast} \mapsto \sum_{a}\varepsilon(e_{a})f(e_{a}) \otimes e^{a}.
\end{equation}
Note that the map is well-defined since $f \in \End(\gg)$ is finitary. 

Moreover, we define the completion
$\widehat{\gg\otimes\gg}$ as the direct sum of the
spaces $\widehat{\gg\otimes\gg}_{(i,\gamma)}$, where
\begin{equation}
\widehat{\gg\otimes\gg}_{(i,\gamma)}
=
\prod_{r\in\mathbb Z}
\bigoplus_{\alpha+\beta=\gamma}
\gg_{(i-r,\alpha)}
\otimes
\gg_{(r,\beta)}.
\end{equation}
The diagonal adjoint action extend
componentwise to this completion. 

The following lemma is immediate and will be useful in what follows.

\begin{lemma} \label{lemm::basis_description} Let $(e_{a})$ be a homogeneous basis of $\gg$ with $B$-dual basis $(e^{a})$. 
\begin{enumerate}
 \item[a)] For any $x \in \gg$, one has 
 \[
 x = \sum_{a}B(x,e^{a})e_{a} = \sum_{a} \varepsilon(e_{a})B(x,e_{a})e^{a} = \sum_{a} B(e_{a},x)e^{a}.
 \]
 \item[b)] For any $x,y \in \gg$, one has
 \[
 B(x,y) = \sum_{a}B(x,e^{a})B(e_{a},y).
 \]
 \item[c)] Let $f:\widehat{\gg\otimes \gg}\to \gg$ be an $\varepsilon$-skew-symmetric linear map. Then
\[
f\!\left(\sum_a \varepsilon(e_{a}) e_a\otimes e^a\right)=0.
\]
In particular, $p = \sum_{a}\varepsilon(e_{a})e_{a}\otimes e^{a} = \sum_{a}e^{a}\otimes e_{a} \in \widehat{\gg\otimes \gg}$ is $\ad$-invariant.
 \end{enumerate}
\end{lemma}

In what follows, $(\gg,B)$, or simply $\gg$, denotes a $\ZZ$-graded quadratic color Lie algebra, unless otherwise stated.

\subsection{Completion of \texorpdfstring{$\varepsilon$}{}-Exterior and \texorpdfstring{$\varepsilon$}{}-Clifford Algebras of \texorpdfstring{$\gg$}{}} \label{subsec::completion_exterior_clifford_superalgebra} Chen–Kang \cite{generalized_Clifford} and Meyer \cite{Meyer} introduced $\varepsilon$-exterior algebras and $\varepsilon$-Clifford algebras for finite-dimensional color Lie algebras. In this section, we adopt their notation and conventions in order to define these objects for the $\ZZ$-graded color Lie algebras $\gg$ introduced above. In particular, we also formulate the corresponding completions.

Consider the tensor algebra $\mathrm{T}(\gg)$, endowed with the $\ZZ\times \Gamma$-grading induced from $\gg$, where each generator in $\gg_{(i,\gamma)}$ is assigned bidegree $(i,\gamma)$. It is a unital associative algebra with identity element $1_{\mathrm{T}(\gg)}$, and moreover it is naturally equipped with an additional canonical $\ZZ$-grading: for each integer $k \geq 0$, the homogeneous component of tensor length $k$, denoted $\mathrm{T}^{k}(\gg)$, is the subspace spanned by all elements $x_{1}\otimes \cdots \otimes x_{k}$ with $x_{i} \in \gg$. By convention one sets $\mathrm{T}^{0}(\gg) = \CC$. This tensor-length grading is compatible with the $\ZZ\times\Gamma$-grading in the evident sense, that is, $
\mathrm{T}(\gg)=\bigoplus_{k\geq 0}\ \bigoplus_{(i,\gamma)\in\ZZ\times\Gamma}\mathrm{T}^{k}(\gg)_{(i,\gamma)}$. Moreover, reducing this grading modulo $2$ endows $\mathrm{T}(\gg)$ with the structure of a superalgebra, with even and odd parts
\begin{equation}
\mathrm{T}(\gg)_{\bar{0}} \coloneqq \bigoplus_{k \in 2\ZZ_{+}} \mathrm{T}^{k}(\gg), \qquad
\mathrm{T}(\gg)_{\bar{1}} \coloneqq \bigoplus_{k \in 2\ZZ_{+}+1} \mathrm{T}^{k}(\gg).
\end{equation}
In what follows, we consider $\mathrm{T}(\gg)$ naturally as a $(\ZZ\times\ZZ)$-graded color algebra.

\subsubsection{\texorpdfstring{
\texorpdfstring{$\varepsilon$}{}-Exterior Algebra and Completion}{Gamma-graded exterior algebra and completion}} \label{subsubsec::Exterior_Algebra_and Completion} Let $I_{\wedge}$ denote the ($\ZZ\times \Gamma$-graded) homogeneous ideal of $\mathrm{T}(\gg)$ generated by the elements
 \begin{equation}
x\otimes y+\varepsilon(x,y)y\otimes x,\qquad x,y\in\gg\ \text{homogeneous}.
\end{equation}
The quotient $\bigwedge_{\varepsilon}(\gg)\coloneqq \mathrm{T}(\gg)/I_{\wedge}$ is the $\varepsilon$-\emph{exterior algebra} of $\gg$, that is, the unital associative $\ZZ$-graded color algebra generated by $\gg$ with defining relations $x\wedge y+\varepsilon(x,y)y\wedge x=0$ for homogeneous $x,y\in\gg$. We call $\wedge$ the \emph{exterior multiplication} and we write $\bigwedge_{\varepsilon}(\gg)_{(i,\gamma)}$ for its homogeneous component of bidegree $(i,\gamma)\in\ZZ\times\Gamma$.

The canonical tensor-length $\ZZ$-grading of $\mathrm{T}(\gg)$ induces an additional $\ZZ$-grading on $\bigwedge_{\varepsilon}(\gg)$, giving the decomposition 
$\bigwedge_{\varepsilon}(\gg) = \bigoplus_{k \geq 0} \bigwedge^{k}_{\varepsilon}(\gg)$, 
where $\bigwedge^{k}_{\varepsilon}(\gg)$ is the subspace spanned by exterior products of $k$ elements of $\gg$. We denote the associated degree by $\vert \cdot\vert_{\wedge}$. By declaring 
$\bigwedge_{\varepsilon}(\gg)_{\bar{0}} \coloneqq \bigoplus_{k \in 2\ZZ_{+}} \bigwedge^{k}_{\varepsilon}(\gg)$ and 
$\bigwedge_{\varepsilon}(\gg)_{\bar{1}} \coloneqq \bigoplus_{k \in 2\ZZ_{+}+1} \bigwedge^{k}_{\varepsilon}(\gg)$, 
the algebra $\bigwedge_{\varepsilon}(\gg)$ is endowed with its natural structure of a superalgebra, compatible with its $\ZZ\times \Gamma$-grading. In what follows, we consider $\bigwedge_{\varepsilon}(\gg)$ as a $\ZZ \times \ZZ$-graded color algebra. Moreover, we write $\bigwedge_{\varepsilon}(\gg_{\pm})$ for the $\varepsilon$-exterior algebra of the $\ZZ$-graded color subalgebras $\gg_{\pm}$, and $\bigwedge_{\varepsilon}(\gg_{\pm})_{(i,\gamma)}$ for their homogeneous components of bidegree $(i,\gamma)$. Both are $\ZZ\times \ZZ$-graded color subalgebras of $\bigwedge_{\varepsilon}(\gg)$ such that $\bigwedge_{\varepsilon}(\gg)\cong \bigwedge_{\varepsilon}(\gg_{-})\otimes \bigwedge_{\varepsilon}(\gg_{+})$ as $\ZZ\times \ZZ$-graded objects in $\cat$.

On $\bigwedge_{\varepsilon}(\gg)$, there are three natural operations. For $v\in\gg$, let $\ell(v)$ denote left exterior multiplication on $\bigwedge_{\varepsilon}(\gg)$. Moreover, the operation
\begin{equation}\label{eq::contraction}
\iota_{v}(x_{1}\otimes\cdots\otimes x_{l})\coloneqq\sum_{k=1}^{l}(-1)^{k-1}\varepsilon(v,x_{1}+\cdots+x_{k-1})B(v,x_{k})x_{1}\otimes\cdots\otimes\widehat{x}_{k}\otimes\cdots\otimes x_{l}
\end{equation}
on $\mathrm{T}(\gg)$ leaves $I_{\wedge}$ invariant (\emph{cf.}~\cite{Meyer}) and therefore descends to an operator on $\bigwedge_{\varepsilon}(\gg)$, called \emph{contraction}. It satisfies the Leibniz rule
\begin{equation}\label{eq::Leibniz_rule_eins}\iota_{v}(\omega\wedge\eta)=\iota_{v}(\omega)\wedge\eta+(-1)^{\vert \omega\vert_{\wedge}}\varepsilon(v,\omega)\omega\wedge\iota_{v}(\eta)\end{equation}
for homogeneous $\omega,\eta\in\bigwedge_{\varepsilon}(\gg)$. Finally, for homogeneous $T\in\End(\gg)$, define
\begin{equation}\label{eq::Lie_derivative}
L_{T}(x_{1}\otimes\cdots\otimes x_{l})\coloneqq\sum_{k=1}^{l}\varepsilon(T,x_{1}+\cdots+x_{k-1})x_{1}\otimes\cdots\otimes T(x_{k})\otimes\cdots\otimes x_{l}
\end{equation}
for $x_{1},\dots,x_{l}\in\gg$. This operator preserves the ideal $I_{\wedge}$ and hence descends to an operator on $\bigwedge_{\varepsilon}(\gg)$, called the \emph{Lie derivative}. It is an $\varepsilon$-derivation $L_{T}\in \Der_{(\Gamma,\varepsilon)}(\bigwedge_{\varepsilon}(\gg))$, that is, 
\begin{equation} \label{eq::Leibniz_rule_zwei}
L_{T}(\omega\wedge\eta)=L_{T}(\omega)\wedge\eta+\varepsilon(T,\omega)\omega\wedge L_{T}(\eta)
\end{equation}
for homogeneous $\omega,\eta\in\bigwedge_{\varepsilon}(\gg)$.

To describe the relation between the Lie derivative and the contraction, let $T\in\End(\gg)$ be homogeneous. Since $B$ is nondegenerate, we can define its $B$-adjoint $T^{\ast}$ by
\begin{equation}
B(T^{\ast}x,y)=\varepsilon(T,x)B(x,Ty)
\end{equation}
for homogeneous $x,y\in\gg$. For example, by invariance of $B$, the adjoint of $\ad_{x}$ is $-\ad_{x}$ for every homogeneous $x\in\gg$, whereas if $T \in \mathfrak{so}_{\varepsilon}(\gg)$, then $T^{\ast}=-T$.
\begin{lemma}\label{lemm::commutator_Lie_derivative_contraction}
For all homogeneous $T\in\End(\gg)$ with adjoint $T^{\ast}$ and all homogeneous $z\in\gg$, one has
\begin{equation*}[L_{T},\iota_{z}]_{\End}=L_{T}\iota_{z}-\varepsilon(T,z)\iota_{z}L_{T}=-\iota_{T^{\ast}(z)}.\end{equation*}
In particular, if $T\in\mathfrak{so}_{\varepsilon}(\gg)$, then
$[L_{T},\iota_{z}]_{\End}=\iota_{T(z)}$
for all $z\in\gg$.
\end{lemma}
\begin{proof}
Since $L_{T}$ and $\iota_{z}$ are determined on $\bigwedge_{\varepsilon}(\gg)$ by their values on $\gg$ together with~\eqref{eq::Leibniz_rule_eins} and~\eqref{eq::Leibniz_rule_zwei}, the same holds for their $\varepsilon$-commutator. Hence it is enough to verify the identity on generators. Let $x\in\gg$ be homogeneous. Then $\iota_{z}(x)=B(z,x)$, $L_{T}(x)=T(x)$, and $L_{T}(B(z,x))=0$. Therefore
\begin{equation*}\begin{split}
[L_{T},\iota_{z}]_{\End}(x)&=L_{T}(\iota_{z}(x))-\varepsilon(T,z)\iota_{z}(L_{T}(x))=L_{T}(B(z,x))-\varepsilon(T,z)\iota_{z}(Tx)
\\&=-\varepsilon(T,z)B(z,Tx)
=-B(T^{\ast}z,x)
=-\iota_{T^{\ast}(z)}(x).
\end{split}\end{equation*}
If $T\in\mathfrak{so}_{\varepsilon}(\gg)$, then $T^{\ast}=-T$, and hence $[L_{T},\iota_{z}]_{\End}=\iota_{T(z)}$.
\end{proof}

In addition, $\bigwedge_{\varepsilon}(\gg)$ carries a bilinear form induced by $B$. For $k\geq 0$, define on $\mathrm{T}^{k}(\gg)$
\begin{equation}
\langle x_{1}\otimes\cdots\otimes x_{k},y_{1}\otimes\cdots\otimes y_{k}\rangle_{\mathrm{T}^{k}}\coloneqq (-1)^{\tfrac{k(k-1)}{2}}\prod_{i=0}^{k-1}B(x_{k-i},y_{1+i}),
\end{equation}
and then set
\begin{equation}
\langle x_{1}\otimes\cdots\otimes x_{k},y_{1}\otimes\cdots\otimes y_{k}\rangle_{\wedge^{k}}\coloneqq\sum_{\sigma\in S_{k}}p(\sigma;x_{1},\ldots,x_{k})\langle x_{\sigma(1)}\otimes\cdots\otimes x_{\sigma(k)},y_{1}\otimes\cdots\otimes y_{k}\rangle_{\mathrm{T}^{k}},
\end{equation}
where $p(\sigma;x_{1},\ldots,x_{k})$ is the signed color permutation factor from Section~\ref{subsec::conventions}. This form descends to $\bigwedge^{k}_{\varepsilon}(\gg)$, and hence defines a bilinear form $\langle\cdot,\cdot\rangle_{\wedge}$ on $\bigwedge_{\varepsilon}(\gg)$. As in \cite[Section 5]{generalized_Clifford}, one proves the following lemma. 

\begin{lemma}\label{lemm::action_iota_inner_product_wedge}
The bilinear form $\langle\cdot,\cdot\rangle_{\wedge}$ is nondegenerate and $\varepsilon$-symmetric. Moreover, for all homogeneous $x\in\gg$ and all $y,z\in\bigwedge_{\varepsilon}(\gg)$,
\begin{equation*}
\langle \iota_{x}y,z\rangle_{\wedge}=\varepsilon(x,y)\langle y,x\wedge z\rangle_{\wedge}.
\end{equation*}
\end{lemma}

Recall that $\gg^{\ast}$ denotes the restricted dual of $\gg$. The natural pairing between $\bigwedge_{\varepsilon}(\gg)$ and $\bigwedge_{\varepsilon}(\gg^{\ast})$ identifies $\bigwedge_{\varepsilon}(\gg)_{(i,\gamma)}$ with a subspace of the linear maps $\bigwedge_{\varepsilon}(\gg^{\ast})_{(-i,-\gamma)} \to \CC$. We use the completion determined by the decomposition $\gg=\gg_{-}\oplus\gg_{+}$. Using
$\bigwedge_{\varepsilon}(\gg) \cong \bigwedge_{\varepsilon}(\gg_{-})\otimes \bigwedge_{\varepsilon}(\gg_{+})$,
we define
$
\widehat{\bigwedge}_{\varepsilon}(\gg)
\coloneqq
\bigoplus_{(i,\gamma)\in \ZZ\times\Gamma}
\widehat{\bigwedge}_{\varepsilon}(\gg)_{(i,\gamma)},
$
where
\begin{equation}
\widehat{\bigwedge}_{\varepsilon}(\gg)_{(i,\gamma)}
\coloneqq
\prod_{r\geq 0}
\ \bigoplus_{\substack{\alpha,\beta\in\Gamma\\ \alpha+\beta=\gamma}}
\bigwedge\nolimits_{\varepsilon}(\gg_{-})_{(i-r,\alpha)}
\otimes
\bigwedge\nolimits_{\varepsilon}(\gg_{+})_{(r,\beta)}.
\end{equation}
The exterior multiplication extends componentwise to $\widehat{\bigwedge}_{\varepsilon}(\gg)$. Thus $\widehat{\bigwedge}_{\varepsilon}(\gg)$ is a $\ZZ$-graded color algebra. The contraction and the Lie derivative extend to the completion and are denoted by the same symbols. The bilinear form $\langle\cdot,\cdot\rangle_{\wedge}$ extends to a pairing between $\widehat{\bigwedge}_{\varepsilon}(\gg)$ and $\bigwedge_{\varepsilon}(\gg)$, again denoted by $\langle\cdot,\cdot\rangle_{\wedge}$. Likewise, for each $k\geq 0$ one defines the completed exterior powers $\widehat{\bigwedge}_{\varepsilon}^{k}(\gg)$.

\subsubsection{\texorpdfstring{$\varepsilon$-Exterior Algebra and $\mathfrak{so}_{\varepsilon}(\gg)$}{}}\label{subsubsec::epsilon_exterior_and_so}
We consider the $\varepsilon$-orthogonal algebra $\mathfrak{so}_{\varepsilon}(\gg)$ of $(\gg,B)$ introduced in Example~\ref{ex::color_Lie_algebras}. It is a $\ZZ$-graded quadratic color Lie subalgebra of $\End(\gg)$, where the $\ZZ$-grading is induced by that of $\gg$.
\begin{lemma}\label{lemm::form_epsilon_orthogonal}
 $\mathfrak{so}_{\varepsilon}(\gg)$ is quadratic with $\varepsilon$-trace form \[
 (S,T) \mapsto \tr_{\varepsilon}(ST).
 \]
\end{lemma}

\begin{proof}
Let $0\neq S\in\mathfrak{so}_\varepsilon(\gg)$ and choose a nonzero homogeneous component $S_{(i,\gamma)}$. Since $S$ is finitary, it is supported on a finite-dimensional nondegenerate graded subspace $W\subseteq\gg$. By the finite-dimensional result, there exists $T\in\mathfrak{so}_\varepsilon(W)_{(-i,-\gamma)}$ such that $\operatorname{tr}_\varepsilon(S_{(i,\gamma)}T)\neq0$. Extending $T$ by zero on $W^\perp$ gives $T\in\mathfrak{so}_\varepsilon(\gg)$ with $\operatorname{tr}_\varepsilon(ST)\neq0$. Hence the trace form is nondegenerate. Its $\varepsilon$-symmetry and invariance follow from the $\varepsilon$-cyclicity of the trace.
\end{proof}

We define the completion $\widehat{\mathfrak{so}}_{\varepsilon}(\gg)$ of $\mathfrak{so}_{\varepsilon}(\gg)$ as in~\eqref{def::completion_End}.

\begin{lemma}\label{lemm::ad_in_completion}
For every homogeneous $x\in\gg$, one has $\ad_{x}\in\widehat{\mathfrak{so}}_{\varepsilon}(\gg)$. Moreover, if $x\in\gg_{i}$, then $\ad_{x}$ has $\ZZ$-degree $i$, that is,
$
\ad_{x}(\gg_{j})\subseteq \gg_{i+j}
$
for all $j\in\ZZ$.
\end{lemma}
\begin{proof}
Let $x,y,z\in\gg$ be homogeneous. Note that $\vert\ad_{x}\vert=\vert x\vert $. By $\varepsilon$-skew-symmetry of $[\cdot,\cdot]$, invariance, and $\varepsilon$-symmetry of $B$,
\begin{equation*}
\begin{split}
B(\ad_{x}y,z)&=B([x,y],z)=-\varepsilon(x,y)B([y,x],z)=-\varepsilon(x,y)B(y,[x,z])=-\varepsilon(x,y)B(y,\ad_{x}z).
\end{split}
\end{equation*}
Hence $\ad_{x}\in\widehat{\mathfrak{so}}_{\varepsilon}(\gg)$. If $x\in\gg_{i}$, then $[x,\gg_{j}]\subseteq\gg_{i+j}$ by the $\ZZ$-grading of $\gg$, so $\ad_{x}$ has $\ZZ$-degree $i$.
\end{proof}

Motivated by the finite-dimensional case \cite{Meyer}, we define the  \emph{moment map} $\upmu\colon\gg\times\gg\to\mathfrak{so}_{\varepsilon}(\gg)
$ by
\begin{equation}\label{eq::explicit_def_moment_map}
\upmu(x,y)(z)=\varepsilon(y,z)B(x,z)y-B(y,z)x,\qquad x,y,z\in\gg.
\end{equation}
The map $\upmu$ is $\varepsilon$-skew-symmetric and homogeneous of degree $(0,\ng)$. 

\begin{proposition} $\upmu:\bigwedge^{2}_{\varepsilon}(\gg)\to\mathfrak{so}_{\varepsilon}(\gg)$ is an isomorphism of $\ZZ\times\Gamma$-graded vector spaces, with inverse $\uplambda \coloneqq \upmu^{-1}:\mathfrak{so}_{\varepsilon}(\gg)\to\bigwedge^{2}_{\varepsilon}(\gg)$, given in a basis $(e_{a})$ of $\gg$, with $B$-dual basis $(e^{a})$, by
\begin{equation*} \label{eq::def_lambda}
 \uplambda(T) \coloneqq \upmu^{-1}(T)=-\tfrac{1}{2}\sum_{a}T(e^{a})\wedge e_{a}.
\end{equation*}
\end{proposition}

\begin{proof}
Meyer proves the corresponding statement for finite-dimensional $\Gamma$-graded quadratic spaces \cite[Prop.~2.13]{Meyer}. Let $\nu\in\bigwedge^{2}_{\varepsilon}(\gg)$ or $T\in\mathfrak{so}_{\varepsilon}(\gg)$. Since $\nu$ is a finite sum of wedges and $T$ is finitary, only finitely many homogeneous components of $\gg$ occur. Hence there exists a finite-dimensional nondegenerate $\ZZ\times\Gamma$-graded subspace $W\subset\gg$ containing them. Applying Meyer's result to $W$ yields mutually inverse maps
$
\bigwedge^{2}_{\varepsilon}(W)\cong\mathfrak{so}_{\varepsilon}(W),
$
where the inverse is given by the stated formula. Since $\uplambda$ and $\upmu$ preserve the $\ZZ\times\Gamma$-grading, they induce mutually inverse maps
$
\bigwedge^{2}_{\varepsilon}(\gg)\cong\mathfrak{so}_{\varepsilon}(\gg)
$
of $\ZZ\times\Gamma$-graded vector spaces.
\end{proof}

\begin{lemma} \label{lemm::Lie_derivative_and_lambd}
 For any $T_{1},T_{2} \in \mathfrak{so}_{\varepsilon}(\gg)$, we have 
 \[
 L_{T_{1}} \uplambda(T_{2}) = \uplambda([T_{1},T_{2}]_{\mathfrak{so}_{\varepsilon}}).
 \]
\end{lemma}

\begin{proof}
Let $(e_{a})$ be a homogeneous basis of $\gg$ with $B$-dual basis $(e^{a})$. Then
\begin{equation*}
\uplambda(T) = -\tfrac{1}{2}\sum_{a} T(e^{a}) \wedge e_{a} = \tfrac{1}{2}\sum_{a} \varepsilon(T,e_{a})e^{a}\wedge T(e_{a}) = -\tfrac{1}{2}\sum_{a}\varepsilon(e_{a})T(e_{a})\wedge e^{a}, \qquad T \in \mathfrak{so}_{\varepsilon}(\gg)
\end{equation*}
using Lemma~\ref{lemm::basis_description}. Moreover, since $T_{1},T_{2}\in \mathfrak{so}_{\varepsilon}(\gg)$, we have in $\gg \otimes \gg$:
\[
0=(T_{1}\otimes 1 + 1 \otimes T_{1})(\sum_{a}e^{a} \otimes e_{a}) = \sum_{a}(T_{1}(e^{a})\otimes e_{a}+\varepsilon(T_{1},e^{a})e^{a}\otimes T_{1}(e_{a})).
\]
Applying $T_{2}\otimes 1$ and descending to the quotient $\bigwedge^{2}_{\varepsilon}(\gg)$ gives
\[
-\sum_{a}T_{2}(T_{1}(e^{a}))\wedge e_{a} = \sum_{a}\varepsilon(T_{1},e^{a})T_{2}(e^{a})\wedge T_{1}(e_{a}).
\]
Consequently, we obtain
\[
\begin{split}
L_{T_{1}}\uplambda(T_{2}) 
&= -\tfrac{1}{2}\sum_{a}\Bigl[T_{1}(T_{2}(e^{a})) \wedge e_{a} 
+ \varepsilon(T_{1},T_{2}+e^{a})T_{2}(e^{a}) \wedge T_{1}(e_{a})\Bigr] \\[2pt]
&=-\tfrac{1}{2}\sum_{a}T_{1}(T_{2}(e^{a}))\wedge e_{a}-\varepsilon(T_{1},T_{2})\tfrac{1}{2}\sum_{a}\varepsilon(T_{1},e^{a})T_{2}(e^{a})\wedge T_{1}(e_{a}) \\[2pt]
&=-\tfrac{1}{2}\sum_{a}T_{1}(T_{2}(e^{a}))\wedge e_{a} + \varepsilon(T_{1},T_{2})\tfrac{1}{2}\sum_{a}T_{2}T_{1}(e^{a})\wedge e_{a}
\\[2pt] &=-\tfrac{1}{2}\sum_{a}\Bigl[T_{1}(T_{2}(e^{a})) \wedge e_{a} - \varepsilon(T_{1},T_{2})T_{2}T_{1}(e^{a})\wedge e_{a}\Bigr] \\ &= \uplambda([T_{1},T_{2}]_{\mathfrak{so}_{\varepsilon}})
. \qedhere
\end{split} 
\]
\end{proof}

Under the identification $\gg \cong \gg^{\ast}$ via $B$, we can describe $\uplambda$ as follows. 
For any $T \in \mathfrak{so}_{\varepsilon}(\gg)$, define $\omega_{T}(x,y) \coloneqq B(Tx,y)$ for all $x,y \in \gg$. 
A direct calculation shows that this form is skew-$\varepsilon$-symmetric.

\begin{lemma}\label{lemm::identification_lambda_omega}
Under the identification $\gg \cong \gg^{\ast}$, we have 
\[
(\uplambda(T))^{\flat} = \omega_{T}, \qquad T \in \mathfrak{so}_{\varepsilon}(\gg).
\]
\end{lemma}

\begin{proof}
For homogeneous $x,y,z,w\in\gg$, one has
\begin{equation}
B(\upmu(x,y)z,w)=\varepsilon(y,z)B(x,z)B(y,w)-B(y,z)B(x,w).
\end{equation}
By definition of the induced map
$
\bigwedge^{2}_{\varepsilon}(\gg)\to \bigwedge^{2}_{\varepsilon}(\gg^{\ast})$, $\nu\mapsto \nu^{\flat}$, the right-hand side is precisely $(x\wedge y)^{\flat}(z,w)$. Thus
\begin{equation*}
\omega_{\upmu(x,y)}=(x\wedge y)^{\flat}.
\end{equation*}
Since $\uplambda=\upmu^{-1}$, it follows that for every $T\in\mathfrak{so}_{\varepsilon}(\gg)$, one has 
$
(\uplambda(T))^{\flat}=\omega_{T}.
$
\end{proof}

Since $\upmu$ is linear, componentwise finite and preserves the $\ZZ\times\Gamma$-grading, it extends componentwise to the completions. The inverse $\uplambda$ extends correspondingly, and both are $\ZZ\times \Gamma$-graded isomorphisms in $\cat$. We denote the extensions by the same symbols.

\subsubsection{\texorpdfstring{$\varepsilon$}{}-Clifford Algebra and Completion} We now introduce the $\varepsilon$-Clifford algebra associated with the quadratic color Lie algebra $(\gg,B)$. Let $I_{\Cl}(\gg)$ be the two-sided ($\ZZ\times \Gamma$-graded) homogeneous ideal of $\mathrm{T}(\gg)$ generated by all elements
\begin{equation}
x\otimes y+\varepsilon(x,y)y\otimes x-2B(x,y)1_{\mathrm{T}(\gg)},\qquad x,y\in\gg \ \text{homogeneous}.
\end{equation}
The quotient
$
\Cl_{\varepsilon}(\gg)\coloneqq \mathrm{T}(\gg)/I_{\Cl}(\gg)
$
is a $\ZZ$-graded associative color algebra, called the $\varepsilon$-Clifford algebra. Identifying $\gg$ with its image in $\Cl_{\varepsilon}(\gg)$, it is generated by $\gg$ subject to the relations
\begin{equation}\label{eq::Clifford_relation}
xy+\varepsilon(x,y)yx=2B(x,y)1_{\Cl_{\varepsilon}(\gg)}
\end{equation}
for all $x,y\in\gg$ homogeneous. The operators~\eqref{eq::contraction} and~\eqref{eq::Lie_derivative} preserve $I_{\Cl}(\gg)$. Hence, for every $x\in\gg$, the operator~\eqref{eq::contraction} descends to an operator $\iota_x$ on $\Cl_{\varepsilon}(\gg)$, and for every $T\in\mathfrak{so}_{\varepsilon}(\gg)$, the operator~\eqref{eq::Lie_derivative} descends to an operator $L_T$ on $\Cl_{\varepsilon}(\gg)$.

they descend to operators on $\Cl_{\varepsilon}(\gg)$, again denoted by $\iota_{x}$ and $L_{T}$. They satisfy the Leibniz rules~\eqref{eq::Leibniz_rule_eins} and~\eqref{eq::Leibniz_rule_zwei} with respect to the Clifford product.

The restriction of $B$ to $\gg_{\pm}$ yields a natural construction of Clifford algebras $\Cl_{\varepsilon}(\gg_{\pm})$, where $\Cl_{\varepsilon}(\gg_{+}) = \bigwedge_{\varepsilon}(\gg_{+})$ as $B$ restricts to zero on $\gg_{+}$. These are $\ZZ$-graded color subalgebras of $\Cl_{\varepsilon}(\gg)$ and the Clifford multiplication defines an isomorphism of $\ZZ\times\Gamma$-graded vector spaces $\Cl_{\varepsilon}(\gg) \cong \Cl_{\varepsilon}(\gg_{-}) \otimes \Cl_{\varepsilon}(\gg_{+})$. As for the $\ZZ$-graded $\varepsilon$-exterior algebra, we define the completion $\widehat{\Cl}_{\varepsilon}(\gg)$ of $\Cl_{\varepsilon}(\gg)$ as the direct sum of all
\begin{equation}
 \widehat{\Cl}_{\varepsilon}(\gg)_{(i,\gamma)} 
\coloneqq
\prod_{r\geq 0}\ \bigoplus_{\substack{\alpha,\beta\in\Gamma\\ \alpha+\beta=\gamma}}
\Cl_{\varepsilon}(\gg_{-})_{(i-r,\alpha)}
\otimes
\Cl_{\varepsilon}(\gg_{+})_{(r,\beta)}.
\end{equation}

\subsubsection{\texorpdfstring{$\widehat{\bigwedge}_{\varepsilon}(\gg), \widehat{\Cl}_{\varepsilon}(\gg)$}{} and Normal-Ordered Quantization}\label{subsubsec::normal_ordered_quantization}

We now relate $\widehat{\bigwedge}_{\varepsilon}(\gg)$ and $\widehat{\Cl}_{\varepsilon}(\gg)$. We define $\gamma : \gg \to \End(\bigwedge_{\varepsilon}(\gg))$ by $\gamma(x) \coloneqq \ell(x) + \iota_{x}$ where $\ell(x)$ denotes exterior multiplication by $x$ and $\iota_{x}$ denotes contraction. It satisfies $\gamma(x)\gamma(y) + \varepsilon(x,y)\gamma(y)\gamma(x) = 2B(x,y)$ for any homogeneous $x,y \in \gg$. Thus, by the universal property of the $\varepsilon$-Clifford algebra (see \cite[Proposition 3.1]{generalized_Clifford}), this realizes $\bigwedge_{\varepsilon}(\gg)$ as a $\Cl_{\varepsilon}(\gg)$-module, \emph{i.e.}, $\gamma : \Cl_{\varepsilon}(\gg) \to \End(\bigwedge_{\varepsilon}(\gg))$. 

\begin{theorem} \label{thm::quantization_map} The map $\eta : \Cl_{\varepsilon}(\gg) \to \bigwedge_{\varepsilon}(\gg)$ with $\eta(v) \coloneqq \gamma(v) 1_{\bigwedge_{\varepsilon}(\gg)}$ is an isomorphism of $\ZZ\times \Gamma$-graded vector spaces. Moreover, the inverse map is given by the quantization map $q \coloneqq \sum_{k} q_{k} : \bigwedge_{\varepsilon}(\gg) \to \Cl_{\varepsilon}(\gg)$ with 
\[
q_{k} (x_{1} \wedge \ldots \wedge x_{k}) \coloneqq \tfrac{1}{k!}\sum_{\sigma \in S_{k}} p(\sigma;x_{1}, \ldots, x_{k}) x_{\sigma(1)}\ldots x_{\sigma(k)},
\]
where
\[
p(\sigma; x_{1},\ldots, x_{k}) = \sgn(\sigma)\prod_{1 \leq i < j \leq k, \ \sigma^{-1}(i) > \sigma^{-1}(j)} \varepsilon(x_{i},x_{j}).
\]
\end{theorem}
\begin{proof} In \cite{generalized_Clifford}, the authors prove the case of finite-dimensional $\Gamma$-graded vector spaces. The same argument applies in the present $\ZZ$-graded setting, since all expressions involved depend on only finitely many homogeneous components of $\gg$. They therefore reduce to the corresponding statement on a finite-dimensional nondegenerate $\ZZ\times\Gamma$-graded subspace $W\subset \gg$, where the finite-dimensional proof applies verbatim.
\end{proof}

\begin{example}
\label{ex:isotropic_Q} \begin{enumerate}
 \item[a)]
If $x_{1},\dots,x_{n}$ span an isotropic subspace of $\gg$, then for any permutation $\sigma$ one obtains
\[
 x_{\sigma(1)} \cdots x_{\sigma(n)}
 = p(\sigma; x_{1},\dots,x_{n})^{-1} x_{1}\cdots x_{n},
\]
and, in particular,
\[
 q_{n}(x_{1}\wedge \dots \wedge x_{n}) = x_{1}\cdots x_{n}.
\]
\item[b)] If $(e_{a})$ denotes a homogeneous basis of $\gg$ with $B$-dual basis $(e^{a})$, then for any $T\in\mathfrak{so}_{\varepsilon}(\gg)$
\[
q_{2}(\uplambda(T))
=-\tfrac{1}{4}\sum_{a}\Bigl(T(e^{a})e_{a}
-\varepsilon(T+e^{a},e_{a})e_{a}T(e^{a})\Bigr) = -\tfrac{1}{2}\sum_{a}\Bigl(T(e^{a})e_{a}-B(T(e^{a}),e_{a})\Bigr).
\]

\end{enumerate}
\end{example}

As in \cite[Lemma 2.7]{Dirac_quadratic}, the quantization map intertwines contraction.

\begin{lemma}\label{lemm::quantization_map_non_completion_and_contraction}
 The quantization map $q$ intertwines contractions, that is
 \[
 \iota_{x} \circ q = q \circ \iota_{x}, \qquad \text{for all} \ x \in \gg.
 \]
\end{lemma}

The quantization map $q : \bigwedge_{\varepsilon}(\gg) \to \Cl_{\varepsilon}(\gg)$ does not extend to completions, which leads to the \emph{normal-ordered quantization map}.

\begin{definition}\label{def::normal_ordered_quantization}
The \emph{normal-ordered quantization map} is the $\ZZ\times\Gamma$-graded linear map
\[
\widehat q:\widehat{\bigwedge}_{\varepsilon}(\gg)\to \widehat{\Cl}_{\varepsilon}(\gg)
\]
which, on each homogeneous component $(i,\gamma)\in \ZZ\times\Gamma$, is given by the direct product over $r\ge 0$ and the direct sum of all $\alpha,\beta\in \Gamma$ with $\alpha+\beta=\gamma$ of the maps
\[
q\otimes q:
\bigwedge\nolimits_{\varepsilon}(\gg_{-})_{(i-r,\alpha)}\otimes
\bigwedge\nolimits_{\varepsilon}(\gg_{+})_{(r,\beta)}
\to
\Cl_{\varepsilon}(\gg_{-})_{(i-r,\alpha)}\otimes
\Cl_{\varepsilon}(\gg_{+})_{(r,\beta)}.
\]
\end{definition}

\begin{example}[Normal-ordered quantization of $\widehat{\bigwedge}_{\varepsilon}^{2}(\gg)$ and $\widehat{\bigwedge}_{\varepsilon}^{3}(\gg)$] \label{ex::quantization_map} \ 
\begin{enumerate} \item[a)] Let $x,y \in \gg$. Then
\[
 \widehat q(x\wedge y)=q(x\wedge y)-B(\pi_+x,y)+\varepsilon(x,y)B(\pi_+y,x).
\]
Indeed, write $x=x_{-}+x_{+}$ and $y=y_{-}+y_{+}$ with $x_{-},y_{-}\in\gg_{-}$ and $x_{+},y_{+}\in\gg_{+}$. It suffices to consider the cases $x,y\in\gg_{-}$, $x,y\in\gg_{+}$, and $x\in\gg_{-}$, $y\in\gg_{+}$. In the first two, the formula is immediate. In the third case, let $x=x_{-}$, $y=y_{+}$. Then
\begin{equation*}
\begin{aligned}
q(x\wedge y)
&=\tfrac12\bigl(x_-y_+ - \varepsilon(x_-,y_+)y_+x_-\bigr)=x_-y_+ - B(x_-,y_+)=\widehat q(x\wedge y) - B(x,\pi_+y)
\end{aligned}
\end{equation*}
since $\widehat q(x_-\wedge y_+)=x_-y_+$ and by the Clifford relation~\eqref{eq::Clifford_relation}.
\item[b)] Set $C(x,y)\coloneqq -B(\pi_+x,y)+\varepsilon(x,y)B(\pi_+y,x)$ for any homogeneous $x,y \in \gg$. For any $x,y,z \in \gg$, a direct calculation as in a) shows 
\[
\widehat q(x\wedge y\wedge z)
=
q(x\wedge y\wedge z)
+
C(x,y)z
-\varepsilon(y,z)C(x,z)y
+\varepsilon(x,y+z)C(y,z)x.
\]
\end{enumerate}
\end{example}
\begin{remark}
One can derive a general formula for $\widehat{q}$ in terms of a sum over contractions (using $C$ to contract two elements of $\mathfrak{g}$), as a straightforward generalization of Wick's theorem on normal ordering to the color case.  
\end{remark}

\begin{proposition}\label{prop::normal_ordered_quantization_contraction}
\begin{enumerate}
 \item[a)] The normal-ordered quantization map $\widehat{q}$ intertwines contractions.
 \item[b)] The normal-ordered quantization map $\widehat{q}$ is an isomorphism of $\ZZ\times \Gamma$-graded vector spaces. 
\end{enumerate}
\end{proposition}

\begin{proof}
 Note that contractions preserve $\bigwedge_{\varepsilon}(\gg_{\pm})$ and $\Cl_{\varepsilon}(\gg_{\pm})$. 
Hence, a) follows from Lemma~\ref{lemm::quantization_map_non_completion_and_contraction}, 
that is, $q \circ \iota_{x} = \iota_{x} \circ q$ for all $x \in \gg$, 
and b) is a direct consequence of Theorem~\ref{thm::quantization_map}.
\end{proof}

The Lie derivative $L_{T}$ for any $T \in \widehat{\mathfrak{so}}_{\varepsilon}(\gg)$ on $\widehat{\Cl}_{\varepsilon}(\gg)$ is naturally related to the quantization map. 
For any $T \in \widehat{\mathfrak{so}}_{\varepsilon}(\gg)$, define
\begin{equation}
\widehat{\gamma}(T) \coloneqq \widehat{q}\left(\uplambda(T)\right).
\end{equation}
Furthermore, let $[\cdot,\cdot]_{\widehat{\Cl}_{\varepsilon}}$ denote the natural Lie $\varepsilon$-bracket on $\widehat{\Cl}_{\varepsilon}(\gg)$, that is, $[v,w]_{\widehat{\Cl}_{\varepsilon}}=vw-\varepsilon(v,w)wv$ for any homogeneous $v,w \in \widehat{\Cl}_{\varepsilon}(\gg)$.

\begin{proposition} \label{prop::Lie_derivative_as_commutator}
For any homogeneous $T \in \widehat{\mathfrak{so}}_{\varepsilon}(\gg)$, one has
\[
-2L_{T} = [\widehat{\gamma}(T), \cdot]_{\widehat{\Cl}_{\varepsilon}}.
\]
\end{proposition}
\begin{proof}
 It suffices to prove the statement for homogeneous $T\in\mathfrak{so}_{\varepsilon}(\gg)$ and $v\in\Cl_{\varepsilon}(\gg)$ since both sides are defined componentwise on the completion. We first express $\widehat{\gamma}(T)$ with respect to a homogeneous basis $(e_{a})$ of $\gg$ and its $B$-dual basis $(e^{a})$. By definition,
\[
\begin{aligned}
q(\uplambda(T))
&=-\tfrac14\sum_{a}\Bigl(T(e^{a})e_{a}-\varepsilon(T+e^{a},e_{a})e_{a}T(e^{a})\Bigr)=-\tfrac12\sum_{a}\Bigl(T(e^{a})e_{a}-B(T(e^{a}),e_{a})\Bigr)\\&=-\tfrac12\sum_{a}T(e^{a})e_{a},
\end{aligned}
\]
since
$
\sum_{a}B(T(e^{a}),e_{a})=\tr_{\varepsilon}(T)=0
$
for $T\in\mathfrak{so}_{\varepsilon}(\gg)$. Using Example~\ref{ex::quantization_map}, one has
\[
\widehat{\gamma}(T)= -\tfrac{1}{2}\sum_{a}\left(T(e^{a})e_{a}-B(\pi_{+}T(e^{a}),e_{a})+\varepsilon(T(e^{a}),e_{a})B(\pi_{+}e_{a},T(e^{a}))\right).
\]
 Moreover, a direct calculation shows that the $\varepsilon$-commutator $[\widehat{\gamma}(T),\cdot]_{\widehat{\Cl}_{\varepsilon}}$ satisfies the Leibniz rule, that is,
\begin{equation}\label{eq::Leibniz_rule_commuatator_bracket}
[\widehat{\gamma}(T),vw]_{\widehat{\Cl}_{\varepsilon}}=[\widehat{\gamma}(T),v]_{\widehat{\Cl}_{\varepsilon}}w+\varepsilon(\widehat{\gamma}(T),v)v[\widehat{\gamma}(T),w]_{\widehat{\Cl}_{\varepsilon}}
\end{equation}
for homogeneous $v,w\in\Cl_{\varepsilon}(\gg)$. It is therefore enough to verify the identity on generators $x\in\gg$. The general case then follows from the Leibniz rule for $L_{T}$ and $[\widehat{\gamma}(T),\cdot]_{\widehat{\Cl}_{\varepsilon}}$. 

One has for any $x\in \gg$:
\[
[\widehat{\gamma}(T),x]_{\widehat{\Cl}_{\varepsilon}}=-\tfrac12\sum_{a}[T(e^{a})e_{a},x]_{\widehat{\Cl}_{\varepsilon}}=-\tfrac12\sum_{a}\bigl(T(e^{a})e_{a}x-\varepsilon(T,x)xT(e^{a})e_{a}\bigr).
\]
Using the Clifford relations,
\[
\begin{aligned}
T(e^{a})e_{a}x&=-\varepsilon(e_{a},x)T(e^{a})xe_{a}+2B(e_{a},x)T(e^{a}),\\
\varepsilon(T,x)xT(e^{a})e_{a}&=-\varepsilon(T,x)\varepsilon(x,T(e^{a}))T(e^{a})xe_{a}+2\varepsilon(T,x)B(x,T(e^{a}))e_{a}.
\end{aligned}
\]
Since $\vert T(e^{a})\vert =\vert T\vert +\vert e^{a}\vert $, one has
\[
\varepsilon(T,x)\varepsilon(x,T(e^{a}))=\varepsilon(x,e^{a})=\varepsilon(e^{a},x)^{-1}=\varepsilon(e_{a},x).
\]
Hence the terms involving $T(e^{a})xe_{a}$ cancel, and
\[
[\widehat{\gamma}(T),x]_{\widehat{\Cl}_{\varepsilon}}=-\sum_{a}B(e_{a},x)T(e^{a})+\varepsilon(T,x)\sum_{a}B(x,T(e^{a}))e_{a}.
\]
Since $T$ is $\varepsilon$-orthogonal, one has $\varepsilon(T,x)B(x,T(e^{a}))=-B(T(x),e^{a})$. Therefore
\[
[\widehat{\gamma}(T),x]_{\widehat{\Cl}_{\varepsilon}}=-\sum_{a}B(e_{a},x)T(e^{a})-\sum_{a}B(T(x),e^{a})e_{a}=-T(x)-T(x)=-2L_{T}(x),
\]
where we used Lemma~\ref{lemm::basis_description}.
\end{proof}

\subsection{Generalized Kac--Peterson Class}\label{subsec::Kac_Peterson_Class} 

We construct an $\varepsilon$-symmetric analogue of the Kac–Peterson cocycle for the $\ZZ$-graded quadratic color Lie algebra $(\gg,B)$. It arises naturally from the failure of the normally ordered infinite-dimensional analogues of the finite-dimensional identities and provides the correction required for the generalized Casimir operator. In the symmetrizable Kac–Moody setting, the corresponding cocycle was introduced by Kac and Peterson \cite{Kac_Peterson} in connection with the generalized Casimir operator and its $\rho$-correction.

In Section~\ref{subsec::conventions} we introduced the $\varepsilon$-trace $\etr$ on $\End(V)$ for finite-dimensional $V\in\cat$. For homogeneous $S,T,U\in\End(V)$, it satisfies $\etr(ST)=\varepsilon(S,T)\etr(TS)$, $\etr([S,T]_{\End}U)=\etr(S[T,U]_{\End})$, and $\etr(ST)=0$ unless $\vert S\vert +\vert T\vert = \ng$. If $(V,\langle\cdot,\cdot\rangle)$ is quadratic and $(e_{a})$ is a homogeneous basis of $V$ with $\langle\cdot,\cdot\rangle$-dual basis $(e^{a})$, then for every $T\in\End(V)$,
\begin{equation}\label{eq::epsilon_trace}
\etr(T)=\sum_{a}\varepsilon(e_{a})\langle T(e_{a}),e^{a}\rangle=\sum_{a}\langle e^{a},T(e_{a})\rangle.
\end{equation}

In general, for $x\in\gg$, one has $\ad_{x}\in\widehat{\mathfrak{so}}_{\varepsilon}(\gg)\subset\widehat{\End}(\gg)$ by Lemma~\ref{lemm::ad_in_completion}. The operator $\ad_x$ need not be finitary. However, $\pi_{-}\ad_{x}\pi_{+}$ has finite rank. This makes the following definition well-defined.

\begin{definition}
The \emph{Kac--Peterson cocycle} on $\gg$ is the bilinear form
\[
 \KP(x,y) \coloneqq \tfrac{1}{2}\etr(\ad_{x}\pi_{-}\ad_{y}\pi_{+}) 
 -\tfrac{1}{2}\varepsilon(x,y)\etr(\ad_{y}\pi_{-}\ad_{x}\pi_{+}), 
 \qquad x,y \in \gg.
\]
\end{definition}

\begin{example}\label{ex::KP_class}
Let $\gg$ be a complex finite-dimensional quadratic Lie superalgebra with non-degenerate invariant supersymmetric bilinear form $B$. This includes all basic classical Lie superalgebras, that is, the simple Lie algebras together with $A(m\vert n)$, $B(m\vert n)$, $C(n)$, $D(m\vert n)$, $F(4)$, $G(3)$, and $D(2,1;\alpha)$ for $\alpha \neq 0,-1$ \cite{Kac}. The loop superalgebra associated to $\gg$ is $L\gg\coloneqq\gg\otimes_{\CC}\CC[t,t^{-1}]$ with bracket
\begin{equation*}
[x\otimes t^m,y\otimes t^n]=[x,y]\otimes t^{m+n}.
\end{equation*}
It is naturally $\ZZ$-graded by declaring $(L\gg)_i=\{x\otimes t^i:x\in\gg\}$. One has
\begin{equation*}
(\ad_{x\otimes t^m}\pi_-\ad_{y\otimes t^n}\pi_+)(z\otimes t^k)=
\begin{cases}
[x,[y,z]]\otimes t^{m+n+k},&0<k\leq-n,\\
0,&\text{otherwise}.
\end{cases}
\end{equation*}
The supertrace of this expression is zero unless $m+n=0$, and hence
\begin{equation*}
\str(\ad_{x\otimes t^m}\pi_-\ad_{y\otimes t^n}\pi_+)
=\delta_{m+n,0}\max(-n,0)\str_{\gg}(\ad_x\circ\ad_y).
\end{equation*}
Analogously,
\begin{equation*}
\str(\ad_{y\otimes t^n}\pi_-\ad_{x\otimes t^m}\pi_+)
=\delta_{m+n,0}\max(-m,0)\str_{\gg}(\ad_y\circ\ad_x).
\end{equation*}
Thus
\begin{equation*}
\KP(x\otimes t^m,y\otimes t^n)
=\tfrac{m}{2}\delta_{m+n,0}\str_{\gg}(\ad_x\circ\ad_y).
\end{equation*}
Here $\str_{\gg}(\ad_x\circ\ad_y)$ is the super Killing form of $\gg$. It can vanish identically, for example for $A(n\vert n)$, $D(2,1;\alpha)$, and $D(n+1\vert n)$. In this case, $\KP\equiv0$.
\end{example}

By the properties of the $\varepsilon$-trace $\etr$, the Kac--Peterson cocycle has the following properties.

\begin{lemma} \label{lemm::properties_KP}\begin{enumerate}
 \item[a)] $\KP(x,y) = 0$ unless $\deg(x)+ \deg(y)=0$ and $\vert x\vert +\vert y\vert = \ng$ for any homogeneous $x,y \in \gg$.
 \item[b)]
The Kac--Peterson cocycle $\KP$ is a color Lie algebra $2$-cocycle:
\begin{enumerate}
 \item[(i)] $\KP$ is skew-$\varepsilon$-symmetric: 
 \[
 \KP(x,y) \;=\; -\varepsilon(x,y)\KP(y,x).
 \]
 \item[(ii)] For all homogeneous $x,y,z \in \gg$, one has
 \[
 \varepsilon(x,z)\KP([x,y],z)
 \;+\;
 \varepsilon(y,x)\KP([y,z],x)
 \;+\;
 \varepsilon(z,y)\KP([z,x],y)
 \;=\; 0.
 \]
\end{enumerate}
In particular, $\KP \in \widehat{\bigwedge}_{\varepsilon}^{2}(\gg^{\ast})_{(0,\ng)}$.
\item[c)] If $\ad_{x}, \ad_{y} \in \End(\gg)$, then 
\[
\KP(x,y) = \tfrac{1}{2}\etr([\ad_{x},\ad_{y}]_{\End}\pi_{+}).
\]
\end{enumerate}
\end{lemma}

\begin{proof}
 We prove c). a) and b) follow directly from the defining formula and the $\varepsilon$-trace identities. If $\ad_{x}, \ad_{y} \in \End(\gg)$, then they have well-defined traces and using $\pi_{-} = \id_{\gg}-\pi_{+}$, we obtain
 \begin{equation*}
 \begin{aligned}
\KP(x,y)
&=\tfrac{1}{2}\etr(\ad_x\pi_{-}\ad_y\pi_{+})
-\tfrac{1}{2}\varepsilon(x,y)\etr(\ad_y\pi_{-}\ad_x\pi_{+})\\
&=\tfrac{1}{2}\tr_{\varepsilon}(\ad_x(\operatorname{id}-\pi_{+})\ad_y\pi_{+})
-\tfrac{1}{2}\varepsilon(x,y)\tr_{\varepsilon}(\ad_y(\operatorname{id}-\pi_{+})\ad_x\pi_{+})\\
&=\tfrac{1}{2}\tr_{\varepsilon}(\ad_x\ad_y\pi_{+})
-\tfrac{1}{2}\tr_{\varepsilon}(\ad_x\pi_{+}\ad_y\pi_{+})\\
&\qquad
-\tfrac{1}{2}\varepsilon(x,y)\tr_{\varepsilon}(\ad_y\ad_x\pi_{+})
+\tfrac{1}{2}\varepsilon(x,y)\tr_{\varepsilon}(\ad_y\pi_{+}\ad_x\pi_{+})\\
&=\tfrac{1}{2}\tr_{\varepsilon}\bigl((\ad_x\ad_y-\varepsilon(x,y)\ad_y\ad_x)\pi_{+}\bigr)\\
&\qquad
-\tfrac{1}{2}\Bigl(\tr_{\varepsilon}(\ad_x\pi_{+}\ad_y\pi_{+})
-\varepsilon(x,y)\tr_{\varepsilon}(\ad_y\pi_{+}\ad_x\pi_{+})\Bigr)\\
&=\tfrac{1}{2}\tr_{\varepsilon}\bigl([\ad_x,\ad_y]_{\End}\pi_{+}\bigr),
\end{aligned}
 \end{equation*}
where we used the $\varepsilon$-cyclicity of $\tr_{\varepsilon}$ in the last equality so that $\tr_{\varepsilon}(\ad_x\pi_{+}\ad_y\pi_{+})
=
\varepsilon(x,y)\tr_{\varepsilon}(\ad_y\pi_{+}\ad_x\pi_{+})$.
\end{proof}

In particular, $\KP$ defines a class in the color Lie algebra cohomology $\operatorname{H}^2(\gg;\CC)$, called the \emph{Kac--Peterson class}. For the definition of $\operatorname{H}^{\bullet}(\gg;\CC)$, we refer to \cite{Scheunert_cohomology}. If $\gg$ is the loop superalgebra of a basic classical Lie superalgebra with non-vanishing super Killing form, then the class of $\KP$ spans $\operatorname{H}^2(\gg;\CC)$ \cite{Iohara}. In this case, $\KP$ defines a central extension of $\gg$, as follows immediately from Lemma~\ref{lemm::properties_KP}.

The Kac–Peterson class vanishes exactly if there exists some $\rho \in (\gg^{\ast})_{(0,\ng)}$ such that $\KP = \operatorname{d}\!\rho$, where
\begin{equation}
 \operatorname{d}\!\rho(x,y) \coloneqq \rho([x,y]) = B(\rho^{\sharp},[x,y]), \qquad x,y \in \gg.
\end{equation}
In Section~\ref{subsubsec::epsilon_exterior_and_so} we identified $\widehat{\bigwedge}_{\varepsilon}^{2}(\gg)$ and $\widehat{\mathfrak{so}}_{\varepsilon}(\gg)$ as $\ZZ\times \Gamma$-graded vector spaces. As $\KP \in \widehat{\bigwedge}_{\varepsilon}^{2}(\gg^{\ast})$, the Kac–Peterson $2$-cocycle corresponds to an element $\Psi_{\operatorname{KP}} \in \widehat{\mathfrak{so}}_{\varepsilon}(\gg)$ via the relation
\begin{equation} \label{eq::relations_psi_Psi}
 \KP(x,y) = B(\Psi_{\operatorname{KP}}(x),y), \qquad x,y\in \gg.
\end{equation}

\begin{lemma} \label{lemm::properties_Psi}
\begin{enumerate}[label=\alph*)]
 \item $\Psi_{\operatorname{KP}}$ preserves the $\ZZ\times\Gamma$-grading of $\gg$.
 \item For $\rho \in (\gg^{\ast})_{(0, \ng)}$ one has $\KP = \operatorname{d}\!\rho$ if and only if 
 $
 \Psi_{\operatorname{KP}} = [\rho^{\sharp}, \cdot].
 $
\end{enumerate}
\end{lemma}

\begin{proof}
a) follows from Lemma~\ref{lemm::properties_KP}.
For b), let $\rho \in (\gg^{\ast})_{(0,\ng)}$. Assume $\KP = \mathrm{d}\rho$, \emph{i.e.}, $\KP$ is a coboundary. Then
\[
 \KP(x,y) = \rho([x,y]) = B(\rho^{\sharp},[x,y]) = B\bigl([\rho^{\sharp},x],y\bigr),
\]
by~\eqref{eq::musical_isomorphisms} and invariance of $B$. Conversely, if $\Psi_{\operatorname{KP}}$ is inner, the same identity shows that $\KP$ is a coboundary. 
\end{proof}

The musical isomorphism $\sharp$ in~\eqref{eq::musical_isomorphisms} allows us to identify $\widehat{\bigwedge}_{\varepsilon}^{2}(\gg^{\ast})$ with $\widehat{\bigwedge}_{\varepsilon}^{2}(\gg)$. We denote the image of $\KP$ under $\sharp$ by $\KPsharp$, that is, $\KPsharp$ is the unique element in $\widehat{\bigwedge}_{\varepsilon}^{2}(\gg)$ defined by the property
\begin{equation} \label{eq::definition_psi_sharp}
\KP(x,y)=B(\iota_{x}\KPsharp,y), \qquad x,y \in \gg.
\end{equation}
\begin{lemma} \label{lemm::komische_Beziehungen_eins}
 The following identity holds for all homogeneous $x \in \gg$:
 \[
\Psi_{\operatorname{KP}}(x) = \iota_{x}\widehat{q}(\KPsharp).
\]
\end{lemma}
\begin{proof}
 By~\eqref{eq::relations_psi_Psi} and~\eqref{eq::definition_psi_sharp}, one has
\begin{equation*}
B(\Psi_{\operatorname{KP}}(x),y)=\KP(x,y)=B(\iota_{x}\psi_{\operatorname{KP}}^{\sharp},y)
\end{equation*}
for all $y\in\gg$. Since $B$ is nondegenerate, it follows that $\Psi_{\operatorname{KP}}(x)=\iota_{x}\psi_{\operatorname{KP}}^{\sharp}.$
Applying $\widehat q$ and using that $\widehat q$ restricts to the identity on $\gg$, we obtain
$$
\Psi_{\operatorname{KP}}(x)=\widehat q(\iota_{x}\psi_{\operatorname{KP}}^{\sharp}).
$$
Since $\widehat q$ intertwines contractions,
$
\widehat q(\iota_{x}\psi_{\operatorname{KP}}^{\sharp})=\iota_{x}\widehat q(\psi_{\operatorname{KP}}^{\sharp}),
$
the claim follows.
\end{proof}

We denote again by $\KP$ the bilinear form on $\widehat{\mathfrak{so}}_{\varepsilon}(\gg)$ defined by
\begin{equation}
\KP(S,T)\coloneqq\tfrac{1}{2}\tr_{\varepsilon}(S\pi_{-}T\pi_{+})-\tfrac{1}{2}\varepsilon(S,T)\tr_{\varepsilon}(T\pi_{-}S\pi_{+})
\end{equation}
for homogeneous $S,T\in\widehat{\mathfrak{so}}_{\varepsilon}(\gg)$. For homogeneous $x,y\in\gg$, one has
$
\KP(\ad_{x},\ad_{y})=\KP(x,y),
$
so this extends the Kac--Peterson cocycle of $\gg$. Moreover, $\KP$ measures the failure of $\widehat q$ to be $\widehat{\mathfrak{so}}_{\varepsilon}(\gg)$-equivariant.

\begin{proposition} \label{prop::commutator_and_quantization} \begin{enumerate}
\item[a)] If $\nu \in \widehat{\bigwedge}_{\varepsilon}^{2}(\gg)$ and $T \in \widehat{\mathfrak{so}}_{\varepsilon}(\gg)$, then
\[
 L_{T}\widehat{q}(\nu) = \widehat{q}(L_{T}\nu) -2 \KP(T,\upmu(\nu)).
\]
\item[b)] For any $T_{1},T_{2}\in \widehat{\mathfrak{so}}_{\varepsilon}(\gg)$, with $\widehat{\gamma}(\cdot)\coloneqq \widehat{q}(\uplambda(\cdot))$, one has
\[
-\tfrac{1}{2}[\widehat{\gamma}(T_{1}),\widehat{\gamma}(T_{2})]_{\widehat{\Cl}_{\varepsilon}} = \widehat{\gamma}([T_{1},T_{2}]_{\widehat{\mathfrak{so}}_{\varepsilon}(\gg)}) -2 \KP(T_{1},T_{2}).
\]
\end{enumerate}
\end{proposition}

\begin{proof}
a) It suffices to treat $T\in\mathfrak{so}_{\varepsilon}(\gg)$ and $\nu=x\wedge y\in\bigwedge^{2}_{\varepsilon}(\gg)$ with $x,y\in\gg$. By Example~\ref{ex::quantization_map},
\[
\widehat q(x\wedge y)=q(x\wedge y)-B(\pi_+x,y)+\varepsilon(x,y)B(\pi_+y,x).
\]
On the other hand, $\upmu(\nu)(z)=\varepsilon(y,z)B(x,z)y-B(y,z)x$ and $\pi_{+}\upmu(\nu)$ is the sum of two rank-one operators. One gets
\[
\tr_{\varepsilon}(\pi_{+}\upmu(\nu))
=-B(\pi_{+}x,y)+\varepsilon(y,x)B(\pi_{+}y,x).
\]
Thus $\widehat{q}(\nu)=q(\nu)+\tr_{\varepsilon}(\pi_{+}\upmu(\nu))$. By Lemma~\ref{lemm::properties_KP},
\[
L_{T}\widehat{q}(\nu)-\widehat{q}(L_{T}\nu)
=-\tr_{\varepsilon}(\pi_{+}\upmu({L_{T}\nu}))
=-\tr_{\varepsilon}(\pi_{+}[T,\upmu(\nu)]_{\mathfrak{so}_{\varepsilon}})
=-2\KP(T,\upmu(\nu)).
\]

b) It suffices to treat $T_{1}, T_{2} \in \mathfrak{so}_{\varepsilon}(\gg)$. Recall that $\upmu(\uplambda(T_{i})) = T_{i}$ for $i=1,2$. 
Using a), Lemma~\ref{lemm::Lie_derivative_and_lambd}, and Proposition~\ref{prop::Lie_derivative_as_commutator}, we obtain
\begin{equation}
\begin{split}
[\widehat{\gamma}(T_{1}), \widehat{\gamma}(T_{2})]_{\widehat{\Cl}_{\varepsilon}} 
&= -2L_{T_{1}}\widehat{\gamma}(T_{2}) 
= -2L_{T_{1}}\widehat{q}(\uplambda(T_{2}))
= -2\widehat{q}(L_{T_{1}}\uplambda(T_{2})) +4 \KP(T_{1}, \upmu(\uplambda(T_{2}))) 
\\ &= -2\widehat{q}(\uplambda([T_{1}, T_{2}]_{\mathfrak{so}_{\varepsilon}})) +4 \KP(T_{1}, T_{2}). \qedhere
\end{split}
\end{equation}
\end{proof}

\subsection{Completion of \texorpdfstring{$\varepsilon$}{}-Symmetric and Universal Enveloping Algebras of \texorpdfstring{$\gg$}{}} \label{subsec::completion_symmetric_universal_superalgebra} In this section, we introduce the completions of the $\varepsilon$-symmetric algebra and the universal enveloping algebra of $\gg$ and relate both by the PBW $\varepsilon$-symmetrization.

\subsubsection{\texorpdfstring{$\varepsilon$}{}-Symmetric Algebras and Completion}
Let $V \in \mathbf{Vec}_{\Gamma}^{\varepsilon}$. Assume in addition that $V$ is $\ZZ$-graded, that is, $V=\bigoplus_{n\in\ZZ}V^n$ with $V^n\in \mathbf{Vec}_{\Gamma}^{\varepsilon}$. Let $I_{S}$ denote the $(\ZZ\times \Gamma)$-graded homogeneous ideal of $\mathrm{T}(V)$ generated by
\begin{equation}
v\otimes w-\varepsilon(v,w)w\otimes v,\qquad v,w\in V\ \text{homogeneous}.
\end{equation}
The \emph{$\varepsilon$-symmetric algebra} of $V$ is the quotient
$
S_{\varepsilon}(V)\coloneqq \mathrm{T}(V)/I_{S}.
$
Equivalently, it is generated by $1_{S_{\varepsilon}(V)}$ and $V$, subject to
$
vw-\varepsilon(v,w)wv=0$ for homogeneous $v,w\in V.$ Since $I_{S}$ is $\ZZ\times\Gamma$-homogeneous, $S_{\varepsilon}(V)$ inherits the $\ZZ\times\Gamma$-grading of $\mathrm{T}(V)$ induced by $V$, and is a $\ZZ$-graded color algebra. Moreover, $S_{\varepsilon}(V)$ carries the canonical $\varepsilon$-symmetric power grading
$S_{\varepsilon}(V)=\bigoplus_{k\geq 0}S^{k}_{\varepsilon}(V)$,
induced by the tensor length grading
$\mathrm{T}(V)=\bigoplus_{k\geq 0}\mathrm{T}^{k}(V)$
of the tensor algebra. It is compatible with the $\ZZ\times \Gamma$-grading induced by $V$. In what follows, we view $S_{\varepsilon}(V)$ as a $\ZZ\times \ZZ$-graded color algebra.

We are interested in the case when $V$ is the $\ZZ$-graded color Lie algebra $\gg$. Recall that $\gg^{\ast}$ denotes the restricted dual and $\gg=\gg_{+}\oplus\gg_{-}$. The natural pairing between $S_{\varepsilon}(\gg)$ and $S_{\varepsilon}(\gg^{\ast})$ identifies $S_{\varepsilon}(\gg)_{(i,\gamma)}$ with a subspace of the linear maps $S_{\varepsilon}(\gg^{\ast})_{(-i,-\gamma)}\to\CC$. Using $S_{\varepsilon}(\gg)\cong S_{\varepsilon}(\gg_{-})\otimes S_{\varepsilon}(\gg_{+})$, we define
\begin{equation}
\widehat{S}_{\varepsilon}(\gg)_{(i,\gamma)}
\coloneqq
\prod_{r\geq 0}
\ \bigoplus_{\substack{\alpha,\beta\in\Gamma\\ \alpha+\beta=\gamma}}
S_{\varepsilon}(\gg_{-})_{(i-r,\alpha)}
\otimes
S_{\varepsilon}(\gg_{+})_{(r,\beta)}.
\end{equation}
The completion $\widehat{S}_{\varepsilon}(\gg)$ of $S_{\varepsilon}(\gg)$ is the direct sum of all $\widehat{S}_{\varepsilon}(\gg)_{(i,\gamma)}$. The multiplication on $S_{\varepsilon}(\gg)$ extends to $\widehat{S}_{\varepsilon}(\gg)$, which is thus a $\ZZ$-graded color algebra.

For each $k\geq 0$, the homogeneous component $S^{k}_{\varepsilon}(\gg)$ admits an analogous completion $\widehat{S}_{\varepsilon}^{k}(\gg)\subseteq\widehat{S}_{\varepsilon}(\gg)$. Thus $\widehat{S}_{\varepsilon}(\gg)$ carries, in addition, a compatible $\ZZ$-grading induced by tensor length.

\subsubsection{\texorpdfstring{$\UE_{\varepsilon}(\gg)$}{} and Completion} Let $I_{U}$ be the two-sided homogeneous ideal of $\mathrm{T}(\gg)$ generated by
\begin{equation} x\otimes y-\varepsilon(x,y)y\otimes x-[x,y],
\qquad x,y\in \gg\ \text{homogeneous}.\end{equation}
The universal enveloping algebra is the quotient $\UE_{\varepsilon}(\gg) \coloneqq \mathrm{T}(\gg)/I_{U}.
$ Equivalently, $\UE_{\varepsilon}(\gg)$ is the associative color algebra generated by $\gg$ subject to the relations
$xy-\varepsilon(x,y)yx=[x,y]$ for all homogeneous $x,y\in \gg$. 

In \cite{Scheunert1979}, Scheunert proves a Poincar\'e--Birkhoff--Witt theorem for $\UE_{\varepsilon}(\gg)$. The tensor-degree filtration on $\mathrm{T}(\gg)$ defined by
$
F^{k}\mathrm{T}(\gg)\coloneqq\bigoplus_{j\leq k}\mathrm{T}^{j}(\gg)$, induces a filtration on $\UE_{\varepsilon}(\gg)$. Let $\operatorname{gr}\UE_{\varepsilon}(\gg)$ denote the associated graded color algebra. Then there is a canonical isomorphism of graded color algebras~\cite[Theorem~1]{Scheunert1979}
\begin{equation}
\operatorname{gr}\UE_{\varepsilon}(\gg)\cong S_{\varepsilon}(\gg).
\end{equation}
Since $I_U$ is $\ZZ\times\Gamma$-graded, this is in fact an isomorphism of $\ZZ$-graded color algebras. Here the $\ZZ$-grading refers to the one induced by $\gg$.

Let $\UE_{\varepsilon}(\gg_{\pm})$ denote the universal enveloping algebra of the $\ZZ$-graded color subalgebras $\gg_{\pm}$ of $\gg$. By the Poincaré–Birkhoff–Witt theorem, the multiplication map induces a $\ZZ \times \Gamma$-graded vector space isomorphism
$
\UE_{\varepsilon}(\gg)\;\cong\;\UE_{\varepsilon}(\gg_{-})\otimes \UE_{\varepsilon}(\gg_{+}).
$ We define a completion $\widehat{\UE}_{\varepsilon}(\gg)$ by declaring its $(i,\gamma)$-graded piece to be
\begin{equation}
\widehat{\UE}_{\varepsilon}(\gg)_{(i,\gamma)}
\coloneqq
\prod_{r\ge 0}\;
\bigoplus_{\substack{\alpha,\beta\in\Gamma\\ \alpha+\beta=\gamma}}
\UE_{\varepsilon}(\gg_{-})_{(i-r,\alpha)}
\otimes
\UE_{\varepsilon}(\gg_{+})_{(r,\beta)},
\end{equation}
and setting
$
 \widehat{\UE}_{\varepsilon}(\gg)\coloneqq \bigoplus_{(i, \gamma)\in\ZZ\times\Gamma}\widehat{\UE}_{\varepsilon}(\gg)_{(i, \gamma)}.
$
The multiplication on $\UE_{\varepsilon}(\gg)$ extends to $\widehat{\UE}_{\varepsilon}(\gg)$, making $\widehat{\UE}_{\varepsilon}(\gg)$ a $\ZZ$-graded color algebra.

\subsubsection{PBW \texorpdfstring{$\varepsilon$}{}-Symmetrization}
The natural inclusion $\gg \hookrightarrow \UE_{\varepsilon}(\gg)$ is a homomorphism in $\textbf{Vec}^{\varepsilon}_{\Gamma}$ of $\ZZ$-graded vector spaces. It extends canonically to a homomorphism of $\ZZ$-graded vector spaces $Q: S_{\varepsilon}(\gg) \to \UE_{\varepsilon}(\gg)$ by \emph{PBW $\varepsilon$-symmetrization}, that is, 
\begin{equation}
 Q(x_{1}\ldots x_{k}) \coloneqq \tfrac{1}{k!}\sum_{\sigma \in S_{k}} p'(\sigma;x_{1}, \ldots, x_{k}) x_{\sigma(1)}\ldots x_{\sigma(k)},
\end{equation}
where $p'(\sigma; x_{1},\ldots, x_{k})
$ is given by (\emph{cf.}~Section~\ref{subsec::conventions})
\begin{equation}
p'(\sigma; x_{1},\ldots, x_{k}) = \prod_{1 \leq i < j \leq k, \ \sigma^{-1}(i) > \sigma^{-1}(j)} \varepsilon(x_{i},x_{j}).
\end{equation}
This is the color analogue of the PBW symmetrization map, and it maps $S_{\varepsilon}(\gg_{\pm})$ to $\UE_{\varepsilon}(\gg_{\pm})$. 

\begin{theorem}\label{thm::PBW_supersymmetrization} The map $Q : S_{\varepsilon}(\gg) \to \UE_{\varepsilon}(\gg)$ is an isomorphism in $\textbf{Vec}^{\varepsilon}_{\Gamma}$ of $\ZZ$-graded vector spaces. 
\end{theorem}

\begin{proof}
 By definition, for every $k\geq 0$, one has
\begin{equation*}
Q\bigl(S_{\varepsilon}^{k}(\gg)\bigr)\subseteq F^{k}\UE_{\varepsilon}(\gg),
\end{equation*}
so $Q$ is filtration-preserving. Modulo $F^{k-1}\UE_{\varepsilon}(\gg)$, the defining relation
$
xy-\varepsilon(x,y)yx=[x,y]
$
for $x,y\in \gg$ homogeneous 
reduces to
$
xy=\varepsilon(x,y)yx,
$
since $[x,y]$ has lower filtration degree. Hence, in the associated graded color algebra, all summands in the symmetrization formula coincide. Indeed, every ordered product
$x_{\sigma(1)}\cdots x_{\sigma(k)}$
becomes, in $\operatorname{gr}\UE_{\varepsilon}(\gg)$, equal to
$p'(\sigma;x_{1},\dots,x_{k})^{-1}x_{1}\cdots x_{k}$. Therefore in the color symmetrization formula each summand contributes the same class:
\[p'(\sigma;x_{1},\dots,x_{k})x_{\sigma(1)}\cdots x_{\sigma(k)}
\equiv x_{1}\cdots x_{k}
\quad \text{in }F^{k}\UE_{\varepsilon}(\gg)/F^{k-1}\UE_{\varepsilon}(\gg).\]
It follows that the induced map
$
S_{\varepsilon}(\gg)\to\operatorname{gr}\UE_{\varepsilon}(\gg)
$
is the identity under the PBW isomorphism
$
\operatorname{gr}\UE_{\varepsilon}(\gg)\cong S_{\varepsilon}(\gg).$ Hence $Q$ is a vector space isomorphism since it is filtration preserving and the filtration is exhaustive.
\end{proof}

The PBW $\varepsilon$-symmetrization does not admit a direct extension to completions. We therefore introduce the following normal-ordered analogue.

\begin{definition}\label{def::normal_ordered_color_PBW_symmetrization}
 The normal-ordered PBW $\varepsilon$-symmetrization map
$
\widehat Q:\widehat S_{\varepsilon}(\gg)\to \widehat{\UE}_{\varepsilon}(\gg)
$
is the $\ZZ\times \Gamma$-graded linear map which, on each homogeneous component $(i,\gamma)\in \ZZ\times \Gamma$, is induced by the direct product over $r\ge 0$ and the direct sum of all $\alpha,\beta\in \Gamma$ with $\alpha+\beta=\gamma$ of the maps
\[
Q\otimes Q:
S_{\varepsilon}(\gg_{-})_{(i-r,\alpha)}\otimes
S_{\varepsilon}(\gg_{+})_{(r,\beta)}
\to
\UE_{\varepsilon}(\gg_{-})_{(i-r,\alpha)}\otimes
\UE_{\varepsilon}(\gg_{+})_{(r,\beta)}.
\]
\end{definition}

\begin{example} \label{ex::S2}
 Let $p = xy \in S^{2}_{\varepsilon}(\gg)$ with $x, y \in \gg$. One has
\[
\widehat{Q}(p) = Q(p) + \tfrac{1}{2}\left([x, \pi_{+}y] + \varepsilon(x,y)[y, \pi_{+}x]\right).
\]
The verification is immediate in each of the cases where $x, y \in \gg_{\pm}$, or $x \in \gg_{+}$ and $y \in \gg_{-}$.
\end{example}

From the construction of the normal-ordered PBW $\varepsilon$-symmetrization, together with Theorem~\ref{thm::PBW_supersymmetrization}, one obtains the following proposition.

\begin{proposition}
 $\widehat{Q}$ is an isomorphism in $\textbf{Vec}^{\varepsilon}_{\Gamma}$ of $\ZZ$-graded vector spaces.
\end{proposition}

We rewrite Example~\ref{ex::S2} in an appropriate way. For this purpose, we introduce some notation. Fix a homogeneous basis $(e_{a})$ of $\gg$ with $B$-dual basis $(e^{a})$, and identify $\gg$ and $\gg^{\ast}$ using the musical isomorphisms~\eqref{eq::musical_isomorphisms}. Recall the identification of $\End(\gg)$ with $\gg \otimes \gg$ given in~\eqref{eq::identification_End(g)} 
\begin{equation}
 \End(\gg) \cong \gg \otimes \gg^{\ast} \cong \gg \otimes \gg, \qquad T \mapsto \sum_{a}T(e_{a})\otimes e_{a}^{\ast} \mapsto \sum_{a}\varepsilon(e_{a})T(e_{a}) \otimes e^{a}.
\end{equation}
In what follows, we write $\mathcal{I} : \End(\gg) \to \gg\otimes \gg$ for this $\ZZ$-graded isomorphism in $\textbf{Vec}^{\varepsilon}_{\Gamma}$. Its inverse is on homogeneous $x,y,z\in \gg$ given by 
\begin{equation}
 \mathcal{I}^{-1}(x\otimes y)(z) =\varepsilon(y)B(z,y)x
\end{equation}
and then extended linearly.

We view $\End(\gg)$ as a $\gg$-module via $\ad$, and $\gg\otimes\gg$ as a $\gg$-module via
\begin{equation}
L_{X}(x\otimes y)\coloneqq [X,x]\otimes y+\varepsilon(X,x)x\otimes [X,y],\qquad X,x,y\in\gg.
\end{equation}
Then $\mathcal I$ is an isomorphism of $\gg$-modules, with inverse $\mathcal I^{-1}$.

\begin{lemma}\label{lemm::equivariance_I} One has for any homogeneous $X \in \gg$, $T \in \End(\gg)$ and $\eta \in \gg\otimes \gg$
\[
\mathcal{I}([\ad_{X},T]_{\End})=L_{X}\mathcal{I}(T), \qquad \mathcal{I}^{-1}(L_{X}\eta)=[\ad_{X},\mathcal{I}^{-1}(\eta)]_{\End}.
\]
\end{lemma}

\begin{proof}
 We prove only that $\mathcal I([\ad_{X},T]_{\End})=L_{X}\mathcal I(T)$, since the proof of the other identity is analogous. One computes
\begin{equation*}
\begin{aligned}
\mathcal I([\ad_{X},T]_{\End})
&=\sum_{a}\varepsilon(e_{a})\Bigl([X,T(e_{a})]-\varepsilon(X,T)T([X,e_{a}])\Bigr)\otimes e^{a},\\
L_{X}(\mathcal I(T))
&=\sum_{a}\varepsilon(e_{a})\Bigl([X,T(e_{a})]\otimes e^{a}+\varepsilon(X,T(e_{a}))T(e_{a})\otimes [X,e^{a}]\Bigr).
\end{aligned}
\end{equation*}
Thus it remains to show that
\begin{equation*}
-\sum_{a}\varepsilon(e_{a})\varepsilon(X,T)T([X,e_{a}])\otimes e^{a}=\sum_{a}\varepsilon(e_{a})\varepsilon(X,T(e_{a}))T(e_{a})\otimes [X,e^{a}].
\end{equation*}
Multiplying by $\varepsilon(T,X)$, this is equivalent to
\begin{equation*}
-\sum_{a}\varepsilon(e_{a})T([X,e_{a}])\otimes e^{a}=\sum_{a}\varepsilon(e_{a})\varepsilon(X,e_{a})T(e_{a})\otimes [X,e^{a}].
\end{equation*}
Now the canonical tensor $\sum_{a}\varepsilon(e_{a})e_{a}\otimes e^{a}$ is $\gg$-invariant by Lemma~\ref{lemm::basis_description}, that is,
\begin{equation*}
\sum_{a}\varepsilon(e_{a})\Bigl([X,e_{a}]\otimes e^{a}+\varepsilon(X,e_{a})e_{a}\otimes [X,e^{a}]\Bigr)=0.
\end{equation*}
Applying $T\otimes\id_{\gg}$ yields the required identity.
\end{proof}

We define the braiding $\br : \End(\gg) \to \gg$ to be the linear map given by the identification $\End(\gg) \cong \gg \otimes \gg$ followed by the $\varepsilon$-bracket, that is, 
\begin{equation} \br(T) = \sum_{a}\varepsilon(e_{a})[T(e_{a}),e^{a}].
\end{equation}

\begin{lemma} \label{lemm::br_and_str}
 The following assertions hold:
 \begin{enumerate} \item[a)] For any homogeneous $T \in \End(\gg)$ and $x\in \gg$, one has 
 \[
 B(\br(T),x) = \varepsilon(x,T)\tr_{\varepsilon}(\ad_{x}\circ T) = \tr_{\varepsilon}(T\circ \ad_{x}).
 \]
 \item[b)] The map $\br : \End(\gg) \to \gg$ is $\gg$-equivariant, that is, 
 \[
 L_{x}\br(T) = 
 \br([\ad_{x},T]_{\End}), \qquad \forall x \in \gg, \ T \in \End(\gg),
 \]
 where we equip $\End(\gg)$ with the natural $\varepsilon$-bracket. 
 \end{enumerate}
\end{lemma}
\begin{proof} a)
 In a homogeneous basis $(e_{a})$ of $\gg$ with $B$-dual basis $(e^{a})$, we identify $T$ with $\sum_{a} \varepsilon(e_{a})T(e_{a})\otimes e^{a}$, and $\br(T) = \sum_{a}\varepsilon(e_{a})[T(e_{a}),e^{a}]$. Note that $T$ is finitary. Let $x \in \gg$ be homogeneous. Then one computes using invariance of $B$, $\vert e_{a}\vert =-\vert e^{a}\vert$, and skew-$\varepsilon$-symmetry of $[\cdot,\cdot]$:
 \[
 \begin{aligned}
 B(\br(T),x) &= \sum_{a}\varepsilon(e_{a})B([T(e_{a}),e^{a}],x) \\&= \sum_{a}\varepsilon(e_{a})\varepsilon(T+e_{a}+e^{a},x)B(x,[T(e_{a}),e^{a}]) \\ &= \varepsilon(x)\sum_{a}\varepsilon(e_{a})B((\ad_{x}\circ T)(e_{a}),e^{a}) \\ &= \varepsilon(x)\tr_{\varepsilon}(\ad_{x}\circ T),
 \end{aligned}
 \]
 where the last equality follows by~\eqref{eq::epsilon_trace}. Moreover, since $B(\br(T),x)$ is zero unless $\vert T\vert + \vert x \vert = \ng$, it follows that $\varepsilon(x)=\varepsilon(x,T)=\varepsilon(T,x)$ by~\eqref{eq::commutation_factor}. This proves a) together with the $\varepsilon$-symmetry of $B$.

b) Since $\operatorname{br}=[\cdot,\cdot]\circ\mathcal I$, and both $[\cdot,\cdot]\colon\gg\otimes\gg\to\gg$ and $\mathcal I\colon\End(\gg)\to\gg\otimes\gg$ are $\gg$-equivariant for the diagonal adjoint action on $\gg\otimes\gg$, the map $\operatorname{br}$ is $\gg$-equivariant. Hence
\begin{equation}
L_{x}\operatorname{br}(T)=\operatorname{br}([\ad_{x},T]_{\End}).\qedhere
\end{equation}
\end{proof}

We define a natural linear map $A : S^{2}_{\varepsilon}(\gg) \to \End(\gg)$ by
\begin{equation} \label{eq::definition_A}A_{xy}\coloneqq\mathcal{I}^{-1}(x\otimes y+\varepsilon(x,y)y\otimes x), \qquad A_{xy}(z)=\varepsilon(y)B(z,y)x+\varepsilon(x)\varepsilon(x,y)B(z,x)y
\end{equation}
for homogeneous $x,y\in \gg$. This is well-defined since it is compatible with the defining relation $xy=\varepsilon(x,y)yx$ in $S^2_{\varepsilon}(\gg)$. It extends degreewise to a linear map $A : \widehat{S}_{\varepsilon}^{2}(\gg) \to \widehat{\End}(\gg)$, denoted by the same symbol.

\begin{lemma} \label{lemm::bracket_and_br} The following assertions hold: \begin{enumerate}
 \item[a)] The map $A : \widehat{S}_{\varepsilon}^{2}(\gg) \to \widehat{\End}(\gg)$ satisfies
 \[
 A_{L_{X}p} \;=\;[\ad_X,A_{p}]_{\widehat{\End}}\;=\;\ad_X\circ A_{p}\;-\;\varepsilon(X,A_{p})A_{p}\circ\ad_X
 \]
 for all homogeneous $X \in \gg$ and homogeneous $p \in \widehat{S}_{\varepsilon}^{2}(\gg)$.
 \item[b)] For all homogeneous $p=xy \in S^{2}_{\varepsilon}(\gg)$, one has
\begin{equation*}
 \operatorname{br}(\pi_{+}A_{xy})=[\pi_{+}x,y]+\varepsilon(x,y)[\pi_{+}y,x].
\end{equation*}
\end{enumerate}
\end{lemma}

\begin{proof} It suffices to prove the statement for $p=xy\in S_{\varepsilon}^{2}(\gg)$ homogeneous, since all maps are defined componentwise on the product decomposition, so the algebraic calculation applies separately to every component. For these, one has
$L_Xp=[X,x]y+\varepsilon(X,x)x[X,y]$, and consequently $A_{L_Xp}
=
A_{[X,x]y}+\varepsilon(X,x)A_{x[X,y]}.$ Using the definition of $A$, we obtain by $\gg$-equivariance of $\mathcal{I}^{-1}$:
\[
\begin{aligned}
A_{L_Xp}
&=
\mathcal I^{-1}\bigl([X,x]\otimes y+\varepsilon(X,x)x\otimes [X,y]\bigr)
+\varepsilon(x,y)\mathcal I^{-1}\bigl([X,y]\otimes x+\varepsilon(X,y)y\otimes [X,x]\bigr) \\ &=\mathcal{I}^{-1}(L_{X}(x\otimes y))+\varepsilon(x,y)\mathcal{I}^{-1}(L_{X}(y\otimes x)) \\ &= \mathcal{I}^{-1}(L_{X}(x\otimes y +\varepsilon(x,y) y\otimes x)) \\ &= [\ad_{X}, \mathcal{I}^{-1}((x\otimes y +\varepsilon(x,y) y\otimes x))]_{\End} \\ &= [\ad_{X},A_{xy}]_{\End}.
\end{aligned}
\]

For b), we have by definition
\begin{equation}
\mathcal I(\pi_{+}A_{xy})=\pi_{+}x\otimes y+\varepsilon(x,y)\pi_{+}y\otimes x,
\end{equation}
and applying $\operatorname{br}$ yields the statement.
\end{proof}

\begin{proposition} \label{prop::Lie_derivative_and_PBW_quantization} For any homogeneous $x \in \gg$ and any $p \in \widehat{S}_{\varepsilon}^{2}(\gg)$, we have 
 \[
 L_{x}(\widehat{Q}(p)) - \widehat{Q}(L_{x}(p)) = \tfrac{1}{2}\br((\pi_{+}\ad_{x}\pi_{-} - \pi_{-}\ad_{x}\pi_{+})A_{p}).
 \]
\end{proposition} 

\begin{proof}
 It suffices to prove the statement for $p=xy \in S^{2}_{\varepsilon}(\gg)$. By Example~\ref{ex::S2} and Lemma~\ref{lemm::bracket_and_br}, we have \[
\widehat{Q}(p) = Q(p) + \tfrac{1}{2}\left([x, \pi_{+}y] + \varepsilon(x,y)[y, \pi_{+}x]\right) = Q(p) - \tfrac{1}{2}\br(\pi_{+}A_{p})
\]
Using $\gg$-equivariance of $Q$ and $\br$, one has
\[
\begin{aligned}
L_x\widehat Q(p)-\widehat Q(L_xp)
&=\tfrac12\bigl(-L_x\operatorname{br}(\pi_+A_p)+\operatorname{br}(\pi_+A_{L_xp})\bigr)\\
&=\tfrac12\operatorname{br}\bigl(-[L_x,\pi_+A_p]_{\widehat{\End}}+\pi_+[L_x,A_p]_{\widehat{\End}}\bigr)\\
&=-\tfrac12\operatorname{br}([L_x,\pi_+]_{\widehat{\End}}A_p)\\
&=\tfrac12\operatorname{br}\bigl((\pi_+\ad_x\pi_--\pi_-\ad_x\pi_+)A_p\bigr),
\end{aligned}
\]
where we use $[L_x,\pi_+]_{\widehat{\End}}
=
\pi_-\ad_x\pi_+ - \pi_+\ad_x\pi_-$ in the last equality, and the fact that $(\pi_+\ad_x\pi_--\pi_-\ad_x\pi_+)A_p$ is finitary.
\end{proof}

\subsection{Completion of the Weil and Quantum Weil Algebra of \texorpdfstring{$\gg$}{}} \label{subsec::completion_Weil_quantum_Weil_superalgebra}

We introduce the Weil algebra and the quantum Weil algebra of $\gg$, describe their completions, and formulate their relation via quantization. Throughout, $\gg^{\ast}$ is identified with $\gg$ by means of the form $B$.

\subsubsection{Weil algebra of \texorpdfstring{$\gg$}{} and Completion} \label{subsubsec::Weil_algebra} 
 The \emph{Weil algebra} of $\gg$ is the $\Gamma$-graded tensor product
\begin{equation}
W_{\varepsilon}(\gg)\coloneqq S_{\varepsilon}(\gg)\otimes\bigwedge\nolimits_{\varepsilon}(\gg).
\end{equation}
It is a $\ZZ$-graded associative color algebra, with $\ZZ$-grading induced by that of $\gg$. We denote the $(i,\gamma)$-graded component by $W_{\varepsilon}(\gg)_{(i,\gamma)}$. It carries in addition the $\ZZ$-grading
\begin{equation} \label{eq::Z_grading_Weil}
W_{\varepsilon}^{k}(\gg)\coloneqq\bigoplus_{2r+s=k}S_{\varepsilon}^{r}(\gg)\otimes \bigwedge\nolimits_{\varepsilon}^{s}(\gg).
\end{equation}
Thus the generators $x\otimes 1$ have degree $2$, whereas the generators $1\otimes x$ have degree $1$. Declaring $x\otimes 1$ to be even and $1\otimes x$ to be odd endows $W_{\varepsilon}(\gg)$ with a compatible $\ZZ_{2}$-grading, and we consider $W_{\varepsilon}(\gg)$ with respect to this $\ZZ_{2}$-grading as a superalgebra, that is, $W_{\varepsilon}(\gg)=W_{\varepsilon}(\gg)_{\bar{0}}\oplus W_{\varepsilon}(\gg)_{\bar{1}}$ is viewed as a $\ZZ$-graded color superalgebra. We emphasize here the additional compatible $\ZZ$- and $\ZZ_{2}$-gradings, since the former is relevant for the completion and the latter for the constructions below. We denote by $p(\cdot)$ the parity with respect to this $\ZZ_{2}$-grading so that $p(x\otimes 1)=\bar{0}$ and $p(1\otimes x)=\bar{1}$. This color superalgebra is $\varepsilon$-supercommutative, that is,
\begin{equation}
 AB = (-1)^{p(A)p(B)}\varepsilon(A,B)BA,
\end{equation}
 with $A,B\in W_{\varepsilon}(\gg)$ homogeneous, or equivalently, with $\varepsilon'(A,B)=(-1)^{p(A)p(B)}\varepsilon(A,B)$, $\varepsilon'$-commutative. Note that $W_{\varepsilon}(\gg)$ is the free graded $\varepsilon$-supersymmetric algebra freely generated by $x\otimes 1$ and $1\otimes x$ with $x \in \gg$.

On $W_{\varepsilon}(\gg)$ there are two natural operations, called \emph{contraction} and \emph{Lie derivative}. For homogeneous $x\in\gg$, the contraction $\iota_{x}$ is defined on generators by
\begin{equation}
\iota_{x}(1\otimes y)=B(x,y)(1\otimes 1),\qquad \iota_{x}(y\otimes 1)=0.
\end{equation}
Further, $x$ acts on $S_{\varepsilon}(\gg)$ and $\bigwedge_{\varepsilon}(\gg)$ by the Lie derivatives $L_{x}^{S}$ and $L_{x}^{\wedge}$, where the superscripts serve here only to indicate the ambient algebra. The \emph{Lie derivative} on $W_{\varepsilon}(\gg)$ is then defined by
\begin{equation}
L_{x}(y\otimes z)\coloneqq L_{x}^{S}(y)\otimes z+\varepsilon(x,y)\,y\otimes L_{x}^{\wedge}(z)
\end{equation}
for any homogeneous $x\in \gg$ and $y\otimes z\in W_{\varepsilon}(\gg)$. The Lie derivative and contraction are even and odd $\varepsilon'$-derivations, respectively, and both are compatible:

\begin{lemma}
For all homogeneous $x,y\in\gg$, one has
\begin{equation*}
\begin{gathered}
[\iota_{x},\iota_{y}]_{\Der}=0,\qquad [L_{x},\iota_{y}]_{\Der}=\iota_{[x,y]},\qquad [L_{x},L_{y}]_{\Der}=L_{[x,y]}.
\end{gathered}
\end{equation*}
\end{lemma}
\begin{proof}
This is a direct computation. For example, for $x,y,z\in\gg$,
\[
\begin{aligned}
[L_{x},\iota_{y}]_{\Der}(z)&=L_{x}\iota_{y}(z)-\varepsilon(x,y)\iota_{y}L_{x}(z)\\
&=L_{x}(B(y,z)(1\otimes 1))-\varepsilon(x,y)B(y,[x,z])\\
&=-\varepsilon(x,y)B(y,[x,z])\\
&=-\varepsilon(x,y)B([y,x],z)\\
&=\varepsilon(x,y)\varepsilon(y,x)\iota_{[x,y]}(z)\\
&=\iota_{[x,y]}(z).
\end{aligned}
\]
The remaining identities are verified similarly.
\end{proof}

Moreover, these operations are compatible with a natural differential on $\widehat{W}_{\varepsilon}(\gg)$, which yields a universal property described in Section~\ref{subsec::CS_element}. 

Similarly, we define the Weil algebras of the $\ZZ$-graded color subalgebras $\gg_{\pm}$. Using $W_{\varepsilon}(\gg)\cong W_{\varepsilon}(\gg_{-})\otimes W_{\varepsilon}(\gg_{+})$, we define the completion $\widehat{W}_{\varepsilon}(\gg)$ as the sum over all $(i,\gamma)\in \ZZ\times\Gamma$ of \begin{equation}
\widehat{W}_{\varepsilon}(\gg)_{(i,\gamma)}
\coloneqq
\prod_{r\ge 0}\;
\bigoplus_{\substack{\alpha,\beta\in\Gamma\\ \alpha+\beta=\gamma}}
W_{\varepsilon}(\gg_-)_{(i-r,\alpha)}
\otimes
W_{\varepsilon}(\gg_+)_{(r,\beta)}.
\end{equation}
The operators $\iota_x$ and $L_x$ extend componentwise to $\widehat{W}_{\varepsilon}(\gg)$, and we denote these extensions by the same symbols. Moreover, since multiplication extends componentwise, we view $\widehat{W}_{\varepsilon}(\gg)$ as a $\ZZ$-graded color superalgebra. 

\subsubsection{Quantum Weil Algebra of \texorpdfstring{$\gg$}{} and Completion} The \emph{quantum Weil algebra}, or \emph{noncommutative Weil algebra}, is defined to be the quantized version of the Weil algebra, which replaces the $\varepsilon$-commutative color algebras $S_{\varepsilon}(\gg)$ and $\bigwedge_{\varepsilon}(\gg)$ by their noncommutative analogues $\UE_{\varepsilon}(\gg)$ and $\Cl_{\varepsilon}(\gg)$, respectively. It was introduced by Alekseev and Meinrenken
\cite{Alekseev_Meinrenken} for Lie algebras.

\begin{definition}
 The \emph{quantum Weil algebra} associated to $\gg$ is the $\Gamma$-graded tensor product
\[
\Weil \coloneqq \UE_{\varepsilon}(\gg) \otimes \Cl_{\varepsilon}(\gg).
\]
\end{definition}

$\Weil$ is a $\ZZ$-graded color algebra with $\ZZ$-grading induced by that of $\gg$. We denote the $(i,\gamma)$-graded component by $\Weil_{(i,\gamma)}$. We denote the generators of $\Weil$ by the elements $x\otimes 1$ and $1 \otimes x$ for $x \in \gg$. By declaring that the generators $1\otimes x$ are odd and the generators $x\otimes 1$ are even, we endow $\Weil$ with an additional $\ZZ_{2}$-grading so that $\Weil = \Weil_{\bar{0}} \oplus \Weil_{\bar{1}}$ is a $\ZZ$-graded color superalgebra. We denote the associated parity by $p(\cdot)$. This color superalgebra is non-$\varepsilon$-supercommutative, and becomes a color Lie superalgebra with bracket 
\begin{equation}\label{eq::bracket_quantum_Weil}
[A,B]_{\WW}\coloneqq AB-\varepsilon'(A,B)\,BA,\qquad \varepsilon'(A,B)\coloneqq(-1)^{p(A)p(B)}\varepsilon(A,B)
\end{equation}
for any homogeneous $A,B \in \Weil$.

On $\Weil$ there are two canonical operations. The \emph{contraction} $\iota_{x}$ is defined on generators by
\begin{equation}
\iota_{x}(1\otimes y)=B(x,y)(1\otimes 1),\qquad \iota_{x}(y\otimes 1)=0,
\end{equation}
for $x,y\in\gg$. The second is the \emph{Lie derivative}. It is the $\Gamma$-grading-preserving combination of the Lie derivatives $L^{U}$ on $\UE_{\varepsilon}(\gg)$ and $L^{C}$ on $\Cl_{\varepsilon}(\gg)$, and is given by
\begin{equation}
L_{x}(y\otimes z)\coloneqq L_{x}^{U}(y)\otimes z+\varepsilon(x,y)y\otimes L_{x}^{C}(z)
\end{equation}
for homogeneous $x\in\gg$, $y\in\UE_{\varepsilon}(\gg)$, and $z\in\Cl_{\varepsilon}(\gg)$. As for $W_{\varepsilon}(\gg)$, the Lie derivative and contraction are compatible:

\begin{lemma}
For all homogeneous $x,y\in\gg$, one has
\begin{equation*}
\begin{gathered}
[\iota_{x},\iota_{y}]_{\Der}=0,\qquad [L_{x},\iota_{y}]_{\Der}=\iota_{[x,y]},\qquad [L_{x},L_{y}]_{\Der}=L_{[x,y]}.
\end{gathered}
\end{equation*}
\end{lemma}

Combining quantization $q : \bigwedge_{\varepsilon}(\gg) \to \Cl_{\varepsilon}(\gg)$ and PBW $\varepsilon$-symmetrization $Q: S_{\varepsilon}(\gg) \to \UE_{\varepsilon}(\gg)$, we obtain the following proposition.
\begin{proposition}\label{prop::quantization_Weil}
The map
$
\mathcal{Q} \coloneqq Q \otimes q : W_{\varepsilon}(\gg) \longrightarrow \Weil
$
is an isomorphism in $\cat$ of $\ZZ\times \ZZ_{2}$-graded vector spaces. Moreover, $\mathcal{Q}$ intertwines the contractions and the Lie derivatives:
\[
L_{x} \circ \calQ = \calQ \circ L_{x}, \qquad 
\iota_{x} \circ \calQ = \calQ \circ \iota_{x}, \qquad x \in \gg.
\]
\end{proposition}
\noindent In what follows, we simply call $\calQ$ the quantization map of $W_{\varepsilon}(\gg)$.

Similarly, we define $\mathcal{W}_{\varepsilon}(\gg_{\pm})$ such that $\Weil \cong \WW_{\varepsilon}(\gg_{-})\otimes \WW_{\varepsilon}(\gg_{+})$. We define the \emph{completed quantum Weil algebra} $\widehat{\mathcal{W}}_{\varepsilon}(\gg)$ as the direct sum over all $(i, \gamma) \in \ZZ\times \Gamma$ of the components
\begin{equation}
\widehat{\mathcal W}_{\varepsilon}(\gg)_{(i,\gamma)}
\coloneqq
\prod_{r\ge 0}\;
\bigoplus_{\substack{\alpha,\beta\in\Gamma\\ \alpha+\beta=\gamma}}
\mathcal W_{\varepsilon}(\gg_-)_{(i-r,\alpha)}\otimes \mathcal W_{\varepsilon}(\gg_+)_{(r,\beta)}.
\end{equation}
The Lie derivative $L_{x}$ and the contraction $\iota_{x}$ extend componentwise to $\widehat{\WW}_{\varepsilon}(\gg)$, and we denote these extensions by the same symbols. Moreover, we can view $\widehat{\WW}_{\varepsilon}(\gg)$ as a color Lie superalgebra with bracket $[\cdot,\cdot]_{\widew}$ extending~\eqref{eq::bracket_quantum_Weil}. 

In what follows, we define for any $x \in \gg$
\begin{equation}
 \gamma^{\WW}(x) \coloneqq x\otimes 1 -\tfrac{1}{2}(1\otimes \widehat{\gamma}(x))\in \widehat{\WW}_{\varepsilon}(\gg), \qquad \widehat{\gamma}(x) \coloneqq \widehat{\gamma}(\ad_{x}).
\end{equation}
The following lemma will be used later. It is proved by a direct computation, as in \cite{Schmidt_perturbations}. 

\begin{lemma} \label{lemm::commutation_relations_qWeil}
 One has for any $x,y \in \gg$:
 \[ \begin{aligned}\begin{gathered}
 [1\otimes x,1\otimes y]_{\widehat{\WW}} = 2B(x,y)(1\otimes 1), \quad [\gamma^{\WW}(x),1\otimes y]_{\widehat{\WW}} = 1\otimes [x,y]_{\gg},\\ [\gamma^{\WW}(x),\gamma^{\WW}(y)]_{\widehat{\WW}} = \gamma^{\WW}([x,y]_{\gg})+\KP(x,y).
 \end{gathered}
 \end{aligned}
 \]
\end{lemma}

Moreover, we obtain a reformulation of the Lie derivative and contraction in $\cWeil$.

\begin{lemma} \label{lemm::contraction_and_Lie_derivative_in_Weil}
For $x,y\in\gg$ one has:
\begin{enumerate}
\item[a)] $\iota_{x}=\tfrac12[1\otimes x,\,\cdot\,]_{\widehat{\WW}}$.
\item[b)] $L_{x}=[\gamma^{\WW}(x),\cdot\,]_{\widehat{\WW}}$.
\end{enumerate}
\end{lemma}

\begin{proof}
a) This is immediate from the definition of $\iota_x$ and the commutator $[1\otimes x,1\otimes y]_{\widehat{\WW}}=2B(x,y)(1\otimes 1)$.
b) It suffices to check the identity on the componentwise generators. For $y\in\gg$,
\[
\begin{split}
L_x(1\otimes y)
=1\otimes L^{C}_x y
=1\otimes\Bigl(-\tfrac12[\widehat{\gamma}(x),y]_{\widehat{\Cl}_{\varepsilon}}\Bigr)
=[\gamma^{\WW}(x),1\otimes y]_{\widehat{\WW}},
\end{split}
\]
using $[x\otimes 1,1\otimes y]_{\widehat{\WW}}=0$ and Proposition~\ref{prop::Lie_derivative_as_commutator}. Since the generators generate componentwise $\cWeil$, this
implies $L_x=[\gamma^{\WW}(x),\cdot]_{\widehat{\WW}}$.
\end{proof}

Finally, we introduce the \emph{normal-ordered quantization}, namely the map
$
\widehat{\mathcal Q}:\widehat{W}_{\varepsilon}(\gg)\longrightarrow \widehat{\mathcal W}_{\varepsilon}(\gg),
$
which, on each homogeneous component $(i,\gamma)\in \ZZ\times \Gamma$, is defined by taking the direct product over $r\geq 0$ and the direct sum over all $\alpha,\beta\in \Gamma$ with $\alpha+\beta=\gamma$ of the maps
\begin{equation}
\mathcal Q\otimes \mathcal Q:
W_{\varepsilon}(\gg_{-})_{(i-r,\alpha)}\otimes W_{\varepsilon}(\gg_{+})_{(r,\beta)}
\longrightarrow
\mathcal W_{\varepsilon}(\gg_{-})_{(i-r,\alpha)}\otimes \mathcal W_{\varepsilon}(\gg_{+})_{(r,\beta)}.
\end{equation}
Since $\mathcal Q$ is an isomorphism of $\ZZ\times \Gamma$-graded super vector spaces on each factor, $\widehat{\mathcal Q}$ is an isomorphism of $\ZZ\times \Gamma$ super vector spaces.

\section{Cubic Dirac Operators for Infinite-Dimensional Quadratic Color Lie Algebras}\label{sec::Cubic_Dirac}
This section extends the cubic Dirac operator $\Dirac$ to arbitrary-dimensional quadratic $\ZZ$-graded color Lie algebras $(\gg,B)$ with trivial Kac–Peterson class and studies its resulting action. As in the finite-dimensional quadratic Lie algebra case, the square of $\Dirac$ is expressed as a sum of normally ordered Casimir operators and a constant. Along the way, we extend the Chern--Weil map to general quadratic $\gg$-differential algebras and identify the element whose quantization yields the cubic Dirac operator with the Chern--Simons element in the completed Weil algebra associated with the quadratic polynomial defined by $B$ and the universal connection. We begin by recalling the necessary background.

\subsection{Normal-ordered Casimir Element} \label{subsec::normal_ordered_Casimir} In analogy with the finite-dimensional case, we define a normal-ordered Casimir element for general infinite-dimensional $\ZZ$-graded color Lie algebras. In general, this element need not be central.

If $\gg$ is finite-dimensional, let $\{e_{a}\}$ be a homogeneous basis with $B$-dual basis $\{e^{a}\}$, and set
$
p\coloneqq\sum_{a}\varepsilon(e_{a})\,e_{a}e^{a}\in S^{2}_{\varepsilon}(\gg).
$
Then $p$ is $\gg$-invariant, and the Casimir element is $Q(p)\in\UE_{\varepsilon}(\gg)$ (\emph{cf.}~\cite{generalized_Clifford}). We now extend this construction to the infinite-dimensional $\ZZ$-graded color Lie algebras $(\gg,B)$.

We fix a homogeneous basis $(e_{a})$ of $\gg$ with $B$-dual basis $(e^{a})$, and define the analogue 
\begin{equation}
p \coloneqq \sum_{a}\varepsilon(e_{a})e_{a}e^{a} = \sum_{a}e^{a}e_{a} \in \widehat{S}_{\varepsilon}^{2}(\gg).
\end{equation}

\begin{lemma}\label{lemm::form_Ap} The following assertions hold: 
\begin{enumerate}
 \item[a)] The element $p \in \widehat{S}_{\varepsilon}^{2}(\gg)$ is $\gg$-invariant.
 \item[b)] One has $A_{p} = 2 \operatorname{id}_{\gg}$.
\end{enumerate}
\end{lemma}

\begin{proof} a) follows from Lemma~\ref{lemm::basis_description}. We prove~b). Recall that $A_{xy}(z)=\varepsilon(y)B(z,y)x+\varepsilon(x)\varepsilon(x,y)B(z,x)y$ (\emph{cf.}~\eqref{eq::definition_A}). Substituting $p=\sum_{a}e^{a}e_{a}$ and using Lemma~\ref{lemm::basis_description} gives for the associated completion map 
\[
\begin{aligned}
A_{p}(z)
&=
\sum_{a}\Bigl(\varepsilon(e_{a})B(z,e_{a})e^{a}
+\varepsilon(e^{a},e_{a})\varepsilon(e^{a})B(z,e^{a})e_{a}\Bigr)\\
&=
\sum_{a}\Bigl(\varepsilon(e_{a})B(z,e_{a})e^{a}+B(z,e^{a})e_{a}\Bigr)\\
&=
z+z
=
2z,
\end{aligned}
\]
where we used $\varepsilon(e^{a},e_{a})=\varepsilon(e^{a})$ and $\varepsilon(e^{a})^{2}=1$. Hence $A_{p}=2\operatorname{id}_{\gg}$.
\end{proof}

This motivates the following definition.

\begin{definition}
 The \emph{normal-ordered Casimir element} is
\begin{equation*}
 \Omega_{\gg}' \coloneqq \widehat{Q}(p) \in \widehat{\UE}_{\varepsilon}(\gg).
\end{equation*}
\end{definition}

However, $\Omega_{\gg}'$ need not lie in the center of $\widehat{\UE}_{\varepsilon}(\gg)$. Its failure to be central is measured by $\Psi_{\operatorname{KP}}\in\widehat{\mathfrak{so}}_{\varepsilon}(\gg)$, the unique element corresponding to $\KP\in\widehat{\bigwedge}_{\varepsilon}^{2}(\gg^{\ast})$ (\emph{cf.}~\eqref{eq::relations_psi_Psi}).

\begin{theorem} \label{thm::Lie_derivative_normal_ordered_Casimir}
 For any $x \in \gg$, we have 
 \[
 L_{x}\Omega_{\gg}' = 2\Psi_{\operatorname{KP}}(x).
 \]
\end{theorem}
\begin{proof}
 It suffices to prove that $B(L_{x}\Omega_{\gg}',y) = B(2\Psi_{\operatorname{KP}}(x),y)$ for all $x,y \in \gg$. We have by Proposition~\ref{prop::Lie_derivative_and_PBW_quantization}
 \[
 L_{x}\Omega_{\gg}' = \widehat{Q}(L_{x}(p))+\tfrac{1}{2}\br((\pi_{+}\ad_{x}\pi_{-} - \pi_{-}\ad_{x}\pi_{+})A_{p} \in \gg
 \]
 However, $p$ is $\gg$-invariant forcing $L_{x}(p) =0$ and $A_{p} = 2\operatorname{id}_{\gg}$ by Lemma~\ref{lemm::form_Ap}, such that pairing with $y\in \gg$ gives
 \[ 
B(L_{x}\Omega_{\gg}',y) = B(\br(\pi_{+}\ad_{x}\pi_{-}-\pi_{-}\ad_{x}\pi_{+}),y). 
 \]
 On the other hand, invoking Lemma~\ref{lemm::br_and_str} and the $\varepsilon$-cyclicity of $\etr$, we have
 \[
\begin{aligned}B(\br(\pi_{+}\ad_{x}\pi_{-} - \pi_{-}\ad_{x}\pi_{+}),y) &= \varepsilon(x,y)(\etr(\ad_{y}\pi_{+}\ad_{x}\pi_{-})-\etr(\ad_{y}\pi_{-}\ad_{x}\pi_{+})) \\ &= \etr(\ad_{x}\pi_{-}\ad_{y}\pi_{+})-\varepsilon(x,y)\etr(\ad_{y}\pi_{-}\ad_{x}\pi_{+}) \\ &= 2\KP(x,y) \\ &= B(2\Psi_{\operatorname{KP}}(x),y),
\end{aligned}
 \]
 where we used the definition of the Kac–Peterson 2-cocycle and~\eqref{eq::relations_psi_Psi}. 
\end{proof}

The theorem indicates that the normal-ordered Casimir element $\Omega_{\gg}'$ admits a correction to a central element in $\widehat{\UE}_{\varepsilon}(\gg)$ precisely when the Kac–Peterson class vanishes, that is, $\KP = \operatorname{d}\!\rho$ for some $\rho \in (\gg^{\ast})_{(0,\ng)}$. Recall the definition of the musical isomorphism 
$\sharp \colon \gg^{\ast} \xrightarrow{\ \sim\ } \gg$ 
introduced in~\eqref{eq::musical_isomorphisms}, and denote by $\rho^{\sharp}$ the image of any $\rho \in \gg^{\ast}$ under this map.

\begin{corollary} \label{cor::Casimir_and_KP_class}
 For $\rho \in (\gg^{\ast})_{(0,\ng)}$ the following two assertions are equivalent:
 \begin{enumerate}
 \item[a)] $\Omega_{\gg} \coloneqq \Omega_{\gg}' + 2\rho^{\sharp}$ lies in the center of $\widehat{\UE}_{\varepsilon}(\gg)$.
 \item[b)] $\KP = \operatorname{d}\!\rho$.
 \end{enumerate}
\end{corollary}
\begin{proof}
By Theorem~\ref{thm::Lie_derivative_normal_ordered_Casimir}, one has
\[
L_{x}\Omega_{\gg}
= L_{x}\Omega_{\gg}' + 2L_{x}\rho^{\sharp}
= 2\Psi_{\operatorname{KP}}(x) + 2[x,\rho^{\sharp}].
\]
According to Lemma~\ref{lemm::properties_Psi}, the condition 
$\KP = \operatorname{d}\!\rho$ is equivalent to 
$\Psi_{\operatorname{KP}}(x) = [\rho^{\sharp}, x] = -[x, \rho^{\sharp}] = -L_{x}\rho^{\sharp}$ 
for every $x \in \gg$, since $\rho$ is even, and consequently $\rho^{\sharp}$ is even. The statement follows.
\end{proof}

\subsection{Structure Constants Tensor} \label{subsec::structure_constants_tensor}
The structure constants tensor $\phi$ is the tensorial incarnation of the Lie bracket after identification of $\gg$ with $\gg^{\ast}$ via $B$. It therefore records the full color Lie algebra structure of $\gg$ as a canonical cubic element of $\widehat{\bigwedge}^{3}_{\varepsilon}(\gg)$.

Concretely, the \emph{structure constants tensor} is the unique cubic element $\phi \in \widehat{\bigwedge}^{3}_{\varepsilon}(\gg)_{(0,\ng)}$ such that 
\begin{equation}\label{eq::def_phi}
 \langle\phi,x\wedge y\wedge z\rangle_{\wedge}=-\tfrac{1}{2}B([x,y],z)
\end{equation}
for all homogeneous $x,y,z\in \gg$. Its existence follows from the $\varepsilon$-skew-symmetry of the bracket together with the invariance of $B$, which imply that $(x,y,z)\mapsto -\tfrac12 B([x,y],z)$ is an $\varepsilon$-alternating trilinear form. It packages the structure constants $f_{abc}\coloneqq B([e_{a},e_{b}],e_{c})$ into a single canonical cubic tensor.

\begin{lemma}\label{lemm::properties_phi}
For all $x,y\in\gg$, one has:
\begin{enumerate}
\item[a)] $2\iota_{x}\phi=-\uplambda(\ad_{x})$.
\item[b)] $\iota_{x}\iota_{y}\phi=\tfrac12B([x,y],\cdot)$. In particular, under the identification $\gg\cong\gg^{\ast}$, one has $\iota_{x}\iota_{y}\phi=\tfrac12[x,y]$.
\item[c)] In a homogeneous basis,
\[
\phi=-\tfrac{1}{12}\sum_{a,b,c}\varepsilon(e_{a},e_{b})\varepsilon(e_{c})f_{abc}\,e^{a}\wedge e^{b}\wedge e^{c}.
\]
\item[d)] $\phi$ is $\gg$-invariant, that is, $L_{x}\phi=0$ for all $x\in\gg$.
\end{enumerate}
\end{lemma}
\begin{proof}
For a), let $x,y,z\in\gg$ be homogeneous. By Lemma~\ref{lemm::identification_lambda_omega},
\[
(\iota_{x}\phi)(y\wedge z)=\phi(x,y,z)=-\tfrac12B([x,y],z)=-\tfrac12\omega_{\ad_{x}}(y,z)=-\tfrac12(\uplambda(\ad_{x}))^{\flat}(y\wedge z).
\]
Hence $2(\iota_{x}\phi)^{\flat}=-(\uplambda(\ad_{x}))^{\flat}$, and therefore $2\iota_{x}\phi=-\uplambda(\ad_{x})$.

For b), by definition and Lemma~\ref{lemm::action_iota_inner_product_wedge},
\[
\begin{aligned}
\iota_{x}\iota_{y}\phi(z)&=(\iota_{x}\iota_{y}\phi,z)_{\wedge}=\varepsilon(x,y)(\iota_{y}\phi,x\wedge z)_{\wedge}=\varepsilon(x,y)(\phi,y\wedge x\wedge z)_{\wedge}\\
&=-\tfrac12\varepsilon(x,y)B([y,x],z)=\tfrac12B([x,y],z).
\end{aligned}
\]

Assertion c) is a direct computation.

For d), let $w,x,y,z\in\gg$ be homogeneous. Since $\phi(x,y,z)=-\tfrac12B([x,y],z)$, one has
\[
(L_{w}^{\wedge}\phi)(x,y,z)=-\phi([w,x],y,z)-\varepsilon(w,x)\phi(x,[w,y],z)-\varepsilon(w,x+y)\phi(x,y,[w,z]).
\]
Substituting the defining formula for $\phi$, the claim follows from the invariance of $B$ and the $\varepsilon$-Jacobi identity.
\end{proof}

The \emph{normal-ordered quantization} of $\phi$ is
\begin{equation}
\phi'\coloneqq \widehat{q}(\phi)\in\widehat{\Cl}_{\varepsilon}(\gg).
\end{equation}
It is odd with respect to the canonical $\ZZ_{2}$-grading. For $r\in\ZZ$, let $\phi'_{r}$ denote the component of $\phi'$ in $\bigwedge_{\varepsilon}(\gg_{-})_{-r}\otimes \bigwedge_{\varepsilon}(\gg_{+})_{r}$. Then $\phi'_{0}$ is the structure constants tensor of the finite-dimensional color Lie algebra $\gg_{0}$.

\begin{theorem}[{\cite[Theorem 1.1]{Dirac_quadratic}}]\label{thm::phi_square_in_null}
The square of $\phi'_{0}$ is the constant
\begin{equation}
(\phi'_{0})^{2}=\tfrac{1}{24}\etr(\ad_{\gg_{0}}(\Omega_{\gg_{0}})).
\end{equation}
\end{theorem}

In general, $\phi'$ is not $\gg$-invariant. Recall that $\KPsharp\in\widehat{\bigwedge}_{\varepsilon}^{2}(\gg)$ is the image of $\KP\in\widehat{\bigwedge}_{\varepsilon}^{2}(\gg^{\ast})$ under the musical isomorphism $\sharp$ in~\eqref{eq::musical_isomorphisms}.

\begin{lemma}\label{lemm::komische_Beziehungen}
For all homogeneous $x\in\gg$, one has
\begin{equation}
L_{x}\phi'=\Psi_{\operatorname{KP}}(x)=\iota_{x}\widehat{q}(\KPsharp).
\end{equation}
\end{lemma}

\begin{proof}
We embed $\gg$ canonically into $\bigwedge_{\varepsilon}(\gg)$. Then it is enough to show that $\iota_{x}(L_{y}\phi'-\Psi_{\operatorname{KP}}(y))=0$ for all homogeneous $x,y\in\gg$. First, by~\eqref{eq::relations_psi_Psi} and the skew-$\varepsilon$-symmetry of $\KP$,
\[
\iota_{x}\Psi_{\operatorname{KP}}(y)=B(x,\Psi_{\operatorname{KP}}(y))=\varepsilon(x,y)B(\Psi_{\operatorname{KP}}(y),x)=\varepsilon(x,y)\KP(y,x)=-\KP(x,y).
\]
On the other hand, using $[\iota_{x},L_{y}]_{\Der}=\iota_{[x,y]}$, Lemma~\ref{lemm::properties_phi}, Proposition~\ref{prop::normal_ordered_quantization_contraction}, and Proposition~\ref{prop::commutator_and_quantization}, one obtains
\[
\begin{aligned}
\iota_{x}L_{y}\phi'&=\varepsilon(x,y)L_{y}\iota_{x}\phi'+\iota_{[x,y]}\phi'\\
&=-\tfrac12\varepsilon(x,y)L_{y}\widehat q(\uplambda(\ad_{x}))-\tfrac12\widehat q(\uplambda(\ad_{[x,y]}))\\
&=-\tfrac12\varepsilon(x,y)\widehat q(L_{y}\uplambda(\ad_{x}))+\varepsilon(x,y)\KP(y,x)-\tfrac12\widehat q(\uplambda(\ad_{[x,y]}))\\
&=-\KP(x,y)-\tfrac12\bigl(\varepsilon(x,y)\widehat q(\uplambda(\ad_{[y,x]}))-\widehat q(\uplambda(\ad_{[x,y]}))\bigr)\\
&=-\KP(x,y).
\end{aligned}
\]
Hence
\begin{equation*}
\iota_{x}(L_{y}\phi'-\Psi_{\operatorname{KP}}(y))=-\KP(x,y)+\KP(x,y)=0.\qedhere
\end{equation*}
\end{proof}

This motivates the definition of a modified structure constants tensor, which is $\gg$-invariant. 

\begin{corollary}
 Suppose $\KP = \operatorname{d}\!\rho$ for some $\rho \in (\gg^{\ast})_{(0,\ng)}$. Then $\upphi \coloneqq \phi' + \rho^{\sharp}$ is $\gg$-invariant.
\end{corollary}

\begin{proposition} \label{prop::phi_squared} The square of $\phi'$ is given by the formula
 \[
 (\phi')^{2} = \widehat{q}(\KPsharp) + \tfrac{1}{24} \etr (\ad_{\gg_{0}}(\Omega_{\gg_{0}})).
 \]
\end{proposition}

\begin{proof}
The difference $(\phi')^{2}-\widehat{q}(\KPsharp)$ is constant iff $\iota_{x}((\phi')^{2}-\widehat{q}(\KPsharp))=0$ for all $x\in\gg$. By Proposition~\ref{prop::Lie_derivative_as_commutator} and Lemma~\ref{lemm::komische_Beziehungen_eins}, we have
\[
\iota_{x}((\phi')^{2}-\widehat{q}(\KPsharp))
=-\tfrac{1}{2}[\widehat{\gamma}(x),\phi']_{\widehat{\Cl}_{\varepsilon}}-\iota_{x}\widehat{q}(\KPsharp)
=L_{x}\phi'-\iota_{x}\widehat{q}(\KPsharp)=0.
\]
Here, we use $\iota_{x}((\phi')^{2})=[\iota_{x}\phi',\phi']_{\widehat{\Cl}_{\varepsilon}}$ since $\phi'$ is odd in $\widehat{\Cl}_{\varepsilon}(\gg)$, and $2\iota_{x} \phi' = 2\widehat{q}(\iota_{x}\phi)= -\widehat{q}(\uplambda(\ad_{x}))=-\widehat{\gamma}(x)$ by Lemma~\ref{lemm::properties_phi}.
The constant is $(\phi_{0}')^{2}$, since for all $r>0$ the commutator of $\phi'$ with $q(\phi_{r})$ lies in the right ideal generated by $\gg_{+}$. The theorem follows from Theorem~\ref{thm::phi_square_in_null}.
\end{proof}

\subsection{Chern--Simons Element Perspective}\label{subsec::CS_element}

We introduced the $\gg$-invariant elements $p=\sum_{a}e^{a}\otimes e_{a}$ and $1\otimes\phi$ in the $\ZZ$-graded color superalgebra $\widehat{W}_{\varepsilon}(\gg)$. Their sum $\calD$ is central to this paper, since its quantization is the cubic Dirac operator of $\gg$. In this section, we study $\calD$ and show that it is the canonical Chern--Simons element associated with the quadratic form $B$ via the universal property of the completed Weil algebra $\widehat{W}_{\varepsilon}(\gg)$. This also explains its appearance in physics as the boundary operator of one-dimensional Chern--Simons theory.

\subsubsection{The Universal Property of the Weil Algebra} The Weil algebra $W(\gg)$ of a Lie algebra $\gg$ is characterized by the following universal property \cite[Chapter 6]{Meinrenken_book}: it is the initial $\gg$-differential algebra equipped with a connection. Accordingly, for every $\gg$-differential algebra $\mathcal A$ with connection $A$, there exists a unique morphism of $\gg$-differential algebras
$
c_{A}\colon W(\gg)\to\mathcal A
$
sending the universal connection $A_{\univ}$ of $W(\gg)$ to $A$, and hence the universal curvature $F_{\univ}$ to the curvature of $A$. We study its completed version for general $\ZZ$-graded color Lie algebras $\gg$.

The $\ZZ$-graded color Lie superalgebra $(\widehat{W}_{\varepsilon}(\gg),[\cdot,\cdot]_{\widehat{W}})$ is componentwise generated by the elements $x\otimes 1$ and $1\otimes x$, with $x\in\gg$, and carries Lie derivative $L_{x}$ and contraction $\iota_{x}$. We regard these generators as $\ZZ\times\Gamma$-homogeneous maps
$
A_{\univ}\colon \gg\to \widehat{W}_{\varepsilon}(\gg)_{\bar1}$ and $F_{\univ}\colon \gg\to \widehat{W}_{\varepsilon}(\gg)_{\bar0}$
defined by
\begin{equation}
A_{\univ}(x)\coloneqq 1\otimes x,\qquad F_{\univ}(x)\coloneqq x\otimes 1.
\end{equation}
With respect to the $\ZZ$-grading \eqref{eq::Z_grading_Weil}, one has, in general,
$
A_{\univ}(x)\in\widehat{W}^{1}_{\varepsilon}(\gg)$ and $F_{\univ}(x)\in\widehat{W}^{2}_{\varepsilon}(\gg).
$

To define a differential in terms of $\Auniv$ and $F_{\univ}$, we first introduce notation and define a bracket $[A,B]$ for $\ZZ\times \Gamma$-homogeneous maps $A,B\colon\gg\to\widehat{W}_{\varepsilon}(\gg)$. We identify $\gg$ with its restricted dual $\gg^{\ast}$ throughout. We define the completion $\widehat{\gg \otimes \gg}$ as the direct sum of the spaces $(\widehat{\gg\otimes \gg})_{(i,\gamma)}$, where
\begin{equation}
 \widehat{\gg\otimes\gg}_{(i,\gamma)}
=
\prod_{r\in \ZZ}
\bigoplus_{\alpha+\beta=\gamma}
\gg_{(i-r,\alpha)}
\otimes
\gg_{(r,\beta)}.
\end{equation}
This contains $\widehat{\bigwedge}_{\varepsilon}^{2}(\gg)$ canonically since the antisymmetrization map extends componentwise. The \emph{dual of the color Lie bracket}, that is, $[\cdot,\cdot]^{\ast} : \gg \to \widehat{\bigwedge}^{2}_{\varepsilon}(\gg)$ so that $\langle [\cdot,\cdot]^{\ast}(z),\,x\wedge y\rangle_{\wedge}
=
\langle z,\,[x,y]\rangle_{\wedge}$, is given by 
\begin{equation}
 [\cdot,\cdot]^{\ast}(z) = \tfrac{1}{2}\sum_{a,b}\varepsilon(e_{a})\varepsilon(e_{b})\varepsilon(e_{b},e_{a})B(z,[e_{a},e_{b}])e^{a}\wedge e^{b}.
\end{equation}
Likewise, using $W_{\varepsilon}(\gg)\otimes W_{\varepsilon}(\gg)\cong (W_{\varepsilon}(\gg_{-})\otimes W_{\varepsilon}(\gg_{-})\otimes (W_{\varepsilon}(\gg_{+})\otimes W_{\varepsilon}(\gg_{+}))$ as $\ZZ$-graded $\varepsilon$-supersymmetric color superalgebras, we define the completion $\widehat{W_{\varepsilon}(\gg)\otimes W_{\varepsilon}(\gg)}$ as the direct sum of all spaces $\widehat{W_{\varepsilon}(\gg)\otimes W_{\varepsilon}(\gg)}_{(i,\gamma)}$, where 
\begin{equation}
 \widehat{W_{\varepsilon}(\gg)\otimes W_{\varepsilon}(\gg)}_{(i,\gamma)} = \prod_{r\geq 0} \bigoplus_{\alpha+\beta=\gamma} (W_{\varepsilon}(\gg_{-})\otimes W_{\varepsilon}(\gg_{-}))_{(i-r,\alpha)} \otimes (W_{\varepsilon}(\gg_{+})\otimes W_{\varepsilon}(\gg_{+}))_{(r,\beta)}.
\end{equation}
For $\ZZ\times\Gamma$-homogeneous maps $A,B\colon\gg\to\widehat W_{\varepsilon}(\gg)$, the map
$A\otimes B$ extends componentwise to a linear map
$
A\otimes B\colon
\widehat{\gg\otimes\gg}
\longrightarrow
\widehat{W_{\varepsilon}(\gg)\otimes W_{\varepsilon}(\gg)}.
$
Indeed, for every fixed homogeneous component of the target, only
finitely many homogeneous components of the source contribute. We then define $[A,B]$ as the composition
\begin{equation}\label{eq::def_bracket}
\gg\xrightarrow{[\cdot,\cdot]^{\ast}}\widehat{\bigwedge}^{2}_{\varepsilon}(\gg)\hookrightarrow \widehat{\gg \otimes \gg}\xrightarrow{A \otimes B} \widehat{W_{\varepsilon}(\gg) \otimes W_{\varepsilon}(\gg)}\xrightarrow{2\cdot\text{mult}} \widehat{W}_{\varepsilon}(\gg).
\end{equation}

Altogether, we define the \emph{Weil differential} on generators by
\begin{equation}\label{eq::def_Weil_differential}
\begin{aligned}
\dw A_{\univ}(x)
&\coloneqq
F_{\univ}(x)-\tfrac12[A_{\univ},A_{\univ}](x)\\
&=
x\otimes 1
-\tfrac12\sum_{a,b}\varepsilon(e_{a})\varepsilon(e_{b})\varepsilon(e_b,e_a)B(x,[e_a,e_b])\,
1\otimes e^{a}\wedge e^{b},
\\
\dw F_{\univ}(x)
&\coloneqq
-[A_{\univ},F_{\univ}](x)\\
&=-\sum_{a,b}\varepsilon(e_{a})\varepsilon(e_{b})B(x,[e_a,e_b])\,
e^{b}\otimes e^{a},
\end{aligned}
\end{equation}
and extend it uniquely to an odd $\varepsilon$-derivation (\emph{cf.}~\eqref{eq::parity_epsilon_derivations}). These formulas are the color analogues of the corresponding formulas in the finite-dimensional case, now in the infinite-dimensional setting and with infinite sums. The following lemma is proved analogously to the finite-dimensional case by direct computation.

\begin{lemma}
One has
\begin{equation*}
(\dw)^{2}=0,\qquad [\dw,\iota_{x}]_{\Der}=L_{x},\qquad [\dw,L_{x}]_{\Der}=0.
\end{equation*}
\end{lemma}

This motivates the following notion of $\gg$-differential algebra, in the spirit of Meinrenken and Alekseev \cite{Lie_Theory_Alekseev_Meinrenken}. Let $\calA$ be a $\ZZ$-graded color superalgebra. Assume $\calA$ is $\varepsilon'$-commutative, that is,
\begin{equation}
xy = \varepsilon'(x,y)yx, \qquad \varepsilon'(x,y) = (-1)^{p(x)p(y)}\varepsilon(x,y)
\end{equation}
for all homogeneous $x,y\in \calA$. Moreover, assume that $\calA$ is quadratic, that is, equipped with a nondegenerate $\varepsilon'$-symmetric bilinear form $\langle\cdot,\cdot\rangle$ such that $\langle x,y\rangle=0$ unless $\vert x\vert +\vert y\vert = \ng$, $p(x)=p(y)$, and $\deg(x)+\deg(y)=0$.

Assume that $\mathcal A$ admits a polarization compatible with the
$\ZZ$-grading. By this we mean a pair of
$\ZZ\times\Gamma\times\ZZ_{2}$-graded unital subalgebras
$\mathcal A^{-},\mathcal A^{+}\subseteq\mathcal A$ such that
\begin{equation}
\mathcal A^{-}_{i}=0\quad\text{for }i>0,
\qquad
\mathcal A^{+}_{i}=0\quad\text{for }i<0,
\qquad
m\colon\mathcal A^{-}\otimes\mathcal A^{+}
\xrightarrow{\sim}\mathcal A,
\tag{3.13}
\end{equation}
where $m(a^{-}\otimes a^{+})=a^{-}a^{+}$ is an isomorphism of
$\ZZ\times\Gamma\times\ZZ_{2}$-graded vector spaces.

Using this polarization, define $\widehat{\mathcal A}
:=
\bigoplus_{(i,\gamma)\in\ZZ\times\Gamma}
\widehat{\mathcal A}_{(i,\gamma)},$
where
\begin{equation}
\widehat{\mathcal A}_{(i,\gamma)}
:=
\prod_{r\geq0}
\bigoplus_{\alpha+\beta=\gamma}
\mathcal A^{-}_{(i-r,\alpha)}
\otimes
\mathcal A^{+}_{(r,\beta)}.
\tag{3.14}
\end{equation}
The $\ZZ_{2}$-grading of $\mathcal A$ induces a $\ZZ_{2}$-grading on
$\widehat{\mathcal A}$.

\begin{definition}\label{def::completed_g_differential_algebra}
A \emph{completed $\gg$-differential algebra} is a tuple
$(\widehat{\calA},\diff,L,\iota)$ consisting of the completion
$\widehat{\calA}$ of an $\varepsilon'$-commutative quadratic
$\ZZ$-graded color superalgebra $\calA$ admitting a polarization,
a derivation
$\diff\in\Der_{(\Gamma,\varepsilon')}^{\bar{1}}(\widehat{\calA})$
of degree $(0,\ng)$, and linear maps
$L\colon\gg\to\Der_{(\Gamma,\varepsilon')}^{\bar{0}}(\widehat{\calA})$
and
$\iota\colon\gg\to\Der_{(\Gamma,\varepsilon')}^{\bar{1}}(\widehat{\calA})$.
For $x\in\gg_{(i,\gamma)}$, the derivations $L_x$ and $\iota_x$
have degree $(i,\gamma)$, and for all $x,y\in\gg$,
\begin{equation*}
\begin{gathered}
\diff^{2}=0,\qquad [\iota_x,\iota_y]_{\Der}=0,\qquad [\diff,\iota_x]_{\Der}=L_x,\\
[L_x,\iota_y]_{\Der}=\iota_{[x,y]},\qquad [L_x,L_y]_{\Der}=L_{[x,y]},\qquad [\diff,L_x]_{\Der}=0.
\end{gathered}
\end{equation*}
A \emph{homomorphism of completed $\gg$-differential algebras}
$\phi:(\widehat{\calA},d^{\widehat{\calA}},L^{\widehat{\calA}},\iota^{\widehat{\calA}})\to (\widehat{{\calB}},d^{\widehat{\calB}},L^{\widehat{\calB}},\iota^{\widehat{\calB}})$
is a homomorphism of $\ZZ$-graded color superalgebras
$\phi:\widehat{\calA}\to \widehat{\mathcal B}$
such that for every $x\in \gg$,
\[
\phi\circ d^{\widehat{\calA}}=d^{\widehat{\mathcal B}}\circ \phi,\qquad
\phi\circ L_x^{\widehat{\calA}}=L_x^{\widehat{\mathcal B}}\circ \phi,\qquad
\phi\circ \iota_x^{\widehat{\calA}}=\iota_x^{\widehat{\mathcal B}}\circ \phi.
\]
\end{definition}

Extracting the roles of $A_{\univ}$ and $F_{\univ}$, the definition of $\dw$, and the identities $\iota_{x}A_{\univ}(y)=B(x,y)$ and $L_{x}A_{\univ}(y)=1\otimes [x,y]=A_{\univ}([x,y])$, one is led to the notions of \emph{connection} and \emph{curvature}. For $\ZZ\times\Gamma$-homogeneous maps $A,B\colon\gg\to\widehat{\calA}$, we define the bracket $[A,B]$ exactly as in~\eqref{eq::def_bracket}. Concretely, if $(e_{a})$ is a homogeneous basis of $\gg$ with $B$-dual basis $(e^{a})$, then
\begin{equation}
[A,B](x)\coloneqq \sum_{a,b}\varepsilon(e_{a})\varepsilon(e_{b})\varepsilon(e_{b},e_{a})B(x,[e_{a},e_{b}])\,A(e^{a})B(e^{b}).
\end{equation}

\begin{definition}
Let $(\widehat{\calA},\diff,\iota,L)$ be a completed $\gg$-differential algebra. A \emph{connection} is a $\ZZ\times\Gamma$-homogeneous map $A\colon \gg\to \widehat{\calA}_{\bar1}$ of degree $(0,\ng)$ such that
\[
\iota_{x}(A(y))=B(x,y),\qquad L_{x}(A(y))=A([x,y])
\]
for all $x,y\in\gg$. The curvature of a connection $A$ is
\[
F^{A}\coloneqq \diff A+\tfrac12[A,A].
\]
\end{definition} 

\begin{remark} \begin{enumerate} \item If the superalgebra structure of $\widehat{\calA}$ is induced by a compatible $\ZZ$-grading $\widehat{\calA}^{k}$, then one requires
$
A:\gg\to\widehat{\calA}^{1}$ and $
F^{A}:\gg\to\widehat{\calA}^{2}.
$ \item
For any $\ZZ\times\Gamma$-homogeneous map $A\colon\gg\to\widehat{\calA}$, the Lie derivative and contraction induce maps
\[
(L_{x}A)(y)\coloneqq L_{x}(A(y))-A([x,y]),\qquad (\iota_{x}A)(y)\coloneqq \iota_{x}(A(y))
\]
for $x,y\in\gg$. Thus $A$ is a connection if and only if
\[
L_{x}A=0,\qquad \iota_{x}A=B(x,\cdot).
\]
\end{enumerate}
\end{remark}

The following lemma collects useful identities.

\begin{lemma}\label{lemm::properties_connection}
Let $\widehat{\calA}$ be a completed $\gg$-differential algebra with connection $A$ and curvature $F^{A}$.
\begin{enumerate}
\item[a)] One has $\iota_{x}F^{A}(y)=0$ and $L_{x}F^{A}(y)=F^{A}([x,y])$ for all $x,y\in\gg$.
\item[b)] $[A,[A,A]]=0$.
\item[c)] $\diff[A,A]=[\diff A,A]-[A,\diff A]$.
\item[d)] $\diff F^{A}+[A,F^{A}]=0$.
\end{enumerate}
\end{lemma}

\begin{proof} a) We first prove $\iota_{x}F^{A}=0$ for all $x\in \gg$. By definition 
\[
\begin{aligned}
\iota_{x}F^{A}(y)&=\iota_{x}\diff A(y) + \tfrac{1}{2}\iota_{x}[A,A](y) = L_{x}A(y)-\diff \iota_{x}A(y) +\tfrac{1}{2}\iota_{x}[A,A](y) \\ &= A([x,y])-\diff B(x,y) + \tfrac{1}{2}\iota_{x}[A,A](y) \\ &= A([x,y])+\tfrac{1}{2}\iota_{x}[A,A](y)
\end{aligned}
\]
and the statement follows if we prove $\tfrac{1}{2}\iota_{x}[A,A](y) = -A([x,y])$. Indeed, using skew-$\varepsilon$-symmetry of the bracket, Lemma~\ref{lemm::basis_description} and $B(x,y)=0$ unless $\vert x\vert +\vert y\vert = \ng$ one obtains for all homogeneous $x,y\in \gg$
\[
\begin{aligned}
 \iota_{x}[A,A](y) &= \sum_{a,b}\varepsilon(e_{a})\varepsilon(e_{b})\varepsilon(e_{b},e_{a})B(y,[e_{a},e_{b}])((\iota_{x}A(e^{a}))A(e^{b})-\varepsilon(x,e^{a})A(e^{a})\iota_{x}(A(e^{b}))) \\ &=\sum_{a,b}\varepsilon(e_{a})\varepsilon(e_{b})\varepsilon(e_{b},e_{a})B(y,[e_{a},e_{b}])(B(x,e^{a})A(e^{b})-\varepsilon(x,e^{a})B(x,e^{b})A(e^{a})) \\ &= \sum_{b}\varepsilon(x)\varepsilon(e_{b})\varepsilon(e_{b},x)B(y,[x,e_{b}])A(e^{b}) +\sum_{a,b}\varepsilon(e_{a})\varepsilon(e_{b})\varepsilon(x,e^{a})B(y,[e_{b},e_{a}])B(x,e^{b})A(e^{a}) \\ &= 2\sum_{b}\varepsilon(x)\varepsilon(e_{b})\varepsilon(e_{b},x)B(y,[x,e_{b}])A(e^{b}) \\&=2\varepsilon(x)\varepsilon(-x-y,x) A(\sum_{b}\varepsilon(e_{b})B([y,x],e_{b})e^{b}) \\ &=2\varepsilon(x)\varepsilon(-x-y,x)A([y,x]) \\ &= -2(\varepsilon(x)\varepsilon(-x-y,x)\varepsilon(y,x)A([x,y]) \\ &= -2A([x,y])
\end{aligned}
\]
where we used $\varepsilon(x,e^{a})=\varepsilon(x,e_{a})^{-1}=\varepsilon(e_{a},x)$ in the fourth equality and change summation indices. 

Further, since $A$ is a connection,
\[
L_{x}F^{A}=L_{x}\diff A+\tfrac12L_{x}[A,A]=\tfrac12L_{x}[A,A].
\]
It therefore suffices to show that $L_{x}[A,A]=0$. By definition of $[A,A]$, since $L_{x}$ is an even $\varepsilon$-derivation and $L_{x}A=0$, one has
\[
L_{x}[A,A](y)=[L_{x}A,A](y)+\varepsilon(x,A)[A,L_{x}A](y)+[A,A](L_{x}y)=[A,A]([x,y]).
\]
Hence
\[
(L_{x}[A,A])(y)=L_{x}([A,A](y))-[A,A]([x,y])=0.
\]

b) This follows from the $\varepsilon$-Jacobi identity. Indeed,
\[
[A,[A,A]](x)=\sum_{a,b,c}C_{abc}B(x,[e_{a},[e_{b},e_{c}]])A(e^{a})A(e^{b})A(e^{c}),
\]
where $C_{abc}=\varepsilon(e_{a})\varepsilon(e_{b})\varepsilon(e_{c})\varepsilon(e_{b}+e_{c},e_{a})\varepsilon(e_{c},e_{b})$. Since $\calA$ is $\varepsilon$-supersymmetric and $A$ is odd,
\[
C_{bca}A(e^{b})A(e^{c})A(e^{a})=\varepsilon(e_{a},e_{b}+e_{c})C_{abc}A(e^{a})A(e^{b})A(e^{c}),
\]
\[
C_{cab}A(e^{c})A(e^{a})A(e^{b})=\varepsilon(e_{c},e_{a}+e_{b})C_{abc}A(e^{a})A(e^{b})A(e^{c}).
\]
Hence
\[
\begin{aligned}
[A,[A,A]](x)&=\tfrac13\sum_{a,b,c}C_{abc}\Bigl(B(x,[e_{a},[e_{b},e_{c}]])+\varepsilon(e_{a},e_{b}+e_{c})B(x,[e_{b},[e_{c},e_{a}]])\\
&\qquad\qquad\qquad\qquad\qquad\qquad+\varepsilon(e_{a}+e_{b},e_{c})B(x,[e_{c},[e_{a},e_{b}]])\Bigr)A(e^{a})A(e^{b})A(e^{c})=0.
\end{aligned}
\]

c) This follows since $A$ and $\diff$ are odd and $\diff$ is an $\varepsilon$-derivation.

d) By b) and c),
\begin{equation*}
\begin{aligned}
\diff F^{A}&=\diff\Bigl(\diff A+\tfrac12[A,A]\Bigr)=\tfrac12\diff[A,A]=-[A,\diff A]\\
&=-[A,F^{A}-\tfrac12[A,A]]=-[A,F^{A}]+\tfrac12[A,[A,A]]=-[A,F^{A}].
\end{aligned} \qedhere
\end{equation*}
\end{proof}

The completed Weil algebra $\widehat{W}_{\varepsilon}(\gg)$ is the universal completed $\gg$-differential algebra equipped with a connection. Its universal connection is $\Auniv$, with universal curvature $F_{\univ}$, and every connection $A$ in a completed $\gg$-differential algebra $\widehat{\calA}$ arises uniquely as the image of $\Auniv$ under a morphism of completed $\gg$-differential algebras. Indeed, for any completed $\gg$-differential algebra $\widehat{\calA}$ with connection $A$ define on generators
\begin{equation}
 c_{A}(A_{\univ})\coloneqq A, \qquad c_{A}(F_{\univ})\coloneqq F^{A}.
\end{equation}
Since $\widehat{\calA}$ is $\varepsilon'$-commutative and $A_{\univ},F_{\univ}$ generate $\widehat{W}_{\varepsilon}(\gg)$ componentwise, the homogeneity of $A$ and $F^{A}$ implies that the assignments $A_{\univ}\mapsto A$ and $F_{\univ}\mapsto F^{A}$ define a unique (continuous) homomorphism $c_A:\widehat{W}_{\varepsilon}(\gg)\to\widehat{\calA}$. We call $c_{A}$ the \emph{Chern--Weil map} for completed $\gg$-differential algebras of $\ZZ$-graded color Lie algebras.

\begin{theorem}\label{thm::Chern_Weil_map}
 Let $(\widehat{\calA},d,L,\iota)$ be a completed $\gg$-differential algebra. Fix a connection $A: \gg \to \widehat{\calA}_{\bar{1}}$. Then there exists a unique completed $\gg$-differential algebra homomorphism $c_{A}:\widehat{W}_{\varepsilon}(\gg) \to \widehat{\calA}$ such that 
 \[
 c_{A}(\Auniv)=A, \qquad c_{A}(F_{\univ})=F^{A}.
 \]
\end{theorem}

\begin{proof}
It remains to verify compatibility of $c_{A}$ with Lie derivatives, contractions, and differentials. Since $c_{A}$ is defined on the componentwise generators $A_{\univ}$ and $F_{\univ}$, it suffices to check this there. This is immediate from the defining properties of a connection and Lemma~\ref{lemm::properties_connection}. For instance,
\[
L_{x}c_{A}(A_{\univ})(y)=L_{x}A(y)=A([x,y])=c_{A}(A_{\univ})([x,y])=c_{A}(L_{x}A_{\univ})(y),
\]
and
\[
L_{x}c_{A}(F_{\univ})(y)=L_{x}F^{A}(y)=F^{A}([x,y])=c_{A}(F_{\univ})([x,y])=c_{A}(L_{x}F_{\univ})(y).
\]
For the differentials,
\[
c_{A}(\diff^{W}A_{\univ})=c_{A}\Bigl(F_{\univ}-\tfrac12[A_{\univ},A_{\univ}]\Bigr)=F^{A}-\tfrac12[A,A]=\diff A=\diff c_{A}(A_{\univ}),
\]
and
\[
c_{A}(\diff^{W}F_{\univ})=-c_{A}([A_{\univ},F_{\univ}])=-[c_{A}(A_{\univ}),c_{A}(F_{\univ})]=-[A,F^{A}]=\diff F^{A}=\diff c_{A}(F_{\univ}).
\]
The compatibility with contractions is verified similarly.
\end{proof}

\subsubsection{Chern--Simons Elements and the Quadratic Case} 
Let $\widehat{\calA}$ be a completed $\gg$-differential algebra, let $A$ be a connection on $\widehat{\calA}$, and let $F^{A}$ be its curvature. For $m\geq 1$, let $\widehat S^{m}(\gg^{\ast})^{\gg}$ denote the completed space of invariant symmetric $m$-linear forms on $\gg$. For $P\in\widehat S^{m}(\gg^{\ast})^{\gg}$, the element
$
P(F^{A})\in\widehat{\calA}_{\bar{0}}
$
is called the \emph{characteristic element} associated with $P$ and $A$. Explicitly, if $(e_{a})$ is a homogeneous basis of $\gg$ with $B$-dual basis $(e^{a})$, then
\begin{equation}\label{eq::characteristic_element}
P(F^{A})=\sum_{a_{1},\ldots,a_{m}}\prod_{i<j}\varepsilon(|e_{a_j}|,|e_{a_i}|)P(e_{a_{1}},\ldots,e_{a_{m}})\,F^{A}(e^{a_{1}})\cdots F^{A}(e^{a_{m}}),
\end{equation}
where the sum is taken in the completed algebra $\widehat{\calA}$.

\begin{lemma}
The characteristic element $P(F^{A})$ satisfies:
\begin{enumerate}
\item[a)] $\iota_{x}(P(F^{A}))=0$ and $L_{x}P(F^{A})=0$ for all $x\in\gg$.
\item[b)] $\diff P(F^{A})=0$.
\end{enumerate}
\end{lemma}
\begin{proof}
For a), $\iota_{x}(P(F^{A}))=0$ since $\iota_{x}F^{A}=0$ for all $x\in\gg$, and $L_{x}P(F^{A})=0$ since $P$ is $\gg$-invariant. For b), one uses $\diff F^{A}+[A,F^{A}]=0$ from Lemma~\ref{lemm::properties_connection} and the invariance of $P$.
\end{proof}

\begin{definition}
Let $P\in\widehat{S}^{m}(\gg^{\ast})^{\gg}$, and let $A$ be a connection on a completed $\gg$-differential algebra $\widehat{\calA}$. A \emph{Chern--Simons element} for $P$ and $A$ is an element $\operatorname{CS}_{P}(A)\in\widehat{\calA}_{\bar1}$ such that
\begin{equation}\label{eq:general-CS-definition}
\diff\operatorname{CS}_{P}(A)=P(F^{A}).
\end{equation}
\end{definition}

\begin{remark}
The Chern--Simons element $\CS_{P}(A)$ is in general not unique: if $\operatorname{CS}_{P}(A)$ and $\operatorname{CS}'_{P}(A)$ both satisfy~\eqref{eq:general-CS-definition}, then
\[
\diff\bigl(\operatorname{CS}_{P}(A)-\operatorname{CS}'_{P}(A)\bigr)=0.
\]
\end{remark}

\begin{proposition}
Let $A$ be a connection on a completed $\gg$-differential algebra $\widehat{\calA}$, and let
$
c_A:\widehat{W}_{\varepsilon}(\gg)\longrightarrow \widehat{\calA}
$
be the associated Chern--Weil homomorphism. Let $P\in \widehat S^{m}(\gg^{\ast})^{\gg}$. Assume that there exists a Chern--Simons element for $P$ and $\Auniv$ in $\widehat{W}_{\varepsilon}(\gg)$. Then
\[
\operatorname{CS}_{P}(A)
\coloneqq
c_A\bigl(\operatorname{CS}_{P}(A_{\univ})\bigr)
\]
is a Chern--Simons element in $\widehat{\calA}$ for $P$ and $A$.
\end{proposition}

\begin{proof}
Since $c_A$ is a morphism of differential algebras, one has
$
\diff c_A=c_A \dw.
$
Therefore
\[
\begin{aligned}
\diff \operatorname{CS}_{P}(A)
=
\diff c_A\bigl(\operatorname{CS}_{P}(A_{\univ})\bigr)=
c_A\bigl(\dw\operatorname{CS}_{P}(A_{\univ})\bigr)=
c_A\bigl(P(F_{\univ})\bigr).
\end{aligned}
\]
By the defining property of the Chern--Weil homomorphism,
$
c_A(F_{\univ})=F^A.
$
Hence
$
c_A\bigl(P(F_{\univ})\bigr)=P(F^A)
$ and the statement follows.
\end{proof}

Since $\gg$ is quadratic the invariant $\varepsilon$-symmetric form $B$ induces a polynomial $P_{B}\in \widehat{S}^{2}_{\varepsilon}(\gg^{\ast})$. We show that the Chern--Simons element for $P_{B}$ and $\Auniv$ is 
\begin{equation}
 \calD \coloneqq \sum_{a}e^{a}\otimes e_{a}+1\otimes \phi \in \widehat{W}_{\varepsilon}(\gg).
\end{equation}
This element is central in our article since its quantization yields the cubic Dirac operator. The following properties are direct consequences of Lemma~\ref{lemm::form_Ap} and Lemma~\ref{lemm::properties_phi}.

\begin{lemma}
\begin{enumerate}
\item[a)] $\calD$ is odd, that is, $p(\calD)=\bar{1}$, and it has $\Gamma$-degree $\ng$.
 \item[b)] $\calD$ is $\gg$-invariant, that is, $L_{x}\calD=0$ for all $x \in \gg$.
 \item[c)] $\iota_{x}\calD=x\otimes 1-\tfrac{1}{2}(1\otimes \uplambda(\ad_{x}))$.
 \item[d)] It squares to the trivial element, that is, $\calD^{2}=0$. 
\end{enumerate}
\end{lemma}

To prove that $\calD$ is a Chern--Simons element, we need the following lemma which follows by definition and the fact that $B$ is a nondegenerate $\varepsilon$-symmetric invariant form on $\gg$.

\begin{lemma}
 One has for homogeneous $A,B,C : \gg \to \calA$
 \[
 P_{B}([A,B],C)=P_{B}(A,[B,C]).
 \]
\end{lemma}

For the following proposition and its proof, we use an equivalent normalization of the Weil generators. Let \(S\) be the automorphism of \(\widehat W_\varepsilon(\mathfrak g)\) defined by \(S(x\otimes1)=2x\otimes1\) and \(S(1\otimes x)=1\otimes x\). We transport the differential to \(S d^W S^{-1}\) and the universal curvature to \(S F_{\mathrm{univ}}\), retaining the notation \(d^W\) and \(F_{\mathrm{univ}}\). Thus \(F_{\mathrm{univ}}(x)=2x\otimes1\). The connection, contractions, and Lie derivatives are unchanged. We also normalize the invariant quadratic polynomial by \(P_B(z)=\tfrac12B(z,z)\). With these conventions, \(P_B(A_{\mathrm{univ}},F_{\mathrm{univ}})=\sum_a e^a\otimes e_a\) and \(P_B(A_{\mathrm{univ}},[A_{\mathrm{univ}},A_{\mathrm{univ}}])=-6\phi\).

\begin{proposition} \label{prop::D_is_CS_element}
 A Chern--Simons element for the quadratic polynomial $P_{B}$ induced by the quadratic form $B$ and the universal connection $A_{\univ}$ of the Weil algebra is $\calD$, that is
 $$
 \CS_{P_{B}}(\Auniv)=\calD.
 $$
\end{proposition}

\begin{proof}
It is enough to prove that
$
\dw\calD=P_{B}(F_{\univ}).
$
One has
\[
\calD=P_{B}(A_{\univ},F_{\univ})-\tfrac16P_{B}(A_{\univ},[A_{\univ},A_{\univ}]).
\]
First,
\[
\begin{aligned}
\dw P_{B}(A_{\univ},F_{\univ})
&=P_{B}(\dw A_{\univ},F_{\univ})-P_{B}(A_{\univ},\dw F_{\univ})\\
&=P_{B}(F_{\univ},F_{\univ})-\tfrac12P_{B}([A_{\univ},A_{\univ}],F_{\univ})+P_{B}(A_{\univ},[A_{\univ},F_{\univ}])\\
&=P_{B}(F_{\univ},F_{\univ})+\tfrac12P_{B}([A_{\univ},A_{\univ}],F_{\univ}),
\end{aligned}
\]
where invariance of $B$ gives
$
P_{B}(A_{\univ},[A_{\univ},F_{\univ}])=P_{B}([A_{\univ},A_{\univ}],F_{\univ}).
$
Further,
\[
\begin{aligned}
\dw P_{B}(A_{\univ},[A_{\univ},A_{\univ}])
&=P_{B}(\dw A_{\univ},[A_{\univ},A_{\univ}])-P_{B}(A_{\univ},\dw[A_{\univ},A_{\univ}])\\
&=P_{B}(\dw A_{\univ},[A_{\univ},A_{\univ}])-P_{B}(A_{\univ},[\dw A_{\univ},A_{\univ}])\\ & \quad +P_{B}(A_{\univ},[A_{\univ},\dw A_{\univ}])\\
&=3P_{B}(\dw A_{\univ},[A_{\univ},A_{\univ}])\\
&=3P_{B}(F_{\univ},[A_{\univ},A_{\univ}])-\tfrac32P_{B}([A_{\univ},A_{\univ}],[A_{\univ},A_{\univ}])\\
&=3P_{B}(F_{\univ},[A_{\univ},A_{\univ}])-\tfrac32P_{B}(A_{\univ},[A_{\univ},[A_{\univ},A_{\univ}]])\\
&=3P_{B}(F_{\univ},[A_{\univ},A_{\univ}]),
\end{aligned}
\]
since $[A_{\univ},[A_{\univ},A_{\univ}]]=0$ by Lemma~\ref{lemm::properties_connection}. Therefore
\[
\dw\calD=P_{B}(F_{\univ},F_{\univ})=P_{B}(F_{\univ}).\qedhere
\]
\end{proof}

\subsection{The Cubic Dirac Operator} 
The \emph{cubic Dirac operator} is a distinguished element of the completed quantum Weil algebra $\widehat{\mathcal{W}}_{\varepsilon}(\gg)$ characterized by a Parthasarathy–type formula \cite{parthasarathy1972dirac}. We keep the conventions of Section~\ref{subsec::conventions}. Let $(e_a)$ be a homogeneous basis of $\gg$ with $B$-dual basis $(e^a)$, so that $B(e_a,e^b)=\delta_{ab}$. Following the ideas of Meinrenken \cite{Meinrenken}, Kostant \cite{Kostant_cubic_Dirac}, and Kang–Chen \cite{Dirac_quadratic}, we define the cubic Dirac operator as the quantization of $\calD = \sum_{a}e^{a}\otimes e_{a}+1\otimes\phi$, where $\phi$ is the structure constants tensor from Section~\ref{subsec::structure_constants_tensor}.

\begin{definition}
 The \emph{cubic Dirac operator} for $\gg$ is 
 \[
 \Dirac'_{\gg} \coloneqq \widehat{\calQ}(\calD) = \sum_{a} e^{a}\otimes e_{a} + 1 \otimes \phi' \in \widehat{\mathcal{W}}_{\varepsilon}(\gg).
 \]
 When no ambiguity arises, it will be denoted simply by $\Dirac'$.
\end{definition}

\begin{remark} \label{rmk::Dirac_basis_independent}
The definition of the cubic Dirac operator $\Dirac'$ is independent of the choice of a basis $(e_{a})$ with $B$-dual basis $(e^{a})$. 
Let $(f_{b})$ be another homogeneous basis with $B$-dual basis $(f^{b})$. 
Then, by Lemma~\ref{lemm::basis_description},
\[
\begin{aligned}
\sum_{a} e^{a}\otimes e_{a}
 &= \sum_{a,b,c} B(f_{b},e^{a})B(e_{a},f^{c}) f^{b}\otimes f_{c}
 = \sum_{b,c} B\!\left(\sum_{a} B(f_{b},e^{a})e_{a}, f^{c}\right) f^{b}\otimes f_{c}
 \\&= \sum_{b,c} B(f_{b},f^{c}) f^{b}\otimes f_{c}
 = \sum_{b} f^{b}\otimes f_{b},
 \end{aligned}
\]
since $B(f_{b},f^{c})=\delta_{b,c}$.
\end{remark}

\begin{lemma} \label{lemm::properties_Dirac'} 
The cubic Dirac operator $\Dirac'$ satisfies the following properties:
\begin{enumerate}
\item[a)] $\Dirac'\in \widehat{\WW}_{\varepsilon}(\gg)$ is odd and has $\ZZ\times \Gamma$-degree $(0,\ng)$. In particular,
\[
(\Dirac')^{2}=\tfrac{1}{2}[\Dirac',\Dirac']_{\widehat{\WW}}.
\]
\item[b)] $L_{x}\Dirac'=1\otimes \iota_{x}\widehat{q}(\KPsharp) =1\otimes\Psi_{\operatorname{KP}}(x)$.
\item[c)] $\iota_{x}\Dirac'=\gamma^{\WW}(x)$, where $\gamma^{\WW}(x)=x\otimes 1-\tfrac{1}{2}(1\otimes \widehat{\gamma}(x))$.
\item[d)] $\iota_{x}((\Dirac')^{2}) = [\iota_{x}\Dirac',\Dirac']_{\widehat{\WW}}$.
\end{enumerate}
\end{lemma}

\begin{proof} a) follows by definition and b) follows from Lemma~\ref{lemm::komische_Beziehungen} since $\sum_{a}e^{a}\otimes e_{a}$ is $\gg$-invariant by Lemma~\ref{lemm::form_Ap}. For~c) we compute
\[
\begin{aligned}
\iota_{x}\Dirac' &= \sum_{a}\varepsilon(x,e^{a})B(x,e_{a})e^{a}\otimes 1 + 1\otimes \iota_{x}\phi' = \sum_{a}\varepsilon(e_{a})B(x,e_{a})e^{a}\otimes1 + \widehat{q}(\iota_{x}\phi) \\ &= x\otimes1 - 1\otimes \tfrac{1}{2}\widehat{q}(\uplambda(\ad_{x})) = x\otimes 1-\tfrac{1}{2}(1\otimes \widehat{\gamma}(x)),
\end{aligned}
\]
where we used $2\iota_{x}\phi=-\uplambda(\ad_{x})$ (Lemma~\ref{lemm::properties_phi}), and $B(x,e_{a})$ is zero unless $\vert x\vert +\vert e_{a}\vert=\ng$ so that $\varepsilon(x,e^{a})=\varepsilon(-e_{a},e^{a})=\varepsilon(e^{a})=\varepsilon(e_{a})$. This proves~c). Finally, d) follows since $\Dirac'$ has bidegree $(0,\ng)$ and $\iota_{x}$ is an odd $\varepsilon$-derivation.
\end{proof}

\begin{theorem} \label{thm::square_uncorrected}
 The square of the cubic Dirac operator is 
 \[
 (\Dirac')^{2} = \Omega_{\gg}'\otimes 1 + 1\otimes \widehat{q}(\KPsharp) + \tfrac{1}{24}\etr(\ad_{\gg_{0}}(\Omega_{\gg_{0}}))(1\otimes 1).
 \]
\end{theorem}
\begin{proof}
We first show that $(\Dirac')^{2}-1\otimes \widehat{q}(\KPsharp)\in\widehat{\UE}_{\varepsilon}(\gg)$. It suffices to verify that $\iota_{x}((\Dirac')^{2}-1\otimes \widehat{q}(\KPsharp))=0$ for all $x\in\gg$. By Lemma~\ref{lemm::properties_Dirac'} we compute
\[
\begin{aligned}
\iota_{x}\bigl((\Dirac')^{2}-1\otimes \widehat{q}(\KPsharp)\bigr)
&=\iota_{x}(\Dirac')^{2}-1\otimes\iota_{x}\widehat{q}(\KPsharp)
=[\iota_{x}\Dirac',\Dirac']_{\WW}-1\otimes \iota_{x}\widehat{q}(\KPsharp)\\
&=[\gamma^{\WW}(x),\Dirac']_{\widehat{\WW}}-1\otimes \iota_{x}\widehat{q}(\KPsharp)
= L_{x}\Dirac'-1\otimes \iota_{x}\widehat{q}(\KPsharp))=0,
\end{aligned}
\]
where we use a) and b) and from Lemma~\ref{lemm::contraction_and_Lie_derivative_in_Weil}.

Next, we determine $(\Dirac')^{2}-1\otimes \widehat{q}(\KPsharp)\in\widehat{\UE}_{\varepsilon}(\gg)$. The normal-order scalar projection is the linear map
\[
\epsilon_{\mathrm{NO}}
:=
\operatorname{pr}_{\wedge^0}\circ\widehat q^{-1}
\colon
\widehat{\operatorname{Cl}}_\varepsilon(\gg)
\longrightarrow\mathbb C,
\]
where $\operatorname{pr}_{\wedge^0}$ denotes the projection onto the
exterior-degree-zero component. Equivalently, it is determined by
$
\epsilon_{\mathrm{NO}}(1)=1$ and
$\epsilon_{\mathrm{NO}}\bigl(\widehat q(\omega)\bigr)=0
$
for every $\omega\in
\widehat{\bigwedge}^{>0}_\varepsilon(\gg)$. Thus it is the
projection associated with the decomposition
$
\widehat{\operatorname{Cl}}_\varepsilon(\gg)
=
\mathbb C1\oplus
\widehat q\left(
\widehat{\bigwedge}^{>0}_\varepsilon(\gg)
\right).
$ For homogeneous $x,y\in\gg$, Example~\ref{ex::quantization_map} gives
\[
xy\equiv
\kappa(x,y)1
\pmod{\ker\epsilon_{\mathrm{NO}}},
\]
where
$
\kappa(x,y)
:=
B(x,y)+B(\pi_+x,y)
-\varepsilon(x,y)B(\pi_+y,x)
$. Set
$
\pi:=\operatorname{id}_{\widehat U_\varepsilon(\gg)}
\otimes\epsilon_{\mathrm{NO}}.$
Since the mixed terms in $(D')^2$ have positive normal Clifford degree,
Proposition~\ref{prop::phi_squared} gives
\begin{equation*}
\begin{aligned}
(\Dirac')^2 
&\equiv
\sum_{a,b}
\varepsilon(e_a,e^b)\kappa(e_a,e_b)e^ae^b\otimes 1
+\tfrac{1}{24}
\operatorname{tr}_{\varepsilon}
\bigl(\operatorname{ad}_{\gg_0}(\Omega_{\gg_0})\bigr)
(1\otimes 1)
\pmod{\ker\pi}
\\
&\equiv
\Omega'_{\gg}\otimes 1
+\tfrac{1}{24}
\operatorname{tr}_{\varepsilon}
\bigl(\operatorname{ad}_{\gg_0}(\Omega_{\gg_0})\bigr)
(1\otimes 1)
\pmod{\ker\pi}.
\end{aligned}
\end{equation*}
The second congruence follows by choosing a homogeneous basis adapted to
$\gg=\gg_-\oplus\gg_+$ and applying the definition
of the normal-ordered PBW map componentwise. Since $\pi$ restricts to the identity on $\widehat{\UE}_{\varepsilon}(\gg)\otimes1$ and $1\otimes\widehat q(\psi_{\mathrm{KP}}^\sharp)\in\ker\pi$, the preceding congruence implies the claimed identity.
\end{proof}

Theorem~\ref{thm::square_uncorrected} motivates the following definition.

\begin{definition}
Assume $\gg$ has trivial Kac–Peterson class $\KP=\operatorname{d}\!\rho$ for some $\rho\in\gg^{\ast}_{(0,\ng)}$. Then the \emph{corrected cubic Dirac operator} of $\gg$ is
\[
\Dirac\coloneqq \Dirac'+1\otimes\rho^{\sharp}.
\]
\end{definition}

The corrected cubic Dirac operator is $\gg$-invariant and it has a nice square.

\begin{corollary} \label{cor::D_square}
 Assume $\gg$ has trivial Kac–Peterson class $\KP = \operatorname{d}\!\rho$ for some $\rho \in \gg^{\ast}_{(0,\ng)}$. Then $\Dirac$ is $\gg$-invariant and it has square 
 \[
 \Dirac^{2} = \Omega_{\gg}\otimes 1 + \tfrac{1}{24}\etr(\ad_{\gg_{0}}(\Omega_{\gg_{0}}))(1\otimes 1) + B(\rho^{\sharp}, \rho^{\sharp})(1\otimes 1),
 \]
 where $\Omega_{\gg}$ is the corrected quadratic Casimir element. In particular, $\diff^{\WW} \coloneqq [\Dirac,\cdot]_{\WW}$ defines a differential on $\widehat{W}_{\varepsilon}(\gg)$, that is, it is odd and squares to zero.
\end{corollary}

\begin{proof}
 Since $\KP=d\rho$ one has $\Psi_{\operatorname{KP}}(x)=[\rho^{\sharp},x]=-[x,\rho^{\sharp}]$. By Lemma~\ref{lemm::properties_Dirac'} it follows that $L_{x}\Dirac=0$. 

 Next, since both $\Dirac'$ and $1\otimes x$
are odd and $\Dirac'$ has $\Gamma$-degree $\ng$, we obtain
\[
\begin{split}
\tfrac{1}{2}[1\otimes x,\Dirac']_{\widehat{\WW}}
=\,\iota_{x}\Dirac'
=\iota_{x}(\sum_{a}e^{a}\otimes e_{a})+1\otimes \widehat{q}(\iota_{x}\phi)=(x\otimes1-\tfrac{1}{2}\bigl(1\otimes \widehat{q}(\uplambda(\ad_{x}))\bigr)=\gamma^{\WW}(x)
\end{split}
\]
and
\[
\begin{aligned}
 \Dirac^{2}&=\tfrac{1}{2}[\Dirac'+1\otimes \rho^{\sharp},\Dirac'+1\otimes \rho^{\sharp}]_{\widehat{\WW}} = (\Dirac')^{2}+1\otimes (\rho^{\sharp})^{2}+[\Dirac',1\otimes\rho^{\sharp}] \\ &=(\Dirac')^{2}+B(\rho^{\sharp},\rho^{\sharp})(1\otimes 1)+2\rho^{\sharp}\otimes 1-1\otimes \widehat{q}(\uplambda(\ad_{\rho^{\sharp}}))
\end{aligned}
\]
 The formula for the square follows by Theorem~\ref{thm::square_uncorrected} since 
$
\widehat{q}(\KPsharp)=\widehat{q}(\uplambda(\ad_{\rho^{\sharp}}))$ by assumption.
\end{proof}

\subsubsection{Some Comments on relations to the BV-BFV formalism}
 The BV-BFV\footnote{Named after Batalin, Fradkin and Vilkovisky \cite{batalin1977relativistic,batalin1981gauge,batalin1983generalized}.} formalism is a tool to treat gauge theories on manifolds with boundary, both at the classical \cite{Cattaneo2014, Mnev2019} and quantum \cite{Cattaneo2018} level. A central object in this formalism are the classical and quantum BFV charges associated with the boundary,\footnote{See also \cite{Schtz2008,Stasheff1997}. } denoted $S^\partial$ and $\hat{S}^\partial$, respectively. In the BV--BFV formulation of $1d$ Chern--Simons theory with coefficient Lie algebra $\mathfrak{g}$ (potentially with Wilson lines) these can be naturally identified with the classical and quantum cubic Dirac operators for $\mathfrak{g}$, see \cite{Mnev1,Mnev2}. 
 While the authors only treat finite-dimensional quadratic Lie algebras, our construction implies the formulation for general $\ZZ$-graded (infinite-dimensional) color Lie algebras. We expect this to be relevant in the BV-BFV formulation of two-dimensional (super)conformal field theories. More generally, one can think of higher-dimensional Chern-Simons theories on cylinders as one-dimensional Chern-Simons theories with target in an infinite-dimensional differential graded Lie algebra or $L_\infty$-algebra, see e.g. \cite{Gwilliam2014} for background on $L_\infty$ Chern-Simons theory and \cite{CMW_CS}, \cite{Cattaneo2022} for the case with boundary. Generalizing the notion of cubic Dirac operator further to such cases and investigating the relationship to the classical and quantum BFV charges are the subject of ongoing research. 

\subsection{Relative Cubic Dirac Operator} \label{subsec::relative_cubic_Dirac_operator}
We modify the construction of the cubic Dirac operator to incorporate more color Lie subalgebras of $\gg$.

Let $\ll$ be a $\ZZ$-graded color Lie subalgebra of $\gg$, that is, $\ll$ is a color Lie subalgebra satisfying $\ll_{i}\subset\gg_{i}$ for all $i\in\ZZ$. Suppose that $B$ restricts to an even, non-degenerate, $\varepsilon$-symmetric bilinear form on $\ll$, denoted by $B_{\ll}$ in the following. Set $\ss\coloneqq \ll^{\perp}$ with respect to $B$, and assume $[\ll, \ss] \subset \ss$. Then one has the orthogonal decomposition
\begin{equation}
\gg=\ll\oplus\ss.
\end{equation}
For example, if $\gg$ is a $\ZZ$-graded quadratic Lie superalgebra, one can take $\ll \coloneqq \even$ and $\ss \coloneqq \odd$. By definition of $\ss$, the bilinear form $B$ restricts to a nondegenerate bilinear form on $\ss$, denoted by $B_{\ss}$. In particular, we can define $\varepsilon$-orthogonal algebra $\mathfrak{so}_{\varepsilon}(\ll) \coloneqq \mathfrak{so}_{\varepsilon}(\ll; B_{\ll})$ and $\mathfrak{so}_{\varepsilon}(\ss)\coloneqq \mathfrak{so}_{\varepsilon}(\ss; B_{\ss})$, as well as the $\varepsilon$-Clifford algebras $\Cl_{\varepsilon}(\ll) \coloneqq \Cl_{\varepsilon}(\ll;B_{\ll})$ and $\Cl_{\varepsilon}(\ss) \coloneqq \Cl_{\varepsilon}(\ss; B_{\ss})$, respectively.

Each of the color Lie algebras $\gg$ and $\ll$ has a Kac--Peterson class, denoted by $\psi_{\operatorname{KP}}$ and $\psi_{\operatorname{KP}}^{\ll}$, respectively. We assume that both are trivial, that is, that there exist $\rho\in(\gg^{\ast})_{(0,\ng)}$ and $\rho_{\ll}\in(\ll^{\ast})_{(0,\ng)}$ such that $d\rho=\psi_{\operatorname{KP}}$ and $d\rho_{\ll}=\psi_{\operatorname{KP}}^{\ll}$. We denote the corresponding corrected Casimir elements by $\Omega_{\gg}$ and $\Omega_{\ll}$, and the associated corrected cubic Dirac operators by $\Dirac_{\gg}\in\widehat{\mathcal W}_{\varepsilon}(\gg)$ and $\Dirac_{\ll}\in\widehat{\mathcal W}_{\varepsilon}(\ll)$.

To define the relative cubic Dirac operator $\Dirac_{\gg,\ll}$, we embed $\widehat{\mathcal W}_{\varepsilon}(\ll)$ into $\widehat{\mathcal W}_{\varepsilon}(\gg)$ via the adjoint action of $\ll$ on $\ss$. For $x\in\ll$, the operator $\ad_{x}\in\widehat{\mathfrak{so}}_{\varepsilon}(\gg)$ decomposes as $\ad_{x}=\ad_{x}^{\ll}+\ad_{x}^{\ss}$ with $\ad_{x}^{\ll}\in\widehat{\mathfrak{so}}_{\varepsilon}(\ll)$ and $\ad_{x}^{\ss}\in\widehat{\mathfrak{so}}_{\varepsilon}(\ss)$. We set $\widehat{\gamma}_{\ss}(x)\coloneqq\widehat q(\uplambda(\ad_{x}^{\ss}))\in\widehat{\Cl}_{\varepsilon}(\ss)$. Since $B$ is invariant, $\ad_{x}$ preserves $\ss$ for every $x\in\ll$. By Proposition~\ref{prop::commutator_and_quantization}, the $\varepsilon$-commutator $[\widehat{\gamma}_{\ss}(x),\widehat{\gamma}_{\ss}(y)]_{\widehat{\Cl}_{\varepsilon}}$ defines a $2$-cocycle $\psi_{\operatorname{KP}}^{\ss}\in\widehat{\bigwedge}_{\varepsilon}^{2}(\ll^{\ast})$, and this yields a color Lie algebra homomorphism provided $\psi_{\operatorname{KP}}^{\ss}=-d\rho_{\ss}$, where $\rho_{\ss}\coloneqq \rho_{\ll}-\rho|_{\ll}\in(\ll^{\ast})_{(0,\ng)}$. 

\begin{lemma} \label{lemm::gamma_s}
 Assume $\psi_{\operatorname{KP}}^{\ss} = -\operatorname{d}\!\rho_{\ss}$ for some $\rho_{\ss} \in (\ll^{\ast})_{(0,\ng)}$. Then the map 
 \[
 \gamma'_{\ss}: \ll \to \widehat{\Cl}_{\varepsilon}(\ss), \qquad \gamma_{\ss}'(x) \coloneqq -\tfrac{1}{2}\widehat{\gamma}_{\ss}(x) - \rho_{\ss}(x)
 \]
 is a color Lie algebra homomorphism.
\end{lemma}

\begin{proof}
By Proposition~\ref{prop::commutator_and_quantization} one has $[\widehat{\gamma}_{\ss}(x),\widehat{\gamma}_{\ss}(y)]_{\widehat{\Cl}_{\varepsilon}}=-2\widehat{\gamma}_{\ss}([x,y])+4\psi_{\operatorname{KP}}^{\ss}(x,y)$, and by assumption
\[
\psi^{\ss}_{\operatorname{KP}}(x,y) = B(\Psi^{\ss}_{\operatorname{KP}}(x),y)=B([x,\rho_{\ss}^{\sharp}],y)=-B(\rho_{\ss}^{\sharp},[x,y])=-\rho_{\ss}([x,y])=-d\rho_{\ss}(x,y).
\]
Consequently
\[
[\gamma'_{\ss}(x),\gamma'_{\ss}(y)]_{\widehat{\Cl}_{\varepsilon}}=\tfrac{1}{4}[\widehat{\gamma}_{\ss}(x),\widehat{\gamma}_{\ss}(y)]_{\widehat{\Cl}_{\varepsilon}}=-\tfrac{1}{2}\widehat{\gamma}_{\ss}([x,y])+\psi^{\ss}_{\operatorname{KP}}(x,y)
= (-\tfrac{1}{2}\widehat{\gamma}_{\ss}-\rho_{\ss})([x,y])=\gamma_{\ss}'([x,y]).
\]
Since $\gamma'_{\ss}$, $\widehat q$, and $\uplambda$ preserve the $\Gamma$-degree, and $\rho_{\ss}$ has degree $\ng$, the claim follows.
\end{proof}

\begin{remark}\label{rmk::generator_adjoint_action}
By Proposition~\ref{prop::Lie_derivative_as_commutator}, the action of $-\tfrac{1}{2}\widehat{\gamma}_{\ss}(x)$ generates the adjoint action of $\ll$ on $\widehat{\Cl}_{\varepsilon}(\ss)$, that is, for all $x \in \ll$,
\[
L_{x} = -\tfrac{1}{2}[\widehat{\gamma}_{\ss}(x), \cdot]_{\widehat{\Cl}_{\varepsilon}}.
\]
\end{remark}

In what follows, we assume $\psi_{\operatorname{KP}}^{\ss} = -\operatorname{d}\!\rho_{\ss}$ for some $\rho_{\ss} \in (\ll^{\ast})_{(0,\ng)}$. We define a map $j : \mathcal{W}_{\varepsilon}(\ll) \to \widehat{\mathcal{W}}_{\varepsilon}(\gg)$, given on generators by 
\begin{equation}
\begin{aligned}
 j(1\otimes x) = 1\otimes x, \qquad j(x\otimes 1) = x\otimes 1 + 1 \otimes \gamma'_{\ss}(x).
 \end{aligned}
\end{equation}
In particular, $j$ defines a diagonal map of $\ll$ into $\widehat{\mathcal{W}}_{\varepsilon}(\gg)$.
\begin{proposition} \label{lemm::properties_map_j}
 The following assertions hold: \begin{enumerate}
 \item[a)] The map $j : \mathcal{W}_{\varepsilon}(\ll) \to \widehat{\mathcal{W}}_{\varepsilon}(\gg)$ is a homomorphism of associative color superalgebras. In particular, it is a color Lie superalgebra homomorphism. 
 \item[b)] The map $j$ extends uniquely to a homomorphism of associative color superalgebras $j: \widehat{\mathcal{W}}_{\varepsilon}(\ll) \to \widehat{\mathcal{W}}_{\varepsilon}(\gg)$. In particular, it is a color Lie superalgebra homomorphism. 
 \item[c)] $j: \widehat{\mathcal{W}}_{\varepsilon}(\ll) \to \widehat{\mathcal{W}}_{\varepsilon}(\gg)$ intertwines with contractions and Lie derivatives, that is, 
 \[
 \iota_{x}\circ j = j \circ \iota_{x}, \qquad L_{x}\circ j = j \circ L_{x}, \qquad x\in \ll.
 \]
 \end{enumerate}
\end{proposition}

\begin{proof} a) This is a direct consequence of Lemma~\ref{lemm::gamma_s} and the definition of $j$.

b) The statement follows if we show for all $i >0$ that $j(\ll_{i}) \subset \prod_{r\geq 0} \WW_{\varepsilon}(\gg_{-})_{-r} \WW_{\varepsilon}(\gg_{+})_{i+r} \subset \cWeil$ for all $i>0$ and $\gamma(\ll_{j}) \subset \prod_{r\geq 0}\WW_{\varepsilon}(\gg_{-})_{j-r}\WW_{\varepsilon}(\gg_{+})_{r}$ for $j<0$, where $\ll_{k}\subset \gg_{k}$ for all $k \in \ZZ$. Then inductively we conclude
\[
j(\WW_{\varepsilon}(\ll_{+})_{i}) \subset \prod_{r\geq 0} \WW_{\varepsilon}(\gg_{-})_{-r}\WW_{\varepsilon}(\gg_{+})_{i+r}, \qquad j(\WW_{\varepsilon}(\ll_{-})_{j}) \subset \prod_{r\geq 0}\WW_{\varepsilon}(\gg_{-})_{j-r}\WW_{\varepsilon}(\gg_{+})_{r}
\]
such that 
\[
j(\WW_{\varepsilon}(\ll_{-})_{-r}\WW_{\varepsilon}(\ll_{+})_{i+r}) \subset \coprod_{k\geq0} \WW_{\varepsilon}(\gg_{-})_{-r-k}\WW_{\varepsilon}(\gg_{+})_{i+r+k}.
\]
Consequently, the summation over all $r \geq 0$ defines a well-defined map $\widehat{\WW}_{\varepsilon}(\ll)_{i} \to \widehat{\WW}_{\varepsilon}(\gg)_{i}$. 

It suffices to show for all $i >0$ that $j(\ll_{i}) \subset \prod_{r\geq 0} \WW_{\varepsilon}(\gg_{-})_{-r} \WW_{\varepsilon}(\gg_{+})_{i+r} \subset \cWeil$, as the other case is analogous. By definition, the statement reduces to prove that $\widehat{\gamma}_{\ss}(\ll_{i}) \subset \prod_{r\geq 0}\Cl_{\varepsilon}(\gg_{-})_{-r}\Cl_{\varepsilon}(\gg_{+})_{i+r}$. Let $x \in \ll_{i}$ for $i>0$. Then $\rho_{\ss}(x)=0$, $B([x,e_{a}],e^{a})=0$ for all $a$ and consequently
\[
\begin{aligned}
 \widehat{\gamma}_{\ss}(x)= -\tfrac{1}{2}\sum_{a}[x,e^{a}]e_{a} = \tfrac{1}{2}\sum_{a}\varepsilon(x+e^{a},e_{a})e_{a}[x,e^{a}] = -\tfrac{1}{2}\sum_{a}\varepsilon(e_{a})[x,e_{a}]e^{a}=\tfrac12\sum_{a} \varepsilon(e_{a},x)e^{a}[x,e_{a}].
\end{aligned}
\]
We decompose 
\[
\begin{aligned}
 2\widehat{\gamma}_{\ss}(x) = \sum_{a\, : \, e_{a}\in \gg_{+}}\varepsilon(e_{a},x)e^{a}[x,e_{a}] + \sum_{a\, :\, e_{a}\in \gg_{-}}\varepsilon(e_{a},x)e^{a}[x,e_{a}]+ \sum_{a\, : \, e_{a}\in \gg_{0}}\varepsilon(e_{a},x)e^{a}[x,e_{a}]
\end{aligned}
\]
The first and the third sum lie in $\prod_{r\geq 0}\Cl_{\varepsilon}(\gg_{-})_{-r}\Cl_{\varepsilon}(\gg_{+})_{i+r}$ and it remains to consider the second sum. One has 
\[
\sum_{a \, : \, e_{a}\in \gg_{-}} \varepsilon(e_{a},x)e^{a}[x,e_{a}]= \sum_{a \, : \, e_{a}\in \gg_{+}}\varepsilon(e^{a},x)e_{a}[x,e^{a}] = - \sum_{a \, : \, e_{a}\in \gg_{+}} \varepsilon(e^{a})[x,e^{a}]e_{a}
\]
The second sum lies in the space for $e_{b} \in \gg_{j}$ with $j\geq i$. If $e_{b} \in \gg_{j}$ for $j<i$, we have $[x,e^{b}] \in \gg_{i-j}$, and $e_{b}[x,e^{b}]\in \Cl_{\varepsilon}(\gg_{+})_{i}$. This finishes the proof.

 c) We prove the statement for Lie derivatives $L_{x}$ for $x \in \ll$. It is enough to prove the statement on the componentwise generators. One has 
 \[
 (L_{x}\circ j)(1\otimes y)= L_{x}(1\otimes y)= 1\otimes L_{x}(y)=(j \circ L_{x})(1\otimes y)
 \]
 and using that $\gamma'_{\ss}$ is a Lie superalgebra homomorphism
 \[
 (L_{x}\circ j)(y\otimes 1)=L_{x}(y\otimes 1+1\otimes \gamma'_{\ss}(y)) = L_{x}(y)\otimes 1 + 1 \otimes \gamma'_{\ss}(L_{x}(y)) = (j\circ L_{x})(y\otimes 1).
 \]

 That $j$ intertwines with contractions is another direct calculation and will be omitted.
\end{proof}

\begin{definition}
 The \emph{relative cubic Dirac} operator is 
 \[
\Dirac_{\gg,\ll} \coloneqq \Dirac_{\gg} - j(\Dirac_{\ll}) \in \cWeil.
 \]
\end{definition}

The following theorem collects the main properties of $\Dirac_{\gg,\ll}$. We denote by
$\WW_{\varepsilon}(\gg,\ll)$ the $\ZZ$-graded color subsuperalgebra of $\Weil$ of elements annihilated by
$L_x$ and $\iota_x$ for all $x\in\ll$. Its completion is denoted by
$\widehat{\WW}_{\varepsilon}(\gg,\ll)\subset \cWeil$.

\begin{theorem} \label{thm::square_Dirac}
 \begin{enumerate}
 \item[a)] $\Dirac_{\gg,\ll} \in \widehat{\WW}_{\varepsilon}(\gg,\ll)$. In particular, for all $x \in \ll$, one has $$\iota_{x}\Dirac_{\gg,\ll}=0, \qquad L_{x}\Dirac_{\gg,\ll}=0.$$
 \item[b)] The square of $\Dirac_{\gg,\ll}$ is 
 \[
 \Dirac_{\gg,\ll}^{2} = \Omega_{\gg}-j(\Omega_{\ll}) + \tfrac{1}{24} \etr(\ad_{\gg_{0}}(\Omega_{\gg_{0}})) - \tfrac{1}{24} \etr(\ad_{\ll_{0}}(\Omega_{\ll_{0}}))+B(\rho^{\sharp},\rho^{\sharp})-B(\rho_{\ll}^{\sharp}, \rho_{\ll}^{\sharp}).
 \]
 \end{enumerate}
\end{theorem}
\begin{proof}
 a) By Proposition~\ref{lemm::properties_map_j}, the color Lie algebra homomorphism $j: \widehat{\mathcal{W}}_{\varepsilon}(\ll) \to \widehat{\mathcal{W}}_{\varepsilon}(\gg)$ intertwines Lie derivatives and contractions, consequently using Corollary~\ref{cor::D_square} one has $L_{x}\Dirac_{\gg,\ll}=0$ for all $x \in \ll$.
 It remains to prove $\iota_{x}\Dirac_{\gg,\ll}=0$
 for all $x \in \ll$. One has \[
j(\gamma^{\WW}(x)+\rho_{\ll}(x)(1\otimes 1))
=
\gamma^{\WW}(x)+(\rho_{\ll}(x)-\rho_{\ss}(x))(1\otimes 1)
=
\gamma^{\WW}(x)+\rho(x)(1\otimes 1),
\]
since \(\rho_{\ll}-\rho_{\ss}=\rho|_{\ll}\). Consequently, one has for any $x\in\ll$, $$\iota_x\Dirac_{\gg,\ll}=\iota_x\Dirac_{\gg}-\iota_xj(\Dirac_{\ll})=(\gamma^{\WW}(x)+\rho(x)(1\otimes 1))-j(\gamma^{\WW}(x)+\rho_{\ll}(x)(1\otimes 1))=0.$$ 

 b) In a), we have proven that $\Dirac_{\gg,\ll}$ lies in $\widehat{\WW}(\gg,\ll)$. In particular, it commutes with the image of $j$ since 
 \[
[j(1\otimes x),\Dirac_{\gg,\ll}]_{\widehat{\WW}} = 2\iota_{x}\Dirac_{\gg,\ll}=0, \qquad [j(x\otimes 1),\Dirac_{\gg,\ll}]_{\widehat{\WW}} = L_{x} \Dirac_{\gg,\ll}=0,
 \]
where we used Lemma~\ref{lemm::contraction_and_Lie_derivative_in_Weil}.
Thus, $[\Dirac_{\gg,\ll}, j(\Dirac_{\ll})]_{\widehat{\WW}}=0$. Using c) of Lemma~\ref{lemm::properties_Dirac'}, we conclude 
 \[
 \Dirac_{\gg,\ll}^{2}= \tfrac{1}{2}[\Dirac_{\gg,\ll}, \Dirac_{\gg,\ll}] = \tfrac{1}{2}[\Dirac_{\gg}, \Dirac_{\gg}]-\tfrac{1}{2}j([\Dirac_{\ll},\Dirac_{\ll}]) = \Dirac_{\gg}^{2}-j(\Dirac_{\ll}^{2}). 
 \]
 The statement now follows with Corollary~\ref{cor::D_square}.
\end{proof}

Assume that $\ll\subset\gg$ is quadratic, with form $B_{\ll}$, and that
$\gg=\ll\oplus\ss$ with $[\ss,\ss]\subset\ll$. Then the cubic term in
$\Dirac_{\gg,\ll}$ vanishes. Thus $\Dirac_{\gg,\ll}$ is quadratic. We call it the
\emph{quadratic Dirac operator} attached to $(\gg,\ll)$. The Kac--Moody superalgebra case treated
below is of this form.

\section{Applications to Kac--Moody Superalgebras} \label{sec::KM} In this section, we study the relative cubic Dirac operator for symmetrizable Kac--Moody superalgebras $\gg=\even\oplus\odd$, with $\ll=\even$ and $\ss=\odd$.

We first recall basic facts on Kac--Moody superalgebras and highest weight supermodules $M$, and describe the action of $\Dirac_{\gg,\even}$ on $M\otimes M(\odd)$, where $M(\odd)$ denotes the oscillator supermodule. We then determine $\ker \Dirac_{\gg,\even}$ in the example $\gg=\widehat{\osp}(1\vert 2n)$ for integrable highest weight supermodules. Finally, we introduce unitarizable $\gg$-supermodules, and study the existence of a Dirac inequality induced by $\Dirac_{\gg,\even}$.

In what follows, we fix $\Gamma=\ZZ_{2}$ and let $\varepsilon$ be the parity. Accordingly, we omit the index $\varepsilon$ throughout and write, for example, $\bigwedge(\gg)$, $\Cl(\gg)$, $S(\gg)$, $\WW(\gg)$, and $\mathfrak{osp}(\gg)$ in place of $\bigwedge_{\varepsilon}(\gg)$, $\Cl_{\varepsilon}(\gg)$, $S_{\varepsilon}(\gg)$, $\WW_{\varepsilon}(\gg)$, and $\mathfrak{so}_{\varepsilon}(\gg)$. Likewise, we speak simply of Lie superalgebras, supermodules, superderivations, and supercommutators, and all signs are understood with respect to parity.

\subsection{Generalities on Kac–Moody Superalgebras}\label{subsec::Generalities_KM_Superalgebras} This section introduces Kac–Moody superalgebras. We follow the standard references \cite{Kac_Wakimoto, KM_Structure, zbMATH04122182}.

Let $A \coloneqq (a_{ij})$ be an $n\times n$ matrix. Set $I=\{1,\ldots,n\}$, and let $p:I\to\ZZ_{2}$ be a map, referred to as the \emph{parity function}. Fix an even vector space $\hh$ of dimension $2\abs{I}-\rk(A)$. Then there exist linearly independent elements $\alpha_{i}\in\hh^{\ast}$ and $h_{i}\in\hh$ such that~\cite{Kac_infinite}
\begin{equation}
\alpha_{j}(h_{i})=a_{ij},\qquad i,j\in I.
\end{equation}
Define the superalgebra $\ggbar(A)$ with generators $X_{i},Y_{i}$ ($i\in I$) of parities $p(X_{i})=p(Y_{i})=p(i)$, together with $\hh$, and relations
\begin{equation}
[h,X_{i}]=\alpha_{i}(h)X_{i},\qquad [h,Y_{i}]=-\alpha_{i}(h)Y_{i},\qquad [X_{i},Y_{j}]=\delta_{ij}h_{i},\qquad [\hh,\hh]=0.
\end{equation}
There exists a maximal ideal $\mathfrak{m}$ of $\ggbar(A)$ intersecting $\hh$ trivially~\cite{Kac_infinite}. The \emph{contragredient Lie superalgebra} associated to $A$ is the quotient
$
\gg(A)\coloneqq \ggbar(A)/\mathfrak{m}.
$
We call $A$ the \emph{Cartan matrix} of the Lie superalgebra $\gg(A)$.
Without loss of generality we assume $a_{ii}\in\{2,0\}$, since $\gg(B)\cong\gg(A)$ whenever $B=DA$ for some invertible diagonal $n\times n$ matrix $D$. Such a Cartan matrix $A$ is called \emph{normalized}. 

A contragredient Lie superalgebra $\gg(A)$ is called \emph{quasisimple} if for every ideal $\jj$ of $\gg(A)$ either $\jj \subset \hh$ or $\jj + \hh = \gg(A)$. A Cartan matrix $A$ is called \emph{admissible} if it satisfies the following conditions \cite{zbMATH05896583}:
\begin{enumerate}
 \item[(i)] If $a_{ii} = 0$ and $p(i)=\bar{0}$, then $a_{ij} = 0$ for every $j \in I$.
 \item[(ii)] If $a_{ii} = 2$ and $p(i)=\bar{0}$, then $a_{ij} \in \ZZ_{\leq 0}$ for any $j \in I$, $i \neq j$ and $a_{ij}=0$ implies $a_{ji}=0$.
 \item[(iii)] If $a_{ii} = 2$ and $p(i)=\bar{1}$, then $a_{ij} \in 2\ZZ_{\leq 0}$ for any $j \in I$, $i \neq j$ and $a_{ij}=0$ implies $a_{ji}=0$.
\end{enumerate}
If $A$ is admissible, we call the associated contragredient Lie superalgebra $\gg(A)$ \emph{admissible}. Note that $\gg(A)$ is admissible if and only if $\ad_{X_{i}}$ and $\ad_{Y_{i}}$ act locally nilpotently \cite{zbMATH05149996}.

Any contragredient Lie superalgebra $\gg \coloneqq \gg(A)$ admits a root space decomposition, since the action of $\hh$ on $\gg(A)$ is diagonalizable:
\begin{equation}
\gg = \hh \oplus \bigoplus_{\alpha \in \Delta} \gg^{\alpha}, \qquad 
\gg^{\alpha} \coloneqq \{x \in \gg : [h,x] = \alpha(h)x \ \text{for all} \ h \in \hh\},
\end{equation}
for some $\Delta \subset \hh^{\ast} \setminus \{0\}$, called the \emph{set of roots}. By definition, $\alpha_{1}, \ldots, \alpha_{n}$ are roots, and every root $\alpha \in \Delta$ can be uniquely expressed as $\alpha = \sum_{i} n_{i}\alpha_{i}$, where all $n_{i}$ are either nonnegative or nonpositive integers. The roots $\alpha_{1}, \ldots, \alpha_{n}$ are called \emph{simple roots}. The set $\Uppi$ of all simple roots is called a \emph{simple system}. Furthermore, by linear independence of the $\alpha_{i}$, each root space $\gg^{\alpha}$ is either purely even or purely odd. We define the parity function $p : \Delta \to \ZZ_{2}$ by setting $p(\alpha) = 0$ or $1$ whenever $\gg^{\alpha}$ is even or odd, respectively. Consequently, $\Uppi=\Uppi_{\bar0}\sqcup\Uppi_{\bar1}$ and $\Delta=\Delta_{\bar0}\sqcup\Delta_{\bar1}$, according to the decomposition into even and odd simple roots and roots, respectively.

A root $\alpha \in \Delta$ is called \emph{positive} if, in the expression $\alpha = \sum_{i} n_{i}\alpha_{i}$, all coefficients $n_{i}$ are non-negative and not all zero. Negative roots are defined analogously. We denote by $\Delta^{+}$ and $\Delta^{-}$ the sets of positive and negative roots, respectively, so that $\Delta = \Delta^{+} \sqcup \Delta^{-}$. In particular, one obtains the triangular decomposition
\begin{equation} \label{eq::triangular_decomposition}
\gg = \nn^{-} \oplus \hh \oplus \nn^{+}, \qquad 
\nn^{\pm} \coloneqq \bigoplus_{\alpha \in \Delta^{\pm}} \gg^{\alpha}.
\end{equation}

For each simple root $\alpha_i\in\Uppi$, set
$h_{\alpha_i}\coloneqq [X_{\alpha_i},Y_{\alpha_i}]$. Then
\begin{equation}
[h,X_{\alpha_{i}}]=\alpha_{i}(h)X_{\alpha_{i}},\quad [h,Y_{\alpha_{i}}]=-\alpha_{i}(h)Y_{\alpha_{i}},\quad [X_{\alpha_{i}},Y_{\alpha_{i}}]=h_{\alpha_{i}}
\end{equation}
for all $h\in \hh$. Define a $\ZZ$-grading on $\gg$ by \begin{equation}\label{eq::Z_grading_KM}\deg(X_{\alpha_{i}})=1, \qquad \deg(Y_{\alpha_{i}})=-1,\qquad \deg(h_{\alpha_{i}})=0\end{equation} for all $i\in I$. Moreover, there exists $\rho\in\hh^{\ast}$ such that
\begin{equation}\label{eq::definition_Weyl_element}
\rho(h_{\alpha})=\tfrac{1}{2}\alpha(h_{\alpha}),\qquad \forall\alpha\in\Uppi.
\end{equation}
We refer to $\rho$ as the \emph{Weyl vector} for $\gg$. 

A simple root $\alpha_{i}$ is called \emph{isotropic} if $a_{ii}=0$. Otherwise it is called \emph{non-isotropic}. Furthermore, $\alpha_{i}$ is called \emph{regular} if for any other simple root $\alpha_{j}$, $a_{ij}=0$ implies $a_{ji}=0$. Otherwise it is called \emph{singular}. We call the Cartan matrix $A$ \emph{regular} if for any $i,j \in I$, $a_{ij}=0$ implies $a_{ji}=0$.

\begin{definition}[\cite{zbMATH04122182}]
 A \emph{Kac–Moody superalgebra} is a regular quasisimple admissible contragredient Lie superalgebra $\gg(A)$.
\end{definition}

We note that the definition adopted in this article differs from that in~\cite{zbMATH05896583}. Moreover, if $p(i) = \bar{0}$ for all $i \in I$, then $A$ is a generalized Cartan matrix, and $\gg(A)$ coincides with the Kac–Moody algebra associated with $A$.

\subsubsection{Symmetrizable Kac–Moody Superalgebras}
We are concerned with a particular class of Kac–Moody superalgebras, namely the \emph{symmetrizable Kac–Moody superalgebras}. A Cartan matrix $A$ is said to be \emph{symmetrizable} if there exists an invertible diagonal matrix $D$ such that $AD$ is symmetric. Symmetrizable Cartan matrices are regular. A contragredient Lie superalgebra $\gg(A)$ is called \emph{symmetrizable} if its Cartan matrix $A$ is symmetrizable.

\begin{lemma}[{\cite{zbMATH04122182}}] \label{lemm::equivalent_characterization_symmetrizable_KM}
A contragredient Lie superalgebra $\gg(A)$ is symmetrizable if and only if there exists a nondegenerate even supersymmetric invariant bilinear form $B(\cdot,\cdot)$ on $\gg(A)$.
\end{lemma}

This lemma justifies our interest in symmetrizable Kac–Moody superalgebras, for which a cubic Dirac operator can be constructed. In the sequel we write $\gg \coloneqq \gg(A)$ for a symmetrizable Kac–Moody superalgebra, where $A$ is fixed. If $D = \diag(d_{1}, \ldots, d_{n})$, we may normalize $B(\cdot,\cdot)$ such that
\begin{equation}\label{eq::normalization_B}
B(h,H_{j}) = d_{j}\alpha_{j}(h), \qquad h \in \hh.
\end{equation}
Such a bilinear form $B$ is called the \emph{standard form}, and we fix this normalization on $\gg$.

The restriction of $B$ to $\hh$ is nondegenerate, and $B(\gg^{\alpha}, \gg^{\beta}) = 0$ for $\alpha \ne -\beta$. Hence $B$ defines a nondegenerate pairing between $\gg^{\alpha}$ and $\gg^{-\alpha}$. Moreover, $B$ induces an isomorphism $\eta : \hh \to \hh^{\ast}$ and thereby a nondegenerate bilinear form on $\hh^{\ast}$, denoted by the same symbol by abuse of notation.

Examples of symmetrizable Kac–Moody superalgebras include the basic classical Lie superalgebras
\begin{equation}
A(m\vert n)\ (m\neq n),\quad B(m\vert n),\quad C(n),\quad D(m\vert n),\quad D(2,1;\alpha)\ (\alpha\neq 0,-1),\quad F(4),\quad G(3)
\end{equation}
classified in~\cite{Kac}, as well as the affine and twisted affine superalgebras. Our main interest lies in the latter.

\subsubsection{Affine and Twisted Affine Lie Superalgebras} Let $\ss$ be a finite-dimensional contragredient simple Lie superalgebra, $\ss \neq \psl(n\vert n)$, equipped with an even nondegenerate supersymmetric invariant form $B$. We define the \emph{affine Lie superalgebra} $\widehat{\ss}$ to be the infinite-dimensional super vector space
\begin{equation}
 \ss \otimes \CC[t,t^{-1}] \oplus \CC D \oplus \CC K, 
\end{equation}
with even elements $K,D$ and with the Lie bracket defined by 
\begin{equation}
 \begin{aligned} \label{eq::commutation_relations_affine_Lie_superalgebra}
 [x \otimes t^{k}, y \otimes t^{l}] = [x,y] \otimes t^{k+l}+k\delta_{k,-l}B(x,y)K, \\ [K,D] = [K, x \otimes t^{l}]=0, \qquad [D, x \otimes t^{l}] = lx \otimes t^{l} 
 \end{aligned}
\end{equation}
for any $x,y \in \ss$ and $k,l \in \ZZ$. Explicitly, if $X_{1}, \ldots, X_{n},Y_{1},\ldots, Y_{n}$ are generators of $\ss$, there exist $x_{0},y_{0}\in \ss$ such that $[Y_{i},x_{0}]=[X_{i},y_{0}]$ for all $i > 0$, and if we set $X_{0} \coloneqq x_{0}\otimes t^{-1}$ and $Y_{0} \coloneqq y_{0} \otimes t$, then $X_{0},X_{1},\ldots,X_{n}, Y_{0},Y_{1},\ldots,Y_{n}, K,D$ generate $\widehat{\ss}$. 

$B$ extends to $\widehat{\ss}$, again denoted by $B$, by $B(K,D)=1$, $B(x\otimes t^{k},y\otimes t^{l})=\delta_{k,-l}B(x,y)$ and $B(K,K)=B(D,D)=B(K,\ss)=B(D,\ss)=0$. By Lemma~\ref{lemm::equivalent_characterization_symmetrizable_KM}, this yields the following proposition.

\begin{proposition}
 $\widehat{\ss}$ is a symmetrizable Kac–Moody superalgebra.
\end{proposition}

Let $\ss=\psl(n\vert n)$, with $n\geq 2$, and set
$J:=\operatorname{diag}(I_n,-I_n)\in\gl(n\vert n)$.
For $x,y\in\ss$, choose lifts $\widetilde{x},\widetilde{y}
\in\sl(n\vert n)$ and define
\[
\partial(x):=\overline{[J,\widetilde{x}]},
\qquad
c(x,y):=\tr([\widetilde{x},\widetilde{y}]),
\]
where the bar denotes the quotient map to $\ss$.
Both expressions are independent of the chosen lifts. The affinization $\widehat{\psl}(n\vert n)$ is the
infinite-dimensional super vector space
\begin{equation}
\ss\otimes\CC[t,t^{-1}]
\oplus\CC D\oplus\CC K\oplus\CC D'\oplus\CC K'
\end{equation}
with additional even elements $D',K'$ and with Lie bracket
\begin{equation}
\begin{aligned}
[x\otimes t^k,y\otimes t^l]
&=[x,y]\otimes t^{k+l}
  +k\delta_{k,-l}B(x,y)K
  +\delta_{k,-l}c(x,y)K',\\
[D,x\otimes t^k]&=kx\otimes t^k,\\
[D',x\otimes t^k]&=\partial(x)\otimes t^k,\\
[D,D']&=0,
\end{aligned}
\end{equation}
where $K$ and $K'$ are central. Here, $\tr$ denotes the ordinary trace. $\widehat{\psl}(n\vert n)$ is a symmetrizable Kac--Moody superalgebra with two-dimensional center. In what follows, $\ss$ denotes a finite-dimensional quadratic contragredient simple Lie superalgebra, and $\widehat{\ss}$ its affinization.

We describe the roots of $\widehat{\ss}$. Let $\hh_{\ss}$ be a Cartan subalgebra of $\ss$, and set $\hh\coloneqq\hh_{\ss}\oplus\CC K\oplus\CC D$. Let $\delta,\kappa\in\hh^{\ast}$ be defined by \begin{equation}\label{eq::definition_delta_kappa}
\delta_{\hh_{\ss}\oplus\CC K}=0, \quad \delta(D)=1, \qquad \kappa\vert_{\hh_{\ss}\oplus \CC D}=0, \quad \kappa(K)=1,\end{equation} so that $\hh^{\ast}=\hh^{\ast}_{\ss}\oplus\CC\kappa\oplus\CC\delta$. If $\Delta_{\ss}$ denotes the root system of $\ss$, then the roots of $\widehat{\ss}$ are the elements $\alpha+k\delta$ with $\alpha\in\Delta_{\ss}$ and $k\in\ZZ$, together with the roots $k\delta$ for $k\in\ZZ\setminus\{0\}$. If $\Uppi_{\ss}=\{\alpha_{1},\ldots,\alpha_{n}\}$ is the standard simple system of $\ss$, then the standard simple system of $\widehat{\ss}$ is $\Uppi \coloneqq \{\alpha_{0},\alpha_{1},\ldots,\alpha_{n}\}$, where $\alpha_{0}\coloneqq\theta+\delta$ and $\theta$ denotes the lowest root of $\ss$. We denote the set of roots of $\widehat{\ss}$ by $\Delta$. The even and odd roots are given by
\begin{equation}
 \Delta_{\bar{0}} =\{\alpha+n\delta : \alpha \in (\Delta_{\ss})_{\bar{0}}, \ n \in \ZZ\} \sqcup \{n\delta : n\in \ZZ\setminus\{0\}\}, \qquad \Delta_{\bar{1}} = \{\alpha + n \delta : \alpha \in (\Delta_{\ss})_{\bar{1}}, \ n \in \ZZ\}.
\end{equation}

Finally, assume that $\ss_{\bar0}$ is reductive, say $\ss_{\bar0}=\bigoplus_{j=0}^{N}\ss_{\bar0,j}$, where $\ss_{\bar0,0}$ is abelian and $\ss_{\bar0,j}$ is a simple Lie algebra for $j\geq 1$. Then, for each $j=0,\ldots,N$, the affinization $\widehat{\ss}$ contains the affine Kac--Moody algebra associated with $\ss_{\bar0,j}$, which we denote by $\widehat{\ss}_{\bar0,j}$.

\medskip
Next we introduce \emph{twisted affine Lie superalgebras}. Consider an automorphism $\phi$ of $\ss$ of finite order $q$, which preserves the invariant form on $\ss$, that is, $B(x,y)=B(\phi(x),\phi(y))$ for all $x,y \in \ss$. In addition, let $\epsilon$ be a $q$-th primitive root of $1$. Extend $\phi$ to $\widehat{\ss}$ via 
\begin{equation}
 \phi(x \otimes t^{k}) = \epsilon^{-k} \phi(x)\otimes t^{k}, \quad \phi(D)=D, \quad \phi(K)=K.
\end{equation}
The \emph{twisted affine Lie superalgebra} $\ss^{\phi}$ is the subsuperalgebra of the fixed points of $\phi$. The construction does not depend on the choice of $\epsilon$, and if $\phi$ and $\psi$ are two automorphisms in the same connected component of $\operatorname{Aut}(\ss)$, then the corresponding twisted affine Lie superalgebras are isomorphic.

\begin{proposition}[\cite{zbMATH04122182}]
 $\ss^{\phi}$ is a symmetrizable Kac–Moody superalgebra.
\end{proposition}

To describe the simple system of $\ss^{\phi}$, we introduce a $\ZZ_{q}$-grading 
\begin{equation}
\ss = \ss^{0} \oplus \ss^{1} \oplus \cdots \oplus \ss^{q-1}, \qquad 
\ss^{k} \coloneqq \{ s \in \ss : \phi(s) = \epsilon^{k}s \}.
\end{equation}
It can be arranged that $\phi$ satisfies the following properties: $\ss^{0}$ is a simple finite-dimensional Lie superalgebra, $\ss^{1}$ is a simple $\ss^{0}$-module, and $\ss^{0} + \ss^{1}$ generates $\ss$~\cite[Chapter~6]{zbMATH05896583}. Let $\Delta^{0}$ denote the set of roots of $\ss^{0}$, and $\Delta^{j}$ the set of roots of $\ss^{j}$ with respect to a Cartan subalgebra of $\ss^{0}$. If $\theta$ is the lowest weight of $\ss^{1}$ and $\{\alpha_{1}, \ldots, \alpha_{n}\}$ is the standard simple system of $\ss^{0}$, then a simple system of $\ss^{\phi}$ is given by $\{\alpha_{0}, \alpha_{1}, \ldots, \alpha_{n}\}$, where $\alpha_{0} = \theta + \delta$. The roots of $\ss^{\phi}$ are of the form $\alpha + k\delta$ with $\alpha \in \Delta^{j}$ and $k \in j + q\ZZ$, or $k\delta$ with $k \in q\ZZ \setminus \{0\}$.

\subsubsection{Kac–Peterson Class}
Let $\gg \coloneqq \gg(A)$ be a symmetrizable Kac–Moody superalgebra, and let $B$ denote the standard form on $\gg$. Let $(e_{a})$ be a homogeneous basis with $B$-dual basis $(e^{a})$. It admits a corrected Casimir element~\cite{KM_Structure}
\begin{equation}
\Omega_{\gg} = 2\rho^{\sharp} + \widehat{Q}(\sum_{a} \varepsilon(e_{a}) e_{a} e^{a}).
\end{equation}
By Corollary~\ref{cor::Casimir_and_KP_class}, this implies that the Kac–Peterson class is trivial.

\begin{proposition}\label{prop::trivial_KP_class_KM}
For a symmetrizable Kac–Moody superalgebra $\gg$ one has
\[
\KP = \operatorname{d}\!\rho,
\]
where $\rho \in \hh^{\ast}$ is the Weyl vector. In particular, the Kac–Peterson class of $\gg$ is trivial.
\end{proposition}

\begin{corollary}
 Any symmetrizable Kac--Moody superalgebra $\gg$ has a corrected cubic Dirac operator $\Dirac \in \widehat{\WW}(\gg)$.
\end{corollary}

In what follows, we refer to the corrected cubic Dirac operator simply as the cubic Dirac operator.

\subsection{Oscillator Module}\label{subsec::Oscillator_supermodule} Let $\gg$ be a symmetrizable Kac--Moody superalgebra with nondegenerate invariant supersymmetric bilinear form $B$ in standard normalization. Assume that $B(\alpha,\beta)\in\RR$ for all simple roots $\alpha,\beta$. This condition is satisfied for all twisted and untwisted affine Lie superalgebras constructed from basic classical Lie superalgebras, except for those associated with $D(2,1;\alpha)$ when $\alpha\notin\RR$. For the rest of this article, we fix $\ll = \even$ and $\ss = \odd$, and consider the relative Dirac operator $\Dirac_{\gg,\even}$, which is an element of $\widehat{\WW}(\gg,\ll)$ and quadratic since $[\odd,\odd]\subset \even$. Moreover, following Section~\ref{subsec::relative_cubic_Dirac_operator}, we set $\rho_{\bar1}\coloneqq \rho_{\bar0}-\rho$ with $\rho_{\bar{0}} \coloneqq \rho_{\even}$. 

In what follows, we fix a simple system $\Uppi$ of $\gg$, whose simple roots we denote by $\alpha_{i}$. If $M$ is a highest weight $\gg$-module (\emph{cf.}~Section~\ref{subsec::HWM}), we will see that $\Dirac_{\gg,\even}$ has a well-defined action on $M\otimes M(\odd)$, where $M(\odd)$ is the simple module over $\Cl(\odd)$ induced by the triangular decomposition, called the \emph{oscillator module}.

The restriction of $B$ to $\odd$ is nondegenerate and skew-symmetric. Hence $\odd$ is symplectic, and $\Cl(\odd)$ identifies with the Weyl algebra of $\odd$, that is, the algebra generated by the elements of $\odd$ with defining relations
$
xy-yx=2B(x,y)
$
for homogeneous $x,y\in\odd$. Let $\widehat{\Cl}(\odd)$ denote its completion.

Fix a triangular decomposition $\gg=\nn^{-}\oplus\hh\oplus\nn^{+}$ (\emph{cf.}~\eqref{eq::triangular_decomposition}). Then
\begin{equation}
\odd=(\odd)_{-}\oplus(\odd)_{+},\qquad (\odd)_{\pm}\coloneqq \odd\cap\nn^{\pm}.
\end{equation}
The subspaces $(\odd)_{-}$ and $(\odd)_{+}$ are isotropic, and $B$ identifies them in the restricted sense, that is, it induces a nondegenerate pairing $(\odd)_{+}\times(\odd)_{-}\longrightarrow\mathbb C$ so that $(\odd)_{+} \cong (\odd)_{-}^{\ast}$ (restricted dual). We call this the \emph{triangular polarization}. Let $(x_{a})$ be a basis of $(\odd)_{-}$ and $(\partial_{a})$ a basis of $(\odd)_{+}$ such that
\begin{equation}
B(x_{a},x_{b})=0,\qquad B(\partial_{a},\partial_{b})=0,\qquad B(\partial_{a},x_{b})=\tfrac{1}{2}\delta_{ab}.
\end{equation} 
With this notation, the Weyl algebra of $\odd$ is generated by the $x_{a}$ and $\partial_{a}$, and identifies with the algebra of polynomial differential operators in the variables $x_{a}$, where $\partial_{a}$ acts as $\partial/\partial x_{a}$.

The simple $\ZZ$-graded $\Cl(\odd)$-module is
\begin{equation}
M(\odd)\coloneqq
\operatorname{Sym}_{\CC}((\odd)_{-})
=
\bigoplus_{k\geq 0}\operatorname{Sym}_{\CC}^{k}((\odd)_{-})
\cong \CC[x_{1},x_{2},\ldots],
\end{equation}
where $\operatorname{Sym}_{\CC}$ denotes the ordinary symmetric algebra
of the underlying vector space. In addition, $M(\odd)$ carries the
natural $\ZZ$-grading induced by that of $\odd$. With respect to this grading, the action of $\Cl(\odd)$ extends componentwise to an algebra homomorphism $\widehat{\Cl}(\odd)\to\widehat{\End}(M(\odd))$, since  for each $v\in M(\odd)_j$, all components of positive degree greater than $-j$ act trivially on $v$. Hence every completed element acts on $v$ through a finite sum, and $M(\odd)$ is a well-defined $\widehat{\Cl}(\odd)$-module.

Additionally, $M(\odd)$ carries a Hermitian form $\langle\cdot,\cdot\rangle_{M(\odd)}$ such that the operators $x_{a}$ and $\partial_{a}$ are adjoint to each other, that is, $\langle x_{a}v,w\rangle_{M(\odd)} = \langle v, \partial_{a}w\rangle_{M(\odd)}$ for any $v,w \in M(\odd)$. Concretely, let $\alpha = (\alpha_1,\alpha_2,\dots)$
be a multi-index with finite support. We write
\begin{equation}
x^\alpha \coloneqq \prod_{i\ge 1} x_i^{\alpha_i}, 
\qquad 
\alpha! \coloneqq \prod_{i\ge 1} \alpha_i!.
\end{equation}
Then the Hermitian inner product on $\mathbb{C}[x_1,x_2,\dots]$ is given by
\begin{equation}
\langle x^\alpha, x^\beta \rangle_{M(\odd)} 
= \delta_{\alpha,\beta}\alpha!.
\end{equation}
Here we use the convention that Hermitian forms are conjugate-linear in the first argument and linear in the second.

In what follows, we regard $M(\odd)$ as a module over $\Cl(\odd)$ equipped with the Hermitian form $\langle\cdot,\cdot\rangle_{M(\odd)}$. 

The assignment $x_{a} \mapsto \partial_{a}$ and $\partial_{a} \mapsto x_{a}$ extends uniquely to a $^{\ast}$-operation $(\cdot)^{\ast}$ on $\Cl(\odd)$, that is, a conjugate-linear map
such that $(x^{\ast})^{\ast}=x$ and $(xy)^{\ast}=y^{\ast}x^{\ast}$ for all $x,y \in \Cl(\odd)$. This allows us to regard $M(\odd)$ as a $^{\ast}$-module over $\Cl(\odd)$, that is, 
\begin{equation}
 \langle xv,w \rangle_{M(\odd)}= \langle v,x^{\ast}w \rangle_{M(\odd)}
\end{equation}
for all $x \in \Cl(\odd)$ and $v,w \in M(\odd)$.

Recall the Lie superalgebra morphism $\gamma_{\odd}'\colon \even\to\widehat{\Cl}(\odd)$ from Lemma~\ref{lemm::gamma_s}, which defines a diagonal embedding of $\even$ into $\widehat{\WW}(\gg)$. Via this map, we regard $M(\odd)$ as a $\even$-module and denote the corresponding representation by $\pi_{M(\odd)}$. Recall $\rho_{\bar1}\coloneqq \rho_{\bar0}-\rho.$

Combining Lemma~\ref{lemm::gamma_s} and Remark~\ref{rmk::generator_adjoint_action}, one has the following.

\begin{lemma}\label{lemm::action_on_oscillator_supermodule}
The restriction of $\pi_{M(\odd)}\colon \even\to\End(M(\odd))$ to the Cartan subalgebra $\hh$ differs from the adjoint action of $\hh$ on $M(\odd)$ by the shift $-\rho_{\bar1}$, that is,
\begin{equation*}
\pi_{M(\odd)}(h)=-\rho_{\bar1}(h)+\ad_{h},\qquad h\in\hh.
\end{equation*}
\end{lemma}

The following corollary is immediate.

\begin{corollary}\label{cor::weights_M_odd}
The $\hh$-weights of $M(\odd)$ are
\begin{equation*}
-\rho_{\bar1}-\sum_{\alpha\in\Delta_{\bar1}^{+}}n_{\alpha}\alpha,\qquad n_{\alpha}\in \ZZ_{+}.
\end{equation*}
\end{corollary}

Let $M$ be a $\mathbb{Z}$-graded $\mathfrak{g}$-supermodule that is bounded above. Then $\mathcal{W}(\mathfrak{g},\mathfrak{g}_{\bar 0})$ acts naturally on $M\otimes M(\mathfrak{g}_{\bar 1})$. Since both factors are bounded above, this action extends componentwise to an action of $\widehat{\mathcal{W}}(\mathfrak{g},\mathfrak{g}_{\bar 0})$ on $\widehat{\operatorname{End}}(M\otimes M(\mathfrak{g}_{\bar 1}))$. In particular, this applies to the highest weight supermodules considered below. We denote the image of the relative cubic Dirac operator under this representation by the same symbol, so that
\begin{equation}
D_{\mathfrak{g},\mathfrak{g}_{\bar 0}}\in\widehat{\operatorname{End}}(M\otimes M(\mathfrak{g}_{\bar 1})).
\end{equation}
We examine this action in the next section.

\subsection{Action of the Cubic Dirac Operator} \label{subsec::action_of_the_cubic_Dirac_operator} In this section we derive an explicit formula for $\Dirac_{\gg,\even}^{2}$, and prove that $\ker \Dirac_{\gg,\even}^{2}\neq\{0\}$. We begin by introducing highest weight supermodules of $\gg$.

\subsubsection{Highest Weight Supermodules}\label{subsec::HWM}
Fix a simple system $\Uppi$ such that $\bb\coloneqq\hh\oplus\nn^{+}$ and $\gg=\nn^{-}\oplus\hh\oplus\nn^{+}$. Let $\rho$ denote the corresponding Weyl vector.

A $\gg$-supermodule $M$ is called a \emph{weight supermodule} if $\hh$ acts semisimply on $M$, that is,
\begin{equation}
M=\bigoplus_{\mu\in\hh^{\ast}}M^{\mu},\qquad M^{\mu}\coloneqq\{v\in M:hv=\mu(h)v\ \text{for all }h\in\hh\},
\end{equation}
with $\dim M^{\mu}<\infty$ for all $\mu\in\hh^{\ast}$. A weight supermodule $M$ is called a \emph{highest weight supermodule} if there exist $\Lambda\in\hh^{\ast}$ and $v_{\Lambda}\in M$ such that
\begin{equation}
\nn^{+}v_{\Lambda}=0,\qquad hv_{\Lambda}=\Lambda(h)v_{\Lambda}\ \text{for all }h\in\hh,\qquad \UE(\gg)v_{\Lambda}=M.
\end{equation}
$M$ is called \emph{integrable} if $X_{\alpha}$ and $Y_{\alpha}$
act locally nilpotently for every even real root $\alpha$. In particular, $\Lambda(h_{\alpha})\in \ZZ_{\geq 0}$ for all even real simple roots $\alpha$.

Any highest weight supermodule of $\gg$ has a canonical $\ZZ$-grading induced by the principal $\ZZ$-grading of $\gg$ (see~\eqref{eq::Z_grading_KM}). This $\ZZ$-grading will be used without further mention.

On every highest weight $\gg$-supermodule $M$ the corrected quadratic Casimir $\Omega_{\gg}$ is a well-defined operator and acts by a scalar.

\begin{lemma}[{\cite{KM_Structure}}]\label{lemm::action_Casimir}
Let $M$ be a highest weight $\gg$-supermodule of highest weight $\Lambda$. Then
\begin{equation*}
\Omega_{\gg}v=B(\Lambda+2\rho,\Lambda)v=(B(\Lambda+\rho,\Lambda+\rho)-B(\rho,\rho))v
\end{equation*}
for all $v\in M$.
\end{lemma}

Examples of highest weight supermodules are given by \emph{Verma supermodules}. For each $\Lambda \in \hh^{\ast}$, we define the \emph{Verma supermodule} $M_{\bb}(\Lambda)$ of highest weight $\Lambda$ by the induced supermodule
\begin{equation}
 M_{\bb}(\Lambda) \coloneqq \UE(\gg) \otimes_{\UE(\bb)} \CC_{\Lambda},
\end{equation}
where $\CC_{\Lambda}$ is the one-dimensional $\bb$-module with an even generator $v$ such that $\nn^{+}v=0$ and $hv=\Lambda(h)v$ for all $h \in \hh$. The vector $v$ is called the \emph{highest weight vector} of $M_{\bb}(\Lambda)$. Since $\bb$ is fixed, we omit the corresponding subscript in what follows.

\begin{proposition}[\cite{KM_Structure}]
 \begin{enumerate}
 \item[a)] $M(\Lambda)$ has a unique simple quotient, denoted by $L(\Lambda)$.
 \item[b)] If $2B(\Lambda+\rho,\alpha)\neq m B(\alpha,\alpha)$ for any $\alpha \in \Delta^{+}$ and $m \in \NN$, then $M(\Lambda)$ is simple.
 \end{enumerate}
\end{proposition}
 
\subsubsection{Action on \texorpdfstring{$M\otimes M(\odd)$}{}} Let $M$ be a highest weight $\gg$-supermodule of highest weight $\Lambda$, and consider $\Dirac_{\gg,\even}\in\widehat{\End}(M\otimes M(\odd))$. Since $\Dirac_{\gg,\even}\in\widehat{\WW}(\gg,\even)$, it is $\even$-invariant. Hence one has the following.

\begin{lemma}
The space $\ker \Dirac_{\gg,\even}$ is a $\even$-module.
\end{lemma}

Next, we apply Theorem~\ref{thm::square_Dirac} to obtain an explicit formula for $\Dirac_{\gg,\even}^{2}$. Since $\gg_{0}=\hh=(\gg_{\bar0})_{0}$ by the $\ZZ$-grading~\eqref{eq::Z_grading_KM}, one has $\str(\ad_{\gg_{0}}(\Omega_{\gg_{0}}))=0$ and $\str(\ad_{(\even)_{0}}(\Omega_{(\even)_{0}}))=0$. Together with Lemma~\ref{lemm::action_Casimir}, this shows that $\Dirac_{\gg,\even}^{2}$ is a well-defined operator on $M\otimes M(\odd)$ and yields the following proposition. 

\begin{proposition}\label{prop::action_on_HW_of_Dirac}
 Let $M$ be a highest weight $\gg$-supermodule with highest weight $\Lambda$. Suppose that $M\otimes M(\odd)$
contains a highest weight $\even$-submodule $V(\mu)$ of highest weight $\mu$. Then
$\Dirac_{\gg,\even}^{2}$ acts on $V(\mu)$ by the scalar
\[
B(\Lambda+\rho,\Lambda+\rho)-B(\mu+\rho_{\bar{0}},\mu+\rho_{\bar{0}}).
\]
In particular, $V(\mu)\subset\ker\Dirac_{\gg,\even}^{2}$ if and only if
\[
B(\Lambda+\rho,\Lambda+\rho)=B(\mu+\rho_{\bar{0}},\mu+\rho_{\bar{0}}).
\]
\end{proposition}

This shows that $\Dirac_{\gg,\even}^{2}$ measures the difference between the Casimir actions of $\gg$ and $\even$.

\begin{corollary}\label{lemm::action_Dirac_square_even_modules}
Let $M$ be a highest weight $\gg$-supermodule, and let $V(\mu)$ be a $\even$-submodule of $M$ with highest weight vector $v_{\mu}$. Then $\Dirac_{\gg,\even}^{2}$ acts on the highest weight $\even$-module generated by $v_{\mu}\otimes 1$ by the scalar
\[
B(\Lambda+\rho,\Lambda+\rho)-B(\mu+\rho,\mu+\rho).
\]
In particular, $v_{\mu}\otimes 1\in\ker\Dirac_{\gg,\even}^{2}$ if and only if the following equivalent conditions hold:
\begin{enumerate}
\item[a)] $B(\Lambda+\rho,\Lambda+\rho)=B(\mu+\rho,\mu+\rho)$;
\item[b)] $B(\Lambda+2\rho,\Lambda)=B(\mu+2\rho,\mu)$.
\end{enumerate}
\end{corollary}

These results imply that the square of the relative cubic Dirac operator has nontrivial kernel on every highest weight supermodule.

\begin{proposition}\label{prop::non_vanishing_HW}
Let $M$ be a highest weight $\gg$-supermodule. Then $\Ker\Dirac_{\gg,\even}^{2}\neq\{0\}$.
\end{proposition}

\begin{proof}
Let $v_{\Lambda}$ be a highest weight vector of $M$. Then $v_{\Lambda}$ generates a nonzero highest weight $\even$-submodule of highest weight $\Lambda$. By Corollary~\ref{lemm::action_Dirac_square_even_modules}, $\Dirac_{\gg,\even}^{2}$ acts trivially on the highest weight $\even$-module generated by $v_{\Lambda}\otimes 1$. Hence $v_{\Lambda}\otimes 1\in\Ker\Dirac_{\gg,\even}^{2}$.
\end{proof}

In general, $\ker \Dirac_{\gg,\even}\neq\ker \Dirac_{\gg,\even}^{2}$. Nevertheless, Proposition~\ref{prop::action_on_HW_of_Dirac} allows explicit computations of $\ker \Dirac_{\gg,\even}^{2}$. We perform these calculations in one example: integrable highest weight supermodules of $\widehat{\osp}(1\vert 2n)$, where this kernel equals $\ker \Dirac_{\gg,\even}$.

\subsection{An Example: Integrable Highest Weight Supermodules of \texorpdfstring{$\widehat{\osp}(1\vert 2n)$}{}} \label{subsec::osp}
We consider the affine Lie superalgebras $\widehat{\osp}(1\vert 2n)$, which constitute the simplest examples of affine Kac--Moody superalgebras (see Section~\ref{subsec::Generalities_KM_Superalgebras}) with a nontrivial odd part. 

\subsubsection{The affine Lie superalgebra \texorpdfstring{$\widehat{\osp}(1\vert 2n)$}{}}

The affine Lie superalgebra $\widehat{\osp}(1\vert 2n)$ is the untwisted affinization of the basic classical Lie superalgebra $\osp(1\vert 2n)$. Its structure and representation theory are described in detail in \cite{zbMATH04122182,Kac_Wakimoto,Frappat_Sciarrino_Sorba, zbMATH05896583}. We briefly recall its construction and structural properties.

Let $V=V_{\bar0}\oplus V_{\bar1}$ be a $\ZZ_{2}$-graded vector space with $\dim V_{\bar0}=1$ and $\dim V_{\bar1}=2n$, endowed with a nondegenerate even supersymmetric bilinear form. The corresponding orthosymplectic Lie superalgebra is denoted by $\osp(1\vert 2n)$. In the notation of Example~\ref{ex::color_Lie_algebras}, it is the $\varepsilon$-orthogonal Lie algebra of $V$ for $\Gamma=\ZZ_{2}$ and $\varepsilon$ given by parity. In Kac's notation, this is the simple Lie superalgebra of type $B(0,n)$. Its even part identifies canonically with $\mathfrak{sp}(V_{\bar{1}})\cong \mathfrak{sp}(2n)$ and its odd part is the standard simple $2n$-dimensional module of $\mathfrak{sp}(2n)$.
Moreover, $\osp(1\vert 2n)$ is simple and admits a nondegenerate even supersymmetric invariant bilinear form $B(\cdot,\cdot)$ which is unique up to an overall scalar, hence it is basic classical. A canonical choice is obtained from the defining representation via the supertrace (\emph{cf.}~Lemma~\ref{lemm::form_epsilon_orthogonal}):
\begin{equation}
B(x,y) \coloneqq -\operatorname{str}(xy), \qquad x,y\in\osp(1\vert 2n).
\end{equation}

Let
$
\hh_{\mathfrak{sp}(2n)}=\left\{\operatorname{diag}(h_{1},\ldots,h_{n},-h_{1},\ldots,-h_{n})\right\}\subseteq\mathfrak{sp}(2n)
$
be a Cartan subalgebra of $\mathfrak{sp}(2n)$, viewed as a subalgebra of the even part of $\mathfrak{osp}(1\vert 2n)$. It is then a Cartan subalgebra of $\osp(1\vert 2n)$, denoted by $\hh_{\osp}$. For $1\leq i\leq n$, define $\varepsilon_{i}\in\hh^{\ast}_{\osp}$ by $\varepsilon_{i}(h)=h_{i}$ for $h\in\hh_{\osp}$. Then the root system of $\osp(1\vert 2n)$ is $\Delta_{\osp}=(\Delta_{\osp})_{\bar{0}}\sqcup (\Delta_{\osp})_{\bar{1}}$, where
\begin{equation}
(\Delta_{\osp})_{\bar0}=\{\pm\epsilon_i\pm\epsilon_j:1\le i<j\le n\}
\cup
\{\pm2\epsilon_i:1\le i\le n\},\qquad
(\Delta_{\osp})_{\bar1}=\{\pm\varepsilon_{i} : 1 \leq i \leq n\}.
\end{equation}
The form $B$ restricts to a nondegenerate symmetric bilinear form on $\hh_{\osp}$. Via this restriction, we identify $\hh_{\osp}$ with $\hh^{\ast}_{\osp}$. This induces a nondegenerate symmetric bilinear form on $\hh^{\ast}_{\osp}$, again denoted by $B$. With the above normalization one has
\begin{equation}
B(\epsilon_i,\epsilon_j)=\tfrac12\delta_{ij}.
\end{equation}

Importantly, every root has positive length
\begin{equation}\label{eq::positivity_roots_osp}B(\pm\epsilon_i,\pm\epsilon_i)=\tfrac12,\qquad
B(\pm\epsilon_i\pm\epsilon_j,\pm\epsilon_i\pm\epsilon_j)=1,\qquad
B(\pm2\epsilon_i,\pm2\epsilon_i)=2.
\end{equation}

The \emph{affine Lie superalgebra} associated with $\osp(1\vert 2n)$ is
$
\widehat{\osp}(1\vert 2n)
=
(\osp(1\vert 2n)\otimes_{\CC} \CC[t,t^{-1}])
\oplus \CC K\oplus \CC D
$
with parity decomposition
\begin{equation}
\widehat{\osp}(1\vert 2n)_{\bar0}
=
(\osp(1\vert 2n)_{\bar0}\otimes_{\CC} \CC[t,t^{-1}])
\oplus \CC K\oplus \CC D,
\qquad
\widehat{\osp}(1\vert 2n)_{\bar1}
=
\osp(1\vert 2n)_{\bar1}\otimes_{\CC} \CC[t,t^{-1}]
\end{equation}
and the commutation relations from~\eqref{eq::commutation_relations_affine_Lie_superalgebra}. In particular, $\widehat{\mathfrak{osp}}(1\vert 2n)_{\bar{0}}=\widehat{\sp}(2n)$. Recall that $\widehat{\osp}(1\vert 2n)$ is a symmetrizable Kac--Moody superalgebra, equipped with the nondegenerate invariant supersymmetric bilinear form induced by $B$, again denoted by $B$ (see Section~\ref{subsec::Generalities_KM_Superalgebras}). As above, let $\hh\coloneqq \hh_{\osp}\oplus\CC K\oplus\CC D$ be the affine Cartan subalgebra of $\widehat{\osp}(1\vert 2n)$, with dual space $\hh^{\ast}=\hh^{\ast}_{\osp}\oplus\CC\kappa\oplus\CC\delta$ as in~\eqref{eq::definition_delta_kappa}. Then $\widehat{\osp}(1\vert 2n)$ admits the root space decomposition with root system
\begin{equation}
\Delta=\{\alpha+n\delta:\alpha\in\Delta_{\osp},\ n\in\ZZ\}\cup\{n\delta:n\in\ZZ\setminus\{0\}\}.
\end{equation}

The set of roots $\Delta$ of $\widehat{\mathfrak{osp}}(1\vert2n)$ decomposes into \emph{real} and \emph{imaginary} roots. The real roots are
\begin{equation}
\pm\epsilon_i+m\delta,\qquad
\pm\epsilon_i\pm\epsilon_j+m\delta,\qquad
\pm2\epsilon_i+m\delta,
\end{equation}
where $m\in\mathbb{Z}$, $1\leq i\leq n$, and $1\leq i<j\leq n$, with the two signs in the second family chosen independently. The roots $\pm\epsilon_i+m\delta$ are odd, whereas the remaining real roots are even. The imaginary roots are the nonzero multiples
\begin{equation}
m\delta,\qquad m\in\mathbb{Z}\setminus\{0\},
\end{equation}
which are even and have multiplicity $n$.

The standard simple system is \begin{equation}\Uppi_{\st}=\{\delta-2\epsilon_{1},\epsilon_{1}-\epsilon_{2},\ldots,\epsilon_{n-1}-\epsilon_{n},\epsilon_{n}\}.\end{equation} Let $\rho_{\osp}$ denote the corresponding Weyl vector of $\osp(1\vert 2n)$. Then the affine Weyl vector is $\rho=\rho_{\osp}+h^{\vee}\kappa$, where $h^{\vee}=n+\tfrac12$ is the dual Coxeter number. 

The \emph{Weyl group} $W$ of $\widehat{\mathfrak{osp}}(1\vert 2n)$ is the subgroup of $\operatorname{GL}(\hh^{\ast})$ generated by the even real reflections
\begin{equation}\label{eq::reflection_osp}
r_{\beta}(\alpha)
=
\alpha-2\tfrac{B(\alpha,\beta)}{B(\beta,\beta)}\beta,
\qquad
\beta\ \text{real even root}.
\end{equation}
The form $B$ is invariant under the action of $W$. Moreover, the simple systems of $\widehat{\osp}(1\vert 2n)$ form a single orbit under the Weyl group. Consequently, $\Uppi_{\st}$ is the unique, up to $W$ action, simple system of $\gg$.

\subsubsection{Action on Integrable Highest Weight Supermodules} \label{subsubsec::action_on_integrable_highest_weight_supermodules}
Fix a simple system $\Uppi=\{\alpha_{1},\ldots,\alpha_{n+1}\}$, and let $\Delta^{+}$ be the associated positive system and $\bb$ the corresponding Borel subalgebra. Since $\Uppi$ is $W$-conjugate to $\Uppi_{\mathrm{st}}$, there is a unique odd simple root $\beta\in\Uppi$ such that $\Uppi_{\bar{0}}\coloneqq(\Uppi\cap\Delta_{\bar{0}})\cup\{2\beta\}$ is a simple system for $\mathfrak{osp}(1\vert 2n)_{\bar{0}}$.

Any element $\Lambda\in\hh^{\ast}$ is uniquely of the form
\begin{equation}
\Lambda=\Lambda\vert_{\hh_{\osp}}+k\kappa+d_{0}\delta.
\end{equation}
We call $\Lambda\vert_{\hh_{\osp}}$ the $\hh$-weight and $k$ the level of $\Lambda$. With this notation, Kac and Wakimoto classified the simple integrable highest weight $\widehat{\osp}(1\vert 2n)$-supermodules in \cite{Kac_Wakimoto}.

\begin{theorem}[{\cite{Kac_Wakimoto}}]\label{thm::integrable_affine_osp}
Set
$
k_i\coloneqq 2\tfrac{B(\Lambda,\alpha_i)}{B(\alpha_i,\alpha_i)}$ for $i=1,\ldots,n+1$.
The simple highest weight $\widehat{\mathfrak{osp}}(1\vert2n)$-supermodule $L_{\mathfrak b}(\Lambda)$ is integrable if and only if
\[
k_i\in\mathbb{Z}_{\geq0}\quad(i=1,\ldots,n),
\qquad
k_{n+1}\in2\mathbb{Z}_{\geq0}.
\]
In this case, its level is
\[
k=\Lambda(K)=\sum_{i=1}^{n}k_i+\tfrac12k_{n+1}\in\mathbb{Z}_{\geq0}.
\]
\end{theorem}

\begin{remark}
Writing $\Lambda=\Lambda\vert_{\hh}+k\kappa+d\delta$, the part $\Lambda\vert_{\hh}$ is the highest weight of a finite-dimensional simple $\osp(1\vert 2n)$-module, and $k\in\ZZ_{+}$. Thus, for fixed level $k$, only finitely many finite parts occur among integrable simple highest weight modules. The parameter $d$ records the eigenvalue of the derivation $D$ on the highest weight vector and is independent.
\end{remark}

In what follows, let $L(\Lambda)$ be a fixed integrable highest weight supermodule. We omit the subscript $\bb$, since $\bb$ is fixed throughout. We determine $\ker \Dirac_{\gg,\even}^{2}$ in $\End_{\CC}(L(\Lambda)\otimes M(\odd))$ by solving, according to Proposition~\ref{prop::action_on_HW_of_Dirac},
\begin{equation}
B(\Lambda+\rho,\Lambda+\rho)=B(\mu+\rho_{\bar{0}},\mu+\rho_{\bar{0}}),
\end{equation}
where $\mu$ is the highest weight of a $\even$-constituent of $L(\Lambda)\otimes M(\odd)$.

We first decompose $L(\Lambda)\otimes M(\odd)$ as an $\even\cong\widehat{\sp}(2n)$-module. We begin by describing its weights. Let $Q$ be the root lattice generated by the simple roots $\alpha_i$, and let $Q_{+}$ be the submonoid generated by the $\alpha_i$. The set of weights $\mathcal{P}_{L_{\bb}(\Lambda)}$ of $L_{\bb}(\Lambda)$ is contained in $\{\Lambda-\gamma:\gamma\in Q_{+}\}$. On the other hand, by Corollary~\ref{cor::weights_M_odd}, the weights of $M(\odd)$ are contained in $\{-\rho_{\bar1}-\sum_{\alpha\in\Delta_{\bar1}^{+}}n_{\alpha}\alpha:n_{\alpha}\in\ZZ_{+}\}$, where $\rho_{\bar1}\coloneqq\rho_{\bar0}-\rho$. Hence the weights of $L(\Lambda)\otimes M(\odd)$ are contained in
\begin{equation}\label{eq::weight_set}
\Bigl\{\Lambda-\rho_{\bar1}-\gamma:\gamma\in Q_{+}\Bigr\}.
\end{equation}

Let $V$ be the natural $2n$-dimensional $\mathfrak{sp}(2n)$-module. The oscillator module associated with the chosen polarization has the polynomial realization
\begin{equation}
M(\odd)=\mathbb C[x_1,\ldots,x_n]\otimes S\!\left(\bigoplus_{r\geq1}Vt^{-r}\right).
\end{equation}
Here $x_i$ has weight $-\varepsilon_i$ and degree zero, while $Vt^{-r}$ has degree $r$. The oscillator vacuum $1$ has weight $-\rho_{\bar1}$ by Lemma~\ref{lemm::action_on_oscillator_supermodule}. In particular, $K$ acts on $M(\odd)$ by $-\tfrac12$, and $d$ acts on its degree-$r$ subspace by $-r$. We denote by $M(\odd)^{+}$ and $M(\odd)^{-}$ the subspaces of even and odd total polynomial degree. After normalization, $M(\odd)$ identifies with the integer-Fock module of \cite[Chapter~II]{1994}, and the following decomposition holds.
\begin{theorem}[{\cite[Theorem~2.1]{1994}}]\label{thm::decomp_oscillator_osp}
One has, as $\even$-modules up to parity, $M(\odd)^{+}\cong L_{0}(-\rho_{\bar{1}})$ and $M(\odd)^{-}\cong L_{0}(-\rho_{\bar{1}}-\beta)$. In particular,
\[
M(\odd)\cong L_{0}(-\rho_{\bar{1}})\oplus L_{0}(-\rho_{\bar{1}}-\beta).
\]
\end{theorem}

Recall that a highest weight $\mu$ is \emph{admissible} if $\langle\mu+\rho_{\bar0},\alpha^\vee\rangle\notin\mathbb Z_{\leq0}$ for every positive real root $\alpha$, and if the real coroots integral on $\mu$ span the full rational coroot space \cite{KW_tensor}. We obtain the following decomposition.

\begin{proposition}\label{prop::deocmposition_TP}
The $\even$-module $L(\Lambda)\otimes M(\odd)$ is a direct sum of simple highest weight modules with admissible highest weights.
\end{proposition}

\begin{proof}
The restriction of $L(\Lambda)$ to $\even$ is an integrable module in the ordinary category $\mathcal O$ of $\even$. Its weight spaces are finite-dimensional, and
$Q_{+}=Q_{\bar{0}}^{+}\cup(\beta+Q_{\bar{0}}^{+})$
implies that its weights lie in
$(\Lambda-Q_{\bar{0}}^{+})\cup(\Lambda-\beta-Q_{\bar{0}}^{+})$.
Hence, by \cite[Theorem~10.7]{Kac_infinite}, it decomposes into simple integrable highest weight $\even$-modules.

By Theorem~\ref{thm::decomp_oscillator_osp}, $M(\odd)$ is the direct sum of two simple highest weight $\even$-modules with admissible highest weights. The Kac--Wakimoto tensor product theorem \cite{KW_tensor} states that the tensor product of a simple integrable highest weight module with a simple admissible highest weight module is a direct sum of simple highest weight modules with admissible highest weights. Applying this theorem to each pair of constituents proves the assertion.
\end{proof}

One has $\sl_{2\beta}\coloneqq\langle e_{2\beta},f_{2\beta},(2\beta)^\vee\rangle\cong\sl(2)$. Moreover, the subspace $\mathbb C[x_n]\subset M(\odd)$ corresponding to the root pair $\{\beta,-\beta\}$ is $\sl_{2\beta}$-stable. Its $\sl_{2\beta}$-action is the one-variable oscillator action. More precisely, after a compatible normalization of the root vectors,
\begin{equation}
e_{2\beta}=\frac12\partial_{x_n}^2,
\qquad
f_{2\beta}=-\frac12x_n^2,
\qquad
(2\beta)^\vee=-x_n\partial_{x_n}-\frac12.
\end{equation}
In particular, $\CC[x_n]$ decomposes into two simple $\sl_{2\beta}$-modules,
$\mathbb C[x_n]=\mathbb C[x_n^2]\oplus x_n\mathbb C[x_n^2]$,
of highest weights $-\tfrac12$ and $-\tfrac32$, respectively. We need the following standard fact about the $\sl(2)$-oscillator module.

\begin{lemma}[{\cite[Chapter~II]{Howe}}]\label{lemm::rank_one_oscillator}
Let $F_m$ be the finite-dimensional highest weight $\sl_{2\beta}$-module of highest weight $m\in\ZZ_{+}$. Then, as $\sl_{2\beta}$-modules,
\begin{equation}
F_m\otimes\mathbb C[x_n]
\cong
\bigoplus_{j=0}^{m}L\left(m-\tfrac12-2j;\sl_{2\beta}\right)
\oplus
\bigoplus_{j=0}^{m}L\left(m-\tfrac32-2j;\sl_{2\beta}\right).
\end{equation}
\end{lemma}

Before treating the general case, we consider an example.

\begin{example}
Let $\gg=\widehat{\mathfrak{osp}}(1|2)$, so that $\gg_{\bar0}=\widehat{\mathfrak{sl}}_2$. We fix the standard simple system
$\Uppi_{\st}=\{\alpha_0\coloneqq\delta-2\epsilon,\alpha_1\coloneqq\epsilon\}$.
Consider the highest weight supermodule
\begin{equation}
L(\Lambda),\qquad \Lambda=\epsilon+\kappa.
\end{equation}
Its Dynkin labels are $k_0=0$ and $k_1=2$. Hence, by Theorem~\ref{thm::integrable_affine_osp}, $L(\Lambda)$ is integrable of level $k=1$. We have
$\rho=\tfrac12\epsilon+\tfrac32\kappa$ and $\rho_{\bar0}=\epsilon+2\kappa$. Thus
$\rho_{\bar1}=\rho_{\bar0}-\rho=\tfrac12\epsilon+\tfrac12\kappa$, and
\begin{equation}
B(\Lambda+\rho,\Lambda+\rho)=\tfrac98.
\end{equation}

By Theorem~\ref{thm::decomp_oscillator_osp},
\begin{equation}
M(\odd)\vert_{\even}\cong L_0(-\rho_{\bar1})\oplus L_0(-\rho_{\bar1}-\beta),
\end{equation}
and both summands are admissible of level $-\tfrac12$. Their highest weight vectors are the vacuum $1$ and $x_{-\alpha_1}$, respectively, where $x_{-\alpha_1}$ denotes a nonzero root vector in $\gg^{-\alpha_1}$.

By Proposition~\ref{prop::deocmposition_TP}, $L(\Lambda)\otimes M(\odd)$ is a direct sum of simple highest weight $\even$-modules with admissible highest weights. Since every $\even$-constituent of $L(\Lambda)$ is integrable of level one, each constituent of $L(\Lambda)\otimes M(\odd)$ has level $\tfrac12$. The admissible highest weights of $\widehat{\mathfrak{sl}}_2$ at this level that occur here are therefore of the form
\begin{equation}
\mu_{c,N}=c\epsilon+\tfrac12\kappa-N\delta,
\qquad N\in\mathbb Z_{\geq0},\qquad
c\in\left\{\tfrac12,-\tfrac12,-\tfrac32,-\tfrac52\right\}.
\end{equation}
Indeed, these are precisely the relevant principal admissible finite Dynkin labels at level $\tfrac12$ \cite{KW_tensor}.

By Proposition~\ref{prop::action_on_HW_of_Dirac}, $\Dirac_{\gg,\even}^{2}$ acts on the constituent $L_0(\mu_{c,N})$ by the scalar
\begin{equation}
B(\Lambda+\rho,\Lambda+\rho)-B(\mu_{c,N}+\rho_{\bar0},\mu_{c,N}+\rho_{\bar0})
=\tfrac98-\tfrac{(c+1)^2}{2}+5N.
\end{equation}
Hence this scalar vanishes precisely for $(c,N)=(\tfrac12,0)$ and $(c,N)=(-\tfrac52,0)$. Both possibilities occur. The first is represented by $v_\Lambda\otimes1$, whose highest weight is
$\mu_+=\Lambda-\rho_{\bar1}=\tfrac12\varepsilon+\tfrac12\kappa$.

For the second, let $F_1$ denote the two-dimensional simple $\sl(2)$-module. The grade-zero subspace contains $F_1\otimes L(-\tfrac32;\sl(2))$. By Lemma~\ref{lemm::rank_one_oscillator}, this tensor product contains a simple $\sl(2)$-constituent of highest weight $-\tfrac52$, with multiplicity one. Since it occurs in affine grade zero, its highest weight vector is also a $\widehat{\mathfrak{sl}}_2$-highest weight vector, of weight
$\mu_-=-\tfrac52\varepsilon+\tfrac12\kappa$.
Therefore
\begin{equation}
\ker\Dirac_{\gg,\even}^{2}
=
L_0\left(\tfrac12\epsilon+\tfrac12\kappa\right)
\oplus
L_0\left(-\tfrac52\epsilon+\tfrac12\kappa\right)
=
L_0(\Lambda-\rho_{\bar1})
\oplus
L_0\bigl(s_{2\epsilon}(\Lambda+\rho)-\rho_{\bar0}\bigr).
\end{equation}
\end{example}

The example motivates a general theorem describing $\ker\Dirac_{\gg,\even}^{2}$. To prove it, we need one technical lemma.

\begin{lemma}\label{lemm::dominant_shift_admissible_osp}
Assume $\Uppi=\Uppi_{\st}$, so that $\beta=\epsilon_n$.
Let $\mu$ be the highest weight of a $\even$-constituent
of $L(\Lambda)\otimes M(\odd)$. Then there exists a unique
$w\in\{1,r_{2\beta}\}$ such that $L(\nu)$ with
\[
\nu\coloneqq w(\mu+\rho_{\bar{0}})-\rho
\]
is integrable. In particular,
$k\geq\nu_1\geq\cdots\geq\nu_n\geq0$
and $\nu_i\in\ZZ$.
\end{lemma}

\begin{proof}
By Proposition~\ref{prop::deocmposition_TP}, $\mu$ is
admissible. Write
\[
\mu+\rho_{\bar{0}}
=
\sum_{i=1}^{n}a_i\epsilon_i+\ell\kappa+d\delta,
\qquad
\ell\coloneqq k+n+\tfrac12,
\]
where $a_i$ are the finite coordinates of
$\mu+\rho_{\bar{0}}$ and $d=\mu(D)$.
Indeed, $K$ acts on $L(\Lambda)\otimes M(\odd)$ by
$k-\tfrac12$, whereas
$\rho_{\bar{0}}(K)=n+1$.
The finite coordinates of the weights of $L(\Lambda)$
are integers, and those of $M(\odd)$ belong to
$\ZZ-\tfrac12$, since its vacuum has weight
$-\rho_{\bar{1}}$.
As the finite coordinates of $\rho_{\bar{0}}$ are
integers, it follows that $a_i\in\ZZ+\tfrac12$.

For $i<j$, the positive even roots
$\epsilon_i-\epsilon_j$ and $\epsilon_i+\epsilon_j$
give the coroot pairings
\[
\left\langle\mu+\rho_{\bar{0}},
(\epsilon_i-\epsilon_j)^\vee\right\rangle=a_i-a_j,
\qquad
\left\langle\mu+\rho_{\bar{0}},
(\epsilon_i+\epsilon_j)^\vee\right\rangle=a_i+a_j.
\]
Likewise, the positive even roots
$\delta\pm2\epsilon_i$ give
\[
\left\langle\mu+\rho_{\bar{0}},
(\delta\pm2\epsilon_i)^\vee\right\rangle
=\ell\pm a_i.
\]
All these pairings are integers. Since $\mu$ is admissible,
they are strictly positive. Consequently,
\[
\ell>a_1>\cdots>a_{n-1}>|a_n|>0.
\]
Here the last strict inequality follows from
$a_n\in\ZZ+\tfrac12$. For $n=1$, the displayed inequalities
are understood as $\ell>|a_1|>0$.

The reflection $r_{2\beta}=r_{2\epsilon_n}$ changes the
sign of the $\epsilon_n$-coordinate and fixes the other
finite coordinates, as well as $\kappa$ and $\delta$.
Define
\[
w\coloneqq
\begin{cases}
1,&a_n>0,\\
r_{2\beta},&a_n<0,
\end{cases}
\qquad
b_i\coloneqq
\begin{cases}
a_i,&i<n,\\
|a_n|,&i=n.
\end{cases}
\]
Thus $b_i$ are the finite coordinates of
$w(\mu+\rho_{\bar{0}})$, and they satisfy
$\ell>b_1>\cdots>b_n>0$.

Using
$\rho=\sum_{i=1}^{n}(n-i+\tfrac12)\epsilon_i
+(n+\tfrac12)\kappa$, we obtain
\[
\nu=w(\mu+\rho_{\bar{0}})-\rho
=\sum_{i=1}^{n}\nu_i\epsilon_i+k\kappa+d\delta,
\qquad
\nu_i\coloneqq b_i-(n-i+\tfrac12).
\]
Each $\nu_i$ is an integer. Since $b_i-b_{i+1}$ and
$\ell-b_1$ are positive integers, and $b_n$ is a positive
half-integer, one has
\[
\nu_i-\nu_{i+1}=b_i-b_{i+1}-1\geq0,
\qquad
\nu_n=b_n-\tfrac12\geq0,
\qquad
k-\nu_1=\ell-b_1-1\geq0.
\]
Hence $k\geq\nu_1\geq\cdots\geq\nu_n\geq0$.

Equivalently, the labels of $\nu$ on the simple roots
$\delta-2\epsilon_1$,
$\epsilon_i-\epsilon_{i+1}$ and $\beta=\epsilon_n$
are, respectively,
$k-\nu_1$, $\nu_i-\nu_{i+1}$ and $2\nu_n$.
The first two are nonnegative integers, and the last is
a nonnegative even integer. Thus $\nu$ satisfies the
integrability conditions of
Theorem~\ref{thm::integrable_affine_osp}.

Finally, the other choice of $w$ gives last finite
coordinate $-|a_n|-\tfrac12<0$ after subtracting $\rho$,
and therefore does not give a dominant weight.
This proves uniqueness.
\end{proof}

\begin{theorem}\label{thm::kernel_Dirac_affine_osp}
Let $L(\Lambda)$ be a simple integrable highest weight
$\gg=\widehat{\osp}(1\vert 2n)$-supermodule, and let
$\beta\in\Uppi$ be the unique odd simple root. Then, as
$\even$-modules up to parity,
\begin{equation}
\ker\Dirac_{\gg,\even}
\cong
L_{0}(\Lambda-\rho_{\bar{1}})
\oplus
L_{0}\bigl(r_{2\beta}(\Lambda+\rho)-\rho_{\bar{0}}\bigr).
\end{equation}
Both constituents have multiplicity one, and their highest
weights are admissible.
\end{theorem}

\begin{proof}
Since $\Uppi$ is $W$-conjugate to $\Uppi_{\st}$, we may assume,
after transporting the Borel subalgebra and the oscillator
polarization, that $\Uppi=\Uppi_{\st}$. First, we determine $\ker \Dirac_{\gg,\even}^{2}$.

By Proposition~\ref{prop::deocmposition_TP},
$L(\Lambda)\otimes M(\odd)$ is a direct sum of simple highest
weight $\even$-modules with admissible highest weights.
Hence, by Proposition~\ref{prop::action_on_HW_of_Dirac}, it
suffices to determine the constituents whose highest weights
$\mu$ satisfy
\begin{equation}\label{eq::kernel_osp_norm_equality}
B(\Lambda+\rho,\Lambda+\rho)
=
B(\mu+\rho_{\bar{0}},\mu+\rho_{\bar{0}}).
\end{equation}

For any $\even$-constituent $L_0(\mu)$, Lemma~\ref{lemm::dominant_shift_admissible_osp}
gives a unique $w\in\{1,r_{2\beta}\}$ such that
$\nu\coloneqq w(\mu+\rho_{\bar{0}})-\rho$ satisfies the
integrability conditions. By $W$-invariance of $B$,
\begin{equation}
B(\mu+\rho_{\bar{0}},\mu+\rho_{\bar{0}})
=
B(\nu+\rho,\nu+\rho).
\end{equation}
Consequently, \eqref{eq::kernel_osp_norm_equality} is equivalent to
\begin{equation}
0=B(\Lambda+\nu+2\rho,\Lambda-\nu).
\end{equation}

We claim that $\Lambda-\nu\in Q_{+}$. Since $\mu$ is the highest
weight of a $\even$-constituent of $L(\Lambda)\otimes M(\odd)$,
it occurs in
$L(\Lambda)\otimes L_0(-\rho_{\bar1}-j\beta)$ for some
$j\in\{0,1\}$. Hence
\begin{equation}
\mu=\xi-\rho_{\bar1}-j\beta,
\qquad
\xi\in\mathcal{P}_{L(\Lambda)}.
\end{equation}
Using $w\in\{1,r_{2\beta}\}$,
$\nu=w(\mu+\rho_{\bar0})-\rho$, and
$r_{2\beta}(\beta)=-\beta$, we obtain
\begin{equation}
\nu=
\begin{cases}
\xi-j\beta,&w=1,\\
r_{2\beta}(\xi)-(1-j)\beta,&w=r_{2\beta}.
\end{cases}
\end{equation}
The set $\mathcal{P}_{L_{\bb}(\Lambda)}$ is $W$-invariant by
integrability. Thus, in either case,
\begin{equation}
\nu\in
\mathcal{P}_{L_{\bb}(\Lambda)}
\cup
\bigl(\mathcal{P}_{L_{\bb}(\Lambda)}-\beta\bigr)
\subseteq
\{\Lambda-\gamma:\gamma\in Q_{+}\}.
\end{equation}
Therefore $\Lambda-\nu\in Q_{+}$.

Both $\Lambda$ and $\nu$ satisfy the integrability conditions
and are therefore dominant. Moreover, $\rho$ pairs strictly
positively with every simple root. Writing
$\Lambda-\nu=\sum_i q_i\alpha_i$, with
$q_i\in\ZZ_{\geq0}$, we obtain
\begin{equation}
B(\Lambda+\nu+2\rho,\Lambda-\nu)
=
\sum_i q_i B(\Lambda+\nu+2\rho,\alpha_i).
\end{equation}
Every pairing on the right is strictly positive. Consequently,
\eqref{eq::kernel_osp_norm_equality} holds if and only if
$\Lambda=\nu$. Thus the only possible highest weights in the
kernel are
\begin{equation}
\mu_+\coloneqq\Lambda-\rho_{\bar{1}},
\qquad
\mu_-\coloneqq r_{2\beta}(\Lambda+\rho)-\rho_{\bar{0}}.
\end{equation}

It remains to show that both occur with multiplicity one.
The weight space of $\mu_+$ in $L(\Lambda)\otimes M(\odd)$ is
one-dimensional, spanned by $v_\Lambda\otimes1$, where
$v_\Lambda$ is a highest weight vector of $L(\Lambda)$ and $1$
is the oscillator vacuum. This vector is a highest weight vector,
and the generated submodule belongs to the kernel by
\eqref{eq::kernel_osp_norm_equality}.

Set $m\coloneqq\Lambda((2\beta)^\vee)$. Since
$\rho((2\beta)^\vee)=\tfrac12$, one has
$\mu_-=\mu_+-(2m+1)\beta$. Since $\beta$ is simple and all weights
of the tensor product lie in $\mu_+-Q_{+}$, only the $\beta$-string
through $v_\Lambda$ and the oscillator polynomials $\CC[x_n]$ can
contribute to the weight space of $\mu_-$. The $\beta$-string is
the finite-dimensional simple $\osp(1\vert2)$-module whose
restriction to its even $\mathfrak{sl}(2)$-subalgebra is
$F_m\oplus F_{m-1}$, where $F_s$ denotes the simple
$\mathfrak{sl}(2)$-module of highest weight $s$ and $F_{m-1}$ is
omitted when $m=0$.

The required $\mathfrak{sl}(2)$-highest weight is
$\mu_-((2\beta)^\vee)=-m-\tfrac32$. By
Lemma~\ref{lemm::rank_one_oscillator}, the highest weights occurring
in $F_{m-1}\otimes\CC[x_n]$ are
\begin{equation}
m-\tfrac32-2j
\qquad\text{and}\qquad
m-\tfrac52-2j,
\qquad 0\leq j\leq m-1,
\end{equation}
so $-m-\tfrac32$ does not occur. On the other hand, the highest
weights occurring in $F_m\otimes\CC[x_n]$ are
\begin{equation}
m-\tfrac12-2j
\qquad\text{and}\qquad
m-\tfrac32-2j,
\qquad 0\leq j\leq m.
\end{equation}
The weight $-m-\tfrac32$ occurs exactly once, namely in the second
family for $j=m$. Hence $\mu_-$ occurs with multiplicity one. Altogether, 
\[
\ker\Dirac_{\gg,\even}^{2}
\cong
L_{0}(\Lambda-\rho_{\bar{1}})
\oplus
L_{0}\bigl(r_{2\beta}(\Lambda+\rho)-\rho_{\bar{0}}\bigr).
\]
Now, $\Dirac_{\gg,\even}$ preserves $\ker \Dirac_{\gg,\even}^{2}$ and commutes with the action of $\even$. Hence it acts by a scalar on each simple constituent, and each such scalar has square zero. Therefore $\ker\Dirac_{\gg,\even}=\ker \Dirac_{\gg,\even}^{2}$.
\end{proof}

\begin{remark} \begin{enumerate}
    \item[(a)] The theorem includes $\Lambda=0$, in which case the
decomposition is precisely the oscillator decomposition
of Theorem~\ref{thm::decomp_oscillator_osp}.
\item[(b)] Let $W^{\mathrm{adm}}_{\Lambda}\subset W$ denote the subset
of Weyl group elements for which
$w(\Lambda+\rho)-\rho_{\bar{0}}$ is admissible. If $\Lambda$ is integrable, this is precisely $W^{\mathrm{adm}}_{\Lambda}=\{1,r_{2\beta}\}$. 
Hence~Theorem~\ref{thm::kernel_Dirac_affine_osp} can be rewritten as
\[
\ker\Dirac_{\gg,\even}
\cong
\bigoplus_{w\in W^{\mathrm{adm}}_{\Lambda}}
L_0\bigl(w(\Lambda+\rho)-\rho_{\bar{0}}\bigr).
\]
\end{enumerate}
\end{remark}

The preceding argument uses both the admissibility of the
$\even$-constituents and the integrability of $L(\Lambda)$.
Admissibility allows us to associate with each highest
weight $\mu$ a dominant weight
$\nu=w(\mu+\rho_{\bar0})-\rho$, where
$w\in\{1,r_{2\beta}\}$.
The oscillator decomposition and the $W$-invariance of
$\mathcal{P}_{L_{\bb}(\Lambda)}$ then imply
$\Lambda-\nu\in Q_{+}$.
By $W$-invariance of $B$, the equality of the shifted norms
is therefore equivalent to
\[
B(\Lambda+\nu+2\rho,\Lambda-\nu)=0.
\]
Since $\Lambda$ and $\nu$ are dominant and
$B(\rho,\alpha)>0$ for every simple root $\alpha$, this
forces $\nu=\Lambda$. The two possible choices of $w$
give the two highest weights occurring in the kernel.

For general basic classical Lie superalgebras, neither
this reduction to a dominant weight nor the required
strict positivity is available in general.
In particular, isotropic odd roots can give additional
equalities of shifted norms. Indeed, if
$\mu=\Lambda-\rho_{\bar1}-\gamma$, then
\[
B(\Lambda+\rho,\Lambda+\rho)
=
B(\mu+\rho_{\bar0},\mu+\rho_{\bar0})
\quad\Longleftrightarrow\quad
2B(\Lambda+\rho,\gamma)=B(\gamma,\gamma).
\]
For a positive isotropic odd root $\gamma$, this reduces
to $B(\Lambda+\rho,\gamma)=0$, the atypicality condition
associated with $\gamma$.
Such an equality does not, by itself, imply that $\mu$
occurs as the highest weight of a $\even$-submodule.
Thus, beyond the case considered here, determining
$\ker\Dirac_{\gg,\even}^{2}$ requires further information
about the restriction to $\even$ and the structure of
the tensor product with $M(\odd)$.

\subsection{A Dirac Inequality in the Affine Setting?}\label{subsec::unitaizable_case} The relative cubic Dirac operator $\Dirac_{\gg,\even}$ for a finite-dimensional symmetrizable Kac--Moody superalgebra $\gg$ is directly relevant to the study of unitarizable highest weight supermodules \cite{Schmidt_Noja}. Its action yields a Dirac inequality and hence a necessary condition for unitarity. However, this argument does not extend to infinite-dimensional symmetrizable Kac--Moody superalgebras, as we explain in this section.

In what follows, $\gg$ denotes a twisted or untwisted affine Lie superalgebra. We retain the notation of Section~\ref{subsec::Generalities_KM_Superalgebras} and Section~\ref{subsec::Oscillator_supermodule}. We begin with unitarizable supermodules.

\subsubsection{Unitarizable Supermodules of Kac--Moody Superalgebras}
Unitarizable highest weight supermodules of affine Kac--Moody superalgebras were studied by Jakobsen in \cite{Jakobsen_affine_Lie_superalgebras}. We follow his conventions.

Let $M$ be a highest weight supermodule of $\gg$ for a fixed Borel subalgebra $\bb \subset \gg$. We equip $M$ with a Hermitian form. A conjugate-linear map $\omega : \gg \to \gg$ is called an \emph{anti-involution} if $\omega^{2} = \operatorname{id}_{\gg}$ and $\omega([x,y]) = [\omega(y), \omega(x)]$ for all $x,y \in \gg$. 
It is \emph{consistent} if $\omega(\gg^{\alpha}) = \gg^{-\alpha}$ for all $\alpha \in \Delta$. 
In what follows, $\omega$ denotes a consistent conjugate-linear anti-involution. 
It extends to $\UE(\gg)$ in the obvious way.

\begin{example}[$\widehat{\mathfrak{sl}}(m\vert n)$]\label{ex::Jakobsen_involution}
Let $X=\begin{psmallmatrix}a&\beta\\ \gamma&d\end{psmallmatrix}\in\mathfrak{sl}(m\vert n)$, where $a\in \Mat(m;\CC)$, $d\in \Mat(n;\CC)$, $\beta\in \Mat(m\times n;\CC)$, and $\gamma\in \Mat(n\times m;\CC)$. Define conjugate-linear anti-involutions by $\omega_{\pm}(X)=\begin{psmallmatrix}a^{\ast}&\pm\gamma^{\ast}\\ \pm\beta^{\ast}&d^{\ast}\end{psmallmatrix}$. They extend to $\widehat{\mathfrak{sl}}(m\vert n)$ by $\omega_{\pm}(X\otimes t^{r})=\omega_{\pm}(X)\otimes t^{-r}$, $\omega_{\pm}(K)=K$, and $\omega_{\pm}(D)=D$, for all $X\in\mathfrak{sl}(m\vert n)$ and $r\in\ZZ$. Then $\omega_{\pm}([u,v])=[\omega_{\pm}(v),\omega_{\pm}(u)]$ for all homogeneous $u,v\in\widehat{\mathfrak{sl}}(m\vert n)$.
\end{example}

Any real form of $\gg$ determines a conjugate-linear anti-involution. As in the case of Kac–Moody algebras, the following proposition holds (see for example \cite{Kac_infinite, Jakobsen_affine_Lie_superalgebras}). 

\begin{proposition} \label{prop::Hermitian_form}
 Let $M$ be a highest weight $\gg$-supermodule of highest weight $\Lambda$ and highest weight vector $v_{\Lambda}$. Assume $\Lambda(h)=\overline{\Lambda(\omega(h))}$ for all $h \in \hh$. Then there exists a unique Hermitian form $\langle\cdot,\cdot\rangle$ on $M$ such that $\langle v_{\Lambda},v_{\Lambda}\rangle =1$ and
 \[
 \langle xv,w\rangle = \langle v, \omega(x)w\rangle, \qquad v,w \in M
 \]
 for all $x \in \gg$. A Hermitian form satisfying this property is called \emph{contravariant}.
\end{proposition}

We call a $\Lambda \in \hh^{\ast}$ such that $\Lambda(h)=\overline{\Lambda(\omega(h))}$ \emph{symmetric}. In what follows, we assume that $M$ is a highest weight $\gg$-supermodule with symmetric highest weight $\Lambda$. Let $\langle\cdot, \cdot\rangle$ denote the associated unique contravariant Hermitian form on $M$.

\begin{definition}
 $M$ is called $\omega$-\emph{unitarizable} if $\langle\cdot,\cdot\rangle$ is positive definite.
\end{definition}

Explicit examples of unitarizable supermodules over affine Lie superalgebras are constructed in~\cite{Jakobsen_affine_Lie_superalgebras}.

\subsubsection{Failure of the Dirac Inequality} 
Let $M$ be an $\omega$-unitarizable $\gg$-supermodule, where $\omega$ is a consistent conjugate-linear anti-involution of $\gg$, and let $\langle\cdot,\cdot\rangle_{M}$ be the corresponding positive-definite Hermitian form. We let $\even$ act on $M\otimes M(\odd)$ by
\begin{equation}
x(v\otimes P)\coloneqq xv\otimes P+v\otimes \pi_{M(\odd)}(x)P,\qquad x\in\even,\ v\in M,\ P\in M(\odd).
\end{equation}
Moreover, we equip $M\otimes M(\odd)$ with the Hermitian form
\begin{equation}
\langle v\otimes P,w\otimes Q\rangle\coloneqq\langle v,w\rangle_{M}\langle P,Q\rangle_{M(\odd)}.
\end{equation}

Since the definition of $\Dirac_{\gg,\even}$ is independent of the choice of basis, the conventions of Section~\ref{subsec::Oscillator_supermodule} give
\begin{equation} \label{eq::explicit_form_Dirac_g_even}
\Dirac_{\gg,\even}=2\sum_{a}(x_{a}\otimes \partial_{a}-\partial_{a}\otimes x_{a}) \in \widehat{\End}(M\otimes M(\odd)).
\end{equation}

A polarization $\odd=(\odd)^{-}\oplus(\odd)^{+}$ is called
\emph{$\omega$-adapted of sign $\sigma\in\{\pm1\}$} if there exist
$B$-dual homogeneous bases $(x_a)$ of $(\odd)^{-}$ and
$(\partial_a)$ of $(\odd)^{+}$, normalized by
$B(\partial_a,x_b)=\tfrac12\delta_{ab}$, such that
$\omega(x_a)=\sigma\partial_a$ and
$\omega(\partial_a)=\sigma x_a$ for all $a$. Such a basis is called
an \emph{$\omega$-adapted basis of sign $\sigma$}.

Our construction uses the triangular polarization of
Section~\ref{subsec::Oscillator_supermodule}. It follows from
\eqref{eq::explicit_form_Dirac_g_even} that
$\Dirac_{\gg,\even}$ is self-adjoint if and only if this polarization
is $\omega$-adapted of sign $-1$, and skew-adjoint if and only if it
is $\omega$-adapted of sign $+1$. The triangular polarization,
however, is never $\omega$-adapted. This problem is specific to the affine
triangular decomposition.

For simplicity, we explain the obstruction for
untwisted affine Lie superalgebras. The twisted case follows analogously. In the finite-dimensional case, for each
odd root pair $\{\alpha,-\alpha\}$ exactly one of $\alpha$ and
$-\alpha$ is positive, so only one of the corresponding root vectors
belongs to $(\odd)^{+}$. In the affine case, for every $r>0$, both
$\alpha+r\delta$ and $-\alpha+r\delta$ are positive. Hence both
$e_{\alpha}\otimes t^r$ and $f_{\alpha}\otimes t^r$ belong to
$(\odd)^{+}$. Normalize the root vectors such that
$B(e_{\alpha},f_{\alpha})=\tfrac12$. Since $B$ is skew-symmetric on
$\odd$, the corresponding $B$-dual pairs are
\begin{equation}
\partial_1=e_{\alpha}\otimes t^r,\qquad
x_1=f_{\alpha}\otimes t^{-r},
\qquad
\partial_2=f_{\alpha}\otimes t^r,\qquad
x_2=-e_{\alpha}\otimes t^{-r}.
\end{equation}
After a compatible normalization, write
$\omega(e_{\alpha})=\tau f_{\alpha}$ and
$\omega(f_{\alpha})=\tau e_{\alpha}$ with $\tau\in\{\pm1\}$. Then
\begin{equation}
\omega(x_1)=\tau\partial_1,
\qquad
\omega(x_2)=-\tau\partial_2.
\end{equation}
Thus the two affine root pairs require opposite signs, and no single
$\sigma\in\{\pm1\}$ makes the triangular polarization
$\omega$-adapted for the loop anti-involutions considered above.
We conclude the following.

\begin{proposition}
The triangular polarization of
Section~\ref{subsec::Oscillator_supermodule} is not $\omega$-adapted
in the affine setting. Consequently, the usual argument deriving
a Dirac inequality from $\omega$-unitarizability and
self-adjointness or skew-adjointness of
$\Dirac_{\gg,\even}$ does not apply with this polarization.
\end{proposition}

This does not exclude self-adjointness or skew-adjointness on
particular supermodules, for instance when
$\Dirac_{\gg,\even}$ acts trivially.

To study $\omega$-unitarizable supermodules of $\gg$, one can replace the triangular polarization by an $\omega$-adapted one. For
$\gg=\widehat{\sl}(m\vert n)$, there is a natural block polarization. Let $\sl(m\vert n)=\sl(m\vert n)_{-1}\oplus \sl(m\vert n)_{0}\oplus \sl(m\vert n)_{+1}$ be the canonical $\ZZ$-grading of $\sl(m\vert n)$. We then take
\begin{equation}\label{eq::block_polarization_affine_sl}
(\odd)^{-}
=
\mathfrak{sl}(m\vert n)_{-1}\otimes\CC[t,t^{-1}],
\qquad
(\odd)^{+}
=
\slmn_{+1}\otimes\CC[t,t^{-1}].
\end{equation}
For $r\in\ZZ$, set
\begin{equation}
x_{ia,r}\coloneqq E_{m+a,i}\otimes t^r,
\qquad
\partial_{ia,r}\coloneqq E_{i,m+a}\otimes t^{-r}.
\end{equation}
Then
\begin{equation}
B(\partial_{ia,r},x_{jb,s})
=
\tfrac12\delta_{ij}\delta_{ab}\delta_{rs}.
\end{equation}
For the anti-involutions $\omega_{\pm}$ introduced above, extended to
the affine algebra by
$\omega_{\pm}(X\otimes t^r)=\omega_{\pm}(X)\otimes t^{-r}$, one has
\begin{equation}
\omega_{\pm}(x_{ia,r})
=
\pm\partial_{ia,r},
\qquad
\omega_{\pm}(\partial_{ia,r})
=
\pm x_{ia,r}.
\end{equation}
Thus the block polarization
\eqref{eq::block_polarization_affine_sl} is $\omega_{+}$-adapted of
sign $+1$ and $\omega_{-}$-adapted of sign $-1$.

However, for an $\omega$-adapted polarization, the explicit expression
\eqref{eq::explicit_form_Dirac_g_even} need no longer define an element of
the completion $\widehat{\End}_{\CC}(M\otimes M(\odd))$ used above. Indeed,
that completion is adapted to the triangular polarization and to the
corresponding degree filtration, which ensures that the infinite sums
defining $\Dirac_{\gg,\even}$ act locally finitely on each homogeneous
component. An $\omega$-adapted polarization is in general incompatible
with this filtration, so the same argument no longer applies.

One could instead introduce a completion adapted to the new polarization.
This, however, requires a different topological framework: one must first
show that the infinite series defining $\Dirac_{\gg,\even}$ converges in
the new completion, that it acts continuously on $M\otimes M(\odd)$, and
that products such as $\Dirac_{\gg,\even}^{2}$ are well-defined there.
Without these additional results, the methods developed above cannot be
applied directly to an $\omega$-adapted polarization.

\thispagestyle{empty}